\documentclass[11pt]{amsart}

\usepackage{graphicx}
\usepackage{amssymb, comment, color}
\usepackage[toc,page]{appendix}
\usepackage{epsf,overpic,amssymb,amscd,times,euscript}
\usepackage[nocompress]{cite}

\usepackage[utf8]{inputenc}
\usepackage[english]{babel}
\usepackage{amsmath}
\usepackage{amsthm}
\usepackage{amsfonts}
\usepackage{amssymb}
\usepackage{amscd}
\usepackage{bezier}
\usepackage{enumerate}

\usepackage{graphicx}
\usepackage{upgreek}
\usepackage{tikz-cd}

\usepackage{upquote}

\usepackage{parskip}

\usepackage{bbm}

\usepackage{tikz-qtree}
\usepackage[colorlinks=true,urlcolor=blue,bookmarks=true,bookmarksopen=true,citecolor=blue]{hyperref}

\usepackage [english]{babel}
\usepackage [autostyle, english = american]{csquotes}
\MakeOuterQuote{"}

\newcommand{\id}{id}

\newcommand{\N}{\mathbbm{N}}                     
\newcommand{\Z}{\mathbbm{Z}}                     
\newcommand{\R}{\mathbbm{R}}                     
\newcommand{\C}{\mathbbm{C}}                     
\newcommand{\J}{\mathcal{J}}
\newcommand{\D}{\mathbb{D}}

\renewcommand{\P}{\mathcal{P}}

\renewcommand{\Im}{\mathrm{Im}}

\newcommand{\dist}{\mathrm{dist\,}}             
\newcommand{\loc}{\mathrm{loc}} 

\newcommand{\CH}{\mathrm{CH}} 

\newcommand{\wtil}{\widetilde{w}}

\newcommand{\ot}{\mathbf{t}}

\newcommand{\topo}{\mathrm{top}}

\newcommand{\os}{\mathbf{s}}

\newtheorem{assump}{Assumption}           

\newtheorem{mainthm}{\sc Theorem}           
\newtheorem{maincor}{\sc Corollary}
\newtheorem{thm}{Theorem}[section]               
\newtheorem*{thm*}{Theorem}               
\newtheorem{cor}[thm]{Corollary}        
\newtheorem*{cor*}{Corollary}        
\newtheorem{lem}[thm]{Lemma}  
\newtheorem*{lem*}{Lemma}
\newtheorem*{assump*}{Standing assumption}
\newtheorem{prop}[thm]{Proposition}     
\newtheorem{conj}[thm]{Conjecture}      
\theoremstyle{definition}
\newtheorem{defn}[thm]{Definition}      
\newtheorem{rec}[thm]{Recollection}      
\newtheorem{rem}[thm]{Remark}           

 \newtheorem*{acknowledgement*}{\protect\acknowledgementname}
\newcounter{claim}
\newenvironment{claim}[1][]{\refstepcounter{claim}\par\noindent\underline{Claim~\theclaim:}\space#1}{}
\newenvironment{claimproof}[1]{\par\noindent\underline{Proof:}\space#1}{\qed}

 \providecommand{\acknowledgementname}{Acknowledgement}

\author[M.R.R. Alves]{Marcelo R.R. Alves}
\thanks{M.R.R. Alves was supported by the Senior Postdoctoral fellowship of the Research Foundation - Flanders (FWO) in fundamental research 1286921N}
\address{Marcelo R.R. Alves, Faculty of Science,\\
University of Antwerp,
 Campus Middelheim,
 Middelheimlaan 1,
BE-2020 Antwerpen,
Belgium.}
\email{\texttt{marcelorralves@gmail.com}}

\author[L. Dahinden]{Lucas Dahinden}
\thanks{L. Dahinden was funded by the Deutsche Forschungsgemeinschaft (DFG, German Research Foundation) – Project-ID 281071066 - TRR 191.}
\address{Lucas Dahinden, Fakult\"at f\"ur Mathematik \\
Ruhr-Universit\"at Bochum, Geb\"aude IB, Etage 3, Raum 67
D-44780 Bochum, Germany.}
\email{\texttt{Lucas.Dahinden@ruhr-uni-bochum.de}}

\author[M. Meiwes ]{Matthias Meiwes}
\thanks{M. Meiwes was supported by the ERC Starting Grant 757585 and by the ISF Grant 938/22}
\address{Matthias Meiwes,
School of Mathematical Sciences, Tel Aviv University, Ramat Aviv, Tel Aviv 69978, Israel.}
\email{\texttt{matthias.meiwes@live.de}}

\author[A. Pirnapasov]{Abror Pirnapasov}
\thanks{A. Pirnapasov was supported by the DFG SFB/TRR 191 “Symplectic Structures in Geometry, Algebra
and Dynamics”, Projektnummer 281071066-TRR 191, and the ANR CoSyDy “Conformally
Symplectic Dynamics beyond symplectic dynamics” (ANR-CE40-0014)}
\address{Abror Pirnapasov, Department of Mathematics\\
University of Maryland \linebreak 
 William E. Kirwan Hall 4176 Campus Dr.  College Park, MD 20742. USA.}
\email{\texttt{abrorpirnapasov@gmail.com }}

\title{${C}^0$-stability of topological entropy for {R}eeb flows in dimension $3$}
\begin{document}
\begin{abstract}
We study stability properties of the topological entropy of Reeb flows on contact $3$-manifolds with respect to the $C^0$-distance on the space of  contact forms. Our main results show that a $C^\infty$-generic contact form on a closed co-oriented contact $3$-manifold $(Y,\xi)$ is a lower semi-continuity point for the topological entropy, seen as a functional on the space of contact forms of $(Y,\xi)$ endowed with the $C^0$-distance. 

We also study the stability of the topological entropy of geodesic flows of Riemannian metrics on closed surfaces. In this setting, we show that a non-degenerate Riemannian metric on a closed surface $S$ is a lower semi-continuity point of the topological entropy, seen as a functional on the space of Riemannian metrics on $S$ endowed with the $C^0$-distance.
\end{abstract}

\maketitle
\tableofcontents

\section{Introduction}

The objective of this article is to study stability properties of the topological entropy of Reeb flows on contact $3$-manifolds with respect to $C^0$-small perturbations of the contact form. Our main Theorem~\ref{thm:main} below states that the topological entropy, seen as a function on the space of contact forms endowed with the $C^0$-topology, is lower semi-continuous on a $C^\infty$-open and dense set. A similar argument  implies the lower semi-continuity of the topological entropy of geodesic flows of Riemannian metrics on surfaces with respect to the $C^0$-distance on the space of non-degenerate Riemannian metrics; this is the content of Theorem \ref{thm:geod}. These results can be seen as stability results for the topological entropy: they say that, in low dimensions, sufficiently $C^0$-small perturbations of Riemannian metrics and contact forms cannot collapse the  
 positive topological entropy of the system under consideration to zero.
In order to state precisely our results, we introduce some preliminary notions. 
\subsection{Preliminary notions} 
\subsubsection*{The space of contact forms} 
Let $(Y,\xi)$ be a closed co-oriented contact $3$-manifold. The 1-form $\alpha$ is a contact form on $(Y,\xi)$ if $\xi=\ker \alpha$ (the condition that $\alpha\wedge d\alpha$ is a volume form follows from $\xi$ being a contact plane field). Any other contact form on $(Y,\xi)$ with the same co-orientation is of the form~$f\alpha$, where $f$ is a positive function. We denote by $\mathcal{R}(Y,\xi)$ the space of contact forms on $(Y,\xi)$. The equations 
\begin{align*}
 \alpha (R_\alpha) &=1, \\   d\alpha (R_\alpha,\cdot)&=0,    
\end{align*} 
define the Reeb vector field $R_\alpha$. From $R_\alpha$ we can recover $\alpha$ by setting \mbox{$\alpha(R_\alpha)=1$}, $\alpha|_\xi=0$. Thus, $\mathcal{R}(Y,\xi)$ can also be thought of as the space of Reeb vector fields. Conversely, we can think of the dynamical properties of the Reeb flow $\phi_{\alpha}$ as properties of the contact form $\alpha$.

\subsubsection*{Dynamical properties} 
A periodic orbit $\gamma$ of the Reeb flow of $\alpha$ is called a Reeb orbit of $\alpha$. 
Its action $A(\gamma):=\int_{\gamma} \alpha$ coincides with its period $T$ since $\alpha(\dot \gamma)=\alpha(R_\alpha)=1.$ The period can be used to filter the Reeb orbits: we are interested in the number of Reeb orbits of length at most $T$. The asymptotic growth rate of this number is linked to topological entropy.

\subsubsection*{The $C^0$-distance for contact forms} 
Given two contact forms $\alpha_1$ and $\alpha_2$ on $(Y,\xi)$ let $f^{\alpha_2}_{\alpha_1}:Y \to (0,+\infty)$ be the function such that $f^{\alpha_2}_{\alpha_1} \alpha_1 = \alpha_2$. We define the $C^0$-distance $d_{C^0}$ between $\alpha_1$ and $\alpha_2$ by
\begin{equation}
    d_{C^0}(\alpha_1,\alpha_2) := \max_{x \in Y} \{ |\log(f^{\alpha_2}_{\alpha_1}(x)) | \}.
\end{equation}
The distance $d_{C^0}$ is natural from the point of view of symplectic topology, and it can be thought of as the ``symplectic distance'' between contact forms (see~\cite{SZ,Usher} and references therein). It can also be thought of as an analogue for Reeb flows of the Hofer metric on the space of Hamiltonian diffeomorphisms. 

However, the dynamical properties of Reeb flows can undergo dramatic changes under arbitrarily small perturbations with respect to $d_{C^0}$. The reason for this is that the Reeb vector field $R_\alpha$ depends on $d\alpha$, and therefore $C^0$-small perturbations of $\alpha$ can lead to drastic changes of $R_\alpha$. We illustrate this in the next paragraphs by a $C^0$-small perturbation of a geodesic flow in Figure~\ref{fig:ElevatedDisk}.

\subsubsection{Geodesic flows as Reeb flows}

One of the most important examples of closed contact $3$-manifolds are the unit cotangent bundles of closed surfaces endowed with the so-called geodesic contact structures. More precisely, given a closed surface $S$ there exists a contact structure $\xi_{\rm geo}$ on the unit cotangent bundle $T^1 S$ of $S$, with the property that  given a Riemannian metric $g$ on $S$, there is a unique contact form $\alpha_g$ on $(T^1 S, \xi_{\rm geo})$ such that the geodesic flow $\phi_g$ of $g$ restricted to the unit tangent bundle $(T_1)_gS=\{v \in TS\, |\, |v|_g=1\}$ and the Reeb flow $\phi_{\alpha_g}$  are the same  under the natural identification between $(T_1)_gS$ and $T^1S$.  A comprehensible exposition of this fact can be found in Section 1.5 of \cite{GeigesBook}.

We will denote by $\mathrm{Met}(S)$ the space of $C^{\infty}$-smooth Riemannian metrics on $S$. It is easy to see that the map $\upphi:\mathrm{Met}(S) \to \mathcal{R}(T^1 S,\xi_{\rm geo})  $ which takes $g$ to $\alpha_g$ is injective.

\subsubsection{The $C^0$-distance for Riemannian metrics}\label{subsec:c0distRiemann}
We consider the following distance on the space of Riemannian metrics of a closed orientable surface $S$. If $g_1$ and $g_2$ are Riemannian metrics on $S$, we denote by $| \cdot |_{g_1}$ and $| \cdot |_{g_2}$ the norms induced by $g_1$ and $g_2$ on the fibers of $TS$. We then define:
\begin{equation}
    \overline{d}_{C^0}(g_1,g_2):= \inf \{ \epsilon > 0 \ \Big\rvert \ e^{-\epsilon}|v|_{g_2} \leq |v|_{g_1} \leq e^{\epsilon}|v|_{g_2} \mbox{ for all } v \in TS  \}.
\end{equation}
From a purely metric point of view, $\overline{d}_{C^0}$ is an interesting distance to consider in $\mathrm{Met}(S)$, since geometric quantities, such as the volume of subsets of $(S,g)$, the diameter of $(S,g)$, and the Riemannian distance function $d_g$ on $S$, are all continuous with respect to $\overline{d}_{C^0}$. In studying these quantities in $\mathrm{Met}(S)$, it is therefore more natural to consider the topology induced by $\overline{d}_{C^0}$ than the stronger $C^k$-topologies, for $k\geq 1$. 

It is elementary to show that the distances $\overline{d}_{C^0}$ and ${d}_{C^0}$ satisfy
\begin{equation}
    \overline{d}_{C^0}(g_1,g_2) = {d}_{C^0}(\alpha_{g_1},\alpha_{g_2}),
\end{equation}
for any Riemannian metrics $g_1,g_2\in\mathrm{Met}(S)$.
In other words, the map \linebreak  $\upphi: (\mathrm{Met}(S),\overline{d}_{C^0}) \to (\mathcal{R}(T^1 S,\xi_{\rm geo}),{d}_{C^0}) $ is an isometric embedding, i.e. \linebreak $\upphi^*({d}_{C^0}) = \overline{d}_{C^0}$. 
This allows us to obtain information about $\overline{d}_{C^0}$  from information about ${d}_{C^0}$.

\subsubsection{The topological entropy.} The dynamical invariant studied in this article is the topological entropy. Recall that the topological entropy $h_{top}$ is a non-negative number that one associates to a dynamical system and which measures the complexity of the dynamics. Positivity of the topological entropy for a dynamical system implies some type of exponential instability. 

{The following definition of $h_{\rm top}$ for flows on compact manifolds is due to Bowen. Let $M$ be a closed manifold, $X$ be a $C^\infty$-smooth vector field on $M$ and $\phi$ be the flow of $X$. We consider an auxiliary distance function $d$ on $M$. Given positive numbers $T,\delta$ we say that a subset $S \subset M$  is  $T,\delta$-separated for $\phi$ if, for all points $p,q \in S$ with $p\neq q$, we have
$$\max_{t \in [0,T]}\{ d(\phi^t(p) ,\phi^t(q)) \} > \delta.  $$ 
We let $n^{\delta}_X(T)$ be the maximal cardinality of a $T,\delta$-separated set for the flow $\phi$ of $X$. The $\delta$-entropy $h_\delta$ is then defined by 
$$h_\delta(\phi):= \limsup_{T\to +\infty} \frac{\log(n^{\delta}_X(T))}{T},$$
and the topological entropy $h_{\rm top}$ is defined by
\begin{equation}
    h_{\rm top}(\phi):= \lim_{\delta \to 0} h_{\delta}(\phi).
\end{equation}
Geometrically, we see that $h_\delta$ measures the exponential growth rate of the number of orbits which are distinguishable with precision $\delta$ as time advances: $h_{\topo}$ is then the limit of these growth rates as $\delta$ goes to $0$. The topological entropy of a $C^1$-smooth fixed-point free flow on a closed manifold is always finite: this implies the finiteness of $h_{\rm top}$ for smooth Reeb flows on closed contact manifolds.
We refer the reader to \cite{Hasselblatt-Katok} for the basic properties of $h_{\topo}$.}

\subsection{Main results} 
In the present paper we investigate stability properties of the topological entropy with respect to the distance $d_{C^0}$ on $\mathcal{R}(Y,\xi)$. For simplicity, we introduce the following terminology:

\begin{defn}
Consider $h_{\rm top}: \mathcal{R}(Y,\xi) \to [0,+\infty) $ as a functional on $\mathcal{R}(Y,\xi)$.
Given a subset $\mathcal{A}$ of $\mathcal{R}(Y,\xi)$ we say that $h_{\rm top}$ is \textbf{lower semi-continuous on \linebreak $\mathcal{A} \subset \mathcal{R}(Y,\xi)$ with respect to $d_{C^0}$}, if every point of $\mathcal{A}$ is a lower semi-continuity point of $h_{\rm top}: \mathcal{R}(Y,\xi) \to [0,+\infty) $ for the topology induced by $d_{C^0}$ on $\mathcal{R}(Y,\xi)$.
The analogous notion for $\mathrm{Met}(S)$ and $\overline{d}_{C^0}$ will also be used in this paper.
\end{defn}

Our main result is the following:
\begin{mainthm} \label{thm:main}
 Let $(Y,\xi)$ be a closed co-oriented contact $3$-manifold. 
 The entropy functional 
 \begin{align*}
 h_{\rm top}:   \mathcal{R}(Y,\xi) &\to [0,+\infty), \\
 \alpha &\mapsto h_{\rm top}(\phi_\alpha), 
 \end{align*}
 is lower semi-continuous with respect to $d_{C^0}$ on a $C^\infty$-open and dense set of $\mathcal{R}(Y,\xi)$.
\end{mainthm}

{As explained above $(\mathrm{Met}(S),\overline{d}_{C^0})$ can be thought of as isometrically embedded in $ (\mathcal{R}(T^1 S,\xi_{\rm geo}),{d}_{C^0}) $. This allows us to apply our techniques to $(\mathrm{Met}(S),\overline{d}_{C^0})$. For this we recall terminology: A Reeb flow on a closed $3$-manifold is said to be non-degenerate if all its periodic orbits are non-degenerate; see Section \ref{sec:Reeb_orbits} for the precise definitions. We say that a Riemannian metric $g$ on a closed surface $S$ is \textit{non-degenerate}\footnote{ Non-degenerate Riemannian metrics are sometimes called \textit{bumpy metrics} in the literature.}, if the contact form $\alpha_g$ associated to $g$ is non-degenerate; i.e. if the geodesic flow of $g$ is a non-degenerate Reeb flow on $(T^1 S, \xi_{\rm geo})$. We denote by $\mathrm{Met}_{\rm nd}(S)$ the set of non-degenerate Riemannian metrics on $S$. It is well-known that $\mathrm{Met}_{\rm nd}(S)$ is a $C^\infty$-dense subset of   $\mathrm{Met}(S)$. In the setting of geodesic flows, we prove the following variation of Theorem \ref{thm:main}. 
\begin{mainthm} \label{thm:geod}
 Let $S$ be a closed orientable surface, and denote by $\mathrm{Met}_{\rm nd}(S)$ the set of non-degenerate Riemannian metrics on $S$.
 The entropy functional 
 \begin{align*}
  h_{\rm top}:   \mathrm{Met}(S) &\to [0,+\infty), \\
  g &\mapsto h_{\rm top}(\phi_g), 
 \end{align*}
is lower semi-continuous with respect to $\overline{d}_{C^0}$ on the $C^\infty$-dense subset $ \mathrm{Met}_{\rm nd}(S)$ of $\mathrm{Met}(S)$. 
\end{mainthm}
}

\subsubsection*{Significance of the results} We proceed to explain the significance of our results. Theorem \ref{thm:main} and Theorem \ref{thm:geod} show that the topological entropy cannot be destroyed by small perturbations with respect to $d_{C^0}$ and $\overline{d}_{C^0}$. 

It is a priori not clear at all whether any dynamical property of geodesic flows should be stable with respect to the distance $\overline{d}_{C^0}$ on $\mathrm{Met}(S)$. The direction of the geodesic vector field of a Riemannian metric $g$ depends on the first derivatives of $g$, so very small perturbations with respect to $\overline{d}_{C^0}$ can lead to drastic changes of the geodesic vector field. For example,  Figure~\ref{fig:ElevatedDisk} shows a $\overline{d}_{C^0}$-small perturbation of the Euclidean metric that makes the unit circle a geodesic. Similarly, for any surface $(\Sigma, g)$ and any simple closed regular curve $\gamma$ one can find a $\overline{d}_{C^0}$-small perturbation of $g$ for which $\gamma$ is a geodesic. For more details on the perturbation method, see Example 43 in~\cite{ADMM}. We also remark that one can arbitrarily increase the topological entropy of Reeb flows and geodesic flows with arbitrarily small perturbations in $d_{C^0}$ and $\overline{d}_{C^0}$, respectively. See Theorem 12 in \cite{ADMM}. 

\begin{figure}
    \centering
    \includegraphics[width=0.5\linewidth]{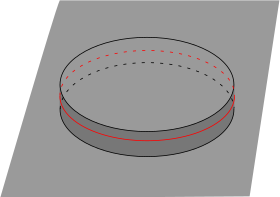}
    \caption{The unit disk in Euclidean space can be elevated (and the corners smoothed) by a $C^0$-small perturbation of the metric. This makes the unit circle a geodesic.}
    \label{fig:ElevatedDisk}
\end{figure}

Many dynamical quantities, such as the growth rate of periodic orbits, the Liouville entropy and others, are \textbf{not} lower semi-continuous with respect to $\overline{d}_{C^0}$. Our results show that $h_{\rm top}$ of geodesic flows of Riemannian surfaces and Reeb flows on $3$-manifolds is surprisingly robust.

One interesting application of Theorem \ref{thm:geod} is the following. In \cite{Mane}, Ma\~n\'e established the following remarkable formula for the topological entropy of the geodesic flow of a $C^\infty$-smooth Riemannian metric $g$ on a manifold $Q$:
\begin{equation} \label{eq:Mane}
h_{\rm top}(\phi_g) = \lim_{T \to +\infty } \frac{1}{T} \log \Big( \int_{Q \times Q} (\mathcal{N}^g_T(p,q)) d\omega_g(p) d\omega_g(q)\Big) .  
\end{equation}
Here, $\mathcal{N}^g_T(p,q)$ denotes the number of geodesic chords of $g$ from $p$ to $q$ with \linebreak length $<T$, $d\omega_g(p)$ means integration in the variable $p$ with respect to the Riemannian volume form $\omega_g$ on $Q$ associated to $g$, and $d\omega_g(q)$ means integration in the variable $q$ with respect to $\omega_g$. This formula gives a characterisation of $h_{\rm top}(\phi_g)$ in terms of the purely geometric quantity which appears on the right side of \eqref{eq:Mane}.

The right hand side of \eqref{eq:Mane} is the exponential growth of the average number of geodesics connecting two points of $Q$. Using Ma\~n\'e's formula, our results provide surprising robustness of this exponential growth in case $Q$ is a closed surface. Theorem \ref{thm:geod} implies that for a non-degenerate Riemannian metric $g$ on a closed surface this exponential growth is positive and stable under $C^0$-small perturbations of $g$. This is far from obvious, given that the counting function $\mathcal{N}^g_T(p,q)$ can undergo drastic changes under arbitrarily $C^0$-small perturbations of $g$.

\subsubsection{Further situations where $h_{\topo}$ is lower semi-continuous with respect to $d_{C^0}$}
Another situation in which our methods apply is that of right-handed Reeb flows.
%
Right-handed vector fields, introduced by Ghys in \cite{Ghys}, are a special class of
non-singular vector fields on homology 3-spheres, and their flows are called right-handed. Their defining property is that all their trajectories link positively (see~\cite{Ghys,FH-righthanded}). Ghys showed in \cite{Ghys} that every periodic orbit of a right-handed flow bounds a global surface of section, and that the space of right-handed vector fields is $C^1$-open. In \cite{FH-righthanded}, Florio and Hryniewicz established the existence of many interesting examples of right-handed Reeb flows and obtained quantitative conditions for certain Reeb flows on $(S^3,\xi_{\rm tight})$ to be right-handed. For example, they showed that if $\updelta \geq 0.7225 $, then the geodesic flow of any $\updelta$-pinched Riemannian metric on $S^2$ lifts to a right-handed Reeb vector field on $(S^3,\xi_{\rm tight})$. We say that a contact form on $(S^3,\xi_{\rm tight})$ is right-handed if its Reeb flow is right-handed. Applying our techniques we obtain 
\begin{maincor} \label{cor:righthanded}
 Let $(S^3,\xi_{\rm tight})$ be the three-sphere endowed with its unique tight contact structure. 
 The entropy functional 
 \begin{align*}
 h_{\rm top}:   \mathcal{R}(S^3,\xi_{\rm tight}) &\to [0,+\infty), \\
 \alpha &\mapsto h_{\rm top}(\phi_\alpha), 
 \end{align*}
 is lower semi-continuous with respect to $d_{C^0}$ on the set of right-handed contact forms $\alpha\in\mathcal{R}(S^3,\xi_{\rm tight})$.
\end{maincor}
Note that in this statement, no assumption on the non-degenericity of any Reeb orbit is made.

\color{black}
\subsubsection{Robustness of the topological entropy for  geodesic flows of the $2$-torus.}
Theorem \ref{thm:geod} generalizes some of the results in~\cite{ADMM}  on robustness of topological entropy: topological entropy is \emph{robust} at a metric $g$ if a $C^0$-small perturbation of the metric cannot collapse $h_\topo(\phi_g)$ to 0. More precisely, entropy is robust at $g$ if there exist $\varepsilon, c>0$ such that for all $\tilde g$ with $\overline{d}_{C^0}(g,\tilde g)<\varepsilon$ we have $h_\topo(\phi_{\tilde g})\geq c$. It is well-known that for surfaces of genus at least two, $h_{\topo}$ is uniformly bounded from below for metrics with the same total area. Hence the notion of robustness is only interesting for surfaces that admit metrics of zero entropy, which is the case for the sphere and the torus. 
In \cite{ADMM}, Merlin and the first three authors studied some robustness phenomena of the topological entropy of geodesic flows on Riemannian manifolds.  
In particular, in Theorem 10 of~\cite{ADMM} it was established that a $C^\infty$-generic metric on the torus has robust topological entropy. 

\emph{Lower semi-continuity and robustness.} A point of lower semi-continuity of $h_\topo$ (with $h_\topo>0$) is a point of robustness. More precisely, if $h_\topo$ is robust at $g$ and if one lowers $\varepsilon$ in the definition, it may be possible to raise $c$. A robust point is lower semi-continuous if $c(\varepsilon)\to h_\topo(\phi_g)$ for $\varepsilon \to 0$. The lower bounds produced in~\cite{ADMM} are far from sharp, thus lower semi-continuity could not be established, which means that  also in the case of the torus, Theorem \ref{thm:geod} amounts to a significant improvement of the results in \cite{ADMM}.    

\emph{Robustness for all metrics on the torus.}
On the other hand, the results in the present paper combined with the methods in \cite{ADMM} allow us to extend the generic robustness result for the torus to all Riemannian metrics. More precisely, we obtain the following corollary (for a proof see Section \ref{sec:geodesicflows}).  
\begin{maincor}\label{cor:robust_T2}
The topological entropy $h_{\topo}$ is robust at all $g\in \mathrm{Met}(T^2)$ for which $h_{\topo}(\phi_g)>0$.
\end{maincor}

\subsubsection{Barcode entropy.} The recent work of \c{C}ineli, Ginzburg, and G\"urel, \cite{CGG}, introduced a new way of relating the topological entropy of Hamiltonian diffeomorphisms to Floer theory. This approach was adapted to study the topological entropy of geodesic flows in \cite{GGM} and of Reeb flows in \cite{FLS1}.  In \cite{GGM} Ginzburg-G\"{u}rel-Mazzucchelli showed that for the geodesic flow $\phi_{g}$ of a Riemannian metric $g$ on a closed surface, $h_{\topo}(\phi_g)$ coincides with the barcode entropy $\hbar(\phi_g)$ of $\phi_g$ which they introduce. The barcode entropy is a measure of the complexity of the Morse complex of the energy functional of $g$ on the free loop space of $S$. Combining Theorem~\ref{thm:geod} with Corollary C of \cite{GGM} we obtain the following
\begin{maincor}
Let $S$ be a closed orientable surface, and denote by $\mathrm{Met}_{\rm nd}(S)$ the set of non-degenerate Riemannian metrics on $S$. The barcode entropy functional
\begin{align*}
  \hbar:   \mathrm{Met}(S) &\to [0,+\infty), \\
  g &\mapsto \hbar(\phi_g), 
 \end{align*}
is lower semi-continuous with respect to $\overline{d}_{C^0}$ on the $C^\infty$-dense subset $ \mathrm{Met}_{\rm nd}(S)$ of $\mathrm{Met}(S)$.
\end{maincor}

\subsection{Strategy of the proof} \label{sec:strategy}

Our strategy for proving Theorem \ref{thm:main} is inspired by the work \cite{AlvesMeiwesBraids}, where the first and third authors showed that the topological entropy of Hamiltonian diffeomorphisms of surfaces is lower semi-continuous with respect to the Hofer metric $d_{\rm Hofer}$ on the space of  Hamiltonian diffeomorphisms. The distances $d_{C^0}$ and $\overline{d}_{C^0}$ can be thought of as analogous of the Hofer metric on the spaces $\mathcal{R}(Y,\xi)$  and ${\rm Met}(S)$, respectively. 

Crucial to our results is the work \cite{Meiwes_horseshoes}, in which the third author shows that the $h_{\rm top}$ of a Reeb flow  on a contact $3$-manifold $(Y,\xi)$ can be recovered from the contact topology of its periodic orbits. More precisely, given a contact form $\alpha_0$ and $\epsilon>0$, he shows that there exists a link $\mathcal{L}^{\alpha_0}_{\epsilon}$ of periodic orbits of  $\phi^t_{\alpha_0}$ such that the number $\# \Omega^T_{\alpha_0}(\mathcal{L}^{\alpha_0}_{\epsilon})$ of free homotopy classes of loops in $Y\setminus \mathcal{L}^{\alpha_0}_{\epsilon}$ that contain a unique periodic orbit of $\phi^t_{\alpha_0}$ with action $\leq T$ satisfies $$\limsup_{T \to +\infty} \frac{\log(\# \Omega^T_{\alpha_0}(\mathcal{L}^{\alpha_0}_{\epsilon})) }{T}> h_{\rm top}(\phi_{\alpha_0})-\epsilon. $$
See Theorem \ref{thm:CH_recover} for the precise statement.

Let now $\alpha$ be a contact form on $(Y,\xi)$ which is $\delta$-close to $\alpha_0$ with respect to $d_{C^0}$. By the work \cite{AlvesPirnapasov} of the first and fourth authors, we have that for any link $\mathcal{L}$ of Reeb orbits of $\alpha$, 
$$h_{\rm top}(\phi_{\alpha})\geq \limsup_{T \to +\infty} \frac{\log(N^T_{\alpha}(\mathcal{L}(\alpha))) }{T},$$ where $N^T_{\alpha}(\mathcal{L}(\alpha))$ is the number of free homotopy classes of loops in $Y\setminus \mathcal{L}$ that contain periodic orbits of $\phi^t_{\alpha}$ with action $\leq T$. 
Our strategy now consists of finding, for sufficiently small $\delta >0$ and any non-degenerate $\alpha$ with $d_{C^0}(\alpha,\alpha_0)<\delta$, a link $\mathcal{L}_\epsilon(\alpha)$ of Reeb orbits of $\alpha$ such that 
$$\limsup_{T \to +\infty} \frac{\log(N^T_{\alpha}({\mathcal{L}_\epsilon(\alpha)))}}{T} > \limsup_{T \to +\infty} \frac{\log(\# \Omega^T_{\alpha_0}(\mathcal{L}^{\alpha_0}_{\epsilon})) }{T} - \epsilon > h_{\rm top}(\phi_{\alpha_0}) - 2\epsilon.$$
The link $\mathcal{L}_\epsilon(\alpha)$ for $\alpha$ is obtained from the link $\mathcal{L}_\epsilon^{\alpha_0}$ given by Theorem \ref{thm:CH_recover} with the help of holomorphic  curves; this is Theorem \ref{thm:stability}, our main technical result. We prove it in Section \ref{sec:proofmainresult}. Altogether this gives the inequality $$h_{\topo}(\phi_{\alpha}) \geq h_{\topo}(\phi_{\alpha_0})-2\epsilon,$$ which, considered as a property of $\alpha$, is closed in the $C^{\infty}$-topology (see~\cite{N89}) and hence extends to degenerate $\alpha$.

Let us explain how to obtain the link $\mathcal{L}_\epsilon(\alpha)$ when $\alpha_0$ is non-degenerate and the Reeb flow $\phi_{\alpha_0}$ admits a global surface of section with binding link $\mathcal{L}_b$, which is a $C^{\infty}$-generic condition (see 
 \cite{Hryniewicz_generic, Contreras_generic}); the argument will extend to a $C^{\infty}$-open and dense set of $\alpha_0$ without much difficulty.  To obtain the link $\mathcal{L}_\epsilon(\alpha)$ in that case, we let $\mathcal{L}_0 = \mathcal{L}_b \cup \mathcal{L}^{\alpha_0}_{\epsilon}$, and consider the trivial cylinders $\R \times \mathcal{L}_0$ over the link $\mathcal{L}_0$ in the symplectization of $\alpha_0$ endowed with a cylindrical almost complex structure. We consider a one-dimensional homotopy of exact symplectic cobordisms starting at the symplectization of $\alpha_0$ and stretching the neck along the contact form $\alpha$. The cylinders $\R \times \mathcal{L}_0$ are holomorphic for the symplectization of $\alpha_0$. We then consider the $1$-dimensional moduli space generated by the holomorphic cylinders $\R \times \mathcal{L}_0$.

The $C^0$-proximity of $\alpha$ to $\alpha_0$ guarantees that the moduli spaces cannot "stop" along the homotopy, and must continue for the whole stretching the neck procedure. The end result is that we can associate a holomorphic building to each component of $\mathcal{L}_0$. These homolorphic buildings must detect Reeb orbits of $\alpha$, and we form the link $\mathcal{L}_\epsilon(\alpha)$ by carefully selecting orbits detected by these buildings. A delicate analysis using Siefring's asymptotic formulas for ends of holomorphic curves shows that the link $\mathcal{L}_\epsilon(\alpha)$ has the desired properties; see Section \ref{sec:non-collaps}.

\subsection{Context}

The relationship between contact and symplectic topology and $h_{\rm top}$ of Reeb flows and symplectomorphisms has been studied extensively and fruitfully in recent years by various methods. For example, there
is an abundance of contact manifolds for which the topological entropy or the exponential orbit growth rate is positive for all Reeb flows. Examples and dynamical properties of those manifolds are investigated in \cite{AASS,Alves-Cylindrical,Alves-Anosov,A2,AlvesColinHonda2017,AlvesMeiwes2018,FrauenfelderSchlenk2006,MacariniSchlenk2011,Meiwes-thesis}. Some of these results generalize to positive contactomorphisms,  \cite{Dahinden2,Dahinden}.
Aspects of the relationship between the topological entropy of Hamiltonian diffeomorphisms and Floer homology are also studied in \cite{CM,CGG,FrauenfelderSchlenk2006}, and other aspects of the relationship between topological entropy of Hamiltonian diffeomorphisms and symplectic topology are studied in~\cite{BM-quasi,Khanevsky}.


Our investigation of stability properties of $h_{\topo}$ with respect to distances important in the symplectic context started in discussions with Leonid Polterovich. In particular, a question of Polterovich is whether $h_{\topo}$ is lower semi-continuous on the group of  Hamiltonian diffeomorphisms equipped with the Hofer metric. \color{black}  
As remarked above, the paper \cite{AlvesMeiwesBraids},  in which that question was answered positively in the case of surfaces, was a crucial  inspiration for the present paper; our approach also has points in common with the one of \cite{Hutchings-braids}, which generalises the results in \cite{AlvesMeiwesBraids}. Moreover, the work \cite{DahindenC0}, by the second author, was the first to use symplectic methods to study the stability of $h_{\rm top}$ with respect to $d_{C^0}$. In \cite{ADMM}, Merlin and the first three authors used variational methods to study the stability of $h_{\rm top}$ with respect to $\overline{d}_{C^0}$ for geodesic flows of closed Riemannian manifolds. Besides \cite{AlvesMeiwesBraids}, the articles \cite{CM} and \cite{Hutchings-braids} also study stability of $h_{\rm top}$ with respect to the Hofer metric.

\subsection{Future directions}

Motivated by the results of the present paper we propose the following conjectures:
\begin{conj} \label{conj:Reeb}
Given a closed co-oriented contact $3$-manifold $(Y,\xi)$, the entropy functional $$h_{\rm top}:\mathcal{R}(Y,\xi) \to [0,\infty)$$
is lower semi-continuous with respect to $d_{C^0}$.
\end{conj}

\begin{conj} \label{conj:geod}
Given a closed orientable surface $S$, the entropy functional $$h_{\rm top}:{\rm Met}(S) \to [0,\infty)$$
is lower semi-continuous with respect to $\overline{d}_{C^0}$.
\end{conj}

Clearly, Conjecture \ref{conj:Reeb} implies Conjecture \ref{conj:geod}.

\subsection*{Plan of the paper}
In Section~\ref{sec:preliminaries} we recall basic notions relevant to the paper, such as the definition of finite energy holomorphic curves in exact symplectic cobordisms. In Section~\ref{sec:robustproof}, we formulate the stability result,  Theorem~\ref{thm:stability},  and show how Theorem~\ref{thm:main} follows from it. Section~\ref{sec:proofmainresult} is the core section of the paper, where a proof of Theorem~\ref{thm:stability} is given. The last section, Section~\ref{sec:geodesicflows}, deals with our results on the lower semi-continuity and robustness of $h_{\topo}$ for geodesic flows. 
\subsection*{Acknowledgements}
The first author thanks Umberto Hryniewicz for helpful discussions. An important part of this work was completed when the first author visited the Tel Aviv University during March and April of 2023: he thanks TAU and Leonid Polterovich for the hospitality. The third author thanks Umberto Hryniewicz, Lev Buhovsky, Yaron Ostrover, and Leonid Polterovich for their support. 
\color{black}
\section{Preliminaries}\label{sec:preliminaries}

We want to study Reeb dynamics in contact manifolds, but our main tools -- holomorphic curves and holomorphic buildings -- require symplectic and complex structures. We start by recalling the definition of the symplectization of a contact manifold  (Section~\ref{subsec:symplectization}) and the definition of symplectic cobordisms (Section~\ref{sec:cobordisms}) between contact manifolds. These sympletic cobordisms can be thought of as an interpolation between the contact manifolds at the ends. These symplectic manifolds admit a class of almost complex structures which are suitable to study the underlying contact structures. 

In Section~\ref{subsec:HolomorphicCurves} we recall the definition of finite energy holomorphic curves in these symplectizations and symplectic cobordisms and describe how to interpret them as connections between periodic Reeb orbits. Multiple of these curves may fit together to holomorphic buildings, which we describe in Section~\ref{sec:buildings}. These are the main tools to connect closed Reeb orbits. 

In the following we fix a closed co-oriented contact $3$-manifold $(Y,\xi)$. As above we denote the space of contact forms on $(Y, \xi)$ by $\mathcal{R}(Y,\xi)$. 
\subsection{Periodic Reeb orbits}\label{sec:Reeb_orbits}
Let $\alpha$ be a contact form on $(Y,\xi)$. We denote by $\P(\alpha)$ the set of periodic orbits of the Reeb vector field $R_{\alpha}$, where we do not distinguish geometrically identical periodic orbits\footnote{I.e., orbits with the same image.} with identical periods, but we do distinguish geometrically identical periodic orbits with distinct periods. We may think of elements of $\P(\alpha)$ as equivalence classes of pairs $(\gamma,T)$ where $\gamma:\R\to Y$ is a $T$-periodic trajectory of $R_{\alpha}$, and two pairs $(\gamma_0,T_0)$, $(\gamma_1,T_1)$ are equivalent if, and only if, $\gamma_0(\R)=\gamma_1(\R)$ and $T_0=T_1$.  The period $T$ of $\gamma$ coincides with its action $\mathcal{A}_{\alpha}(\gamma) = \int_\gamma \alpha$.  If $(\gamma,T)$ is such that $T$ is minimal among the periods of the periodic orbits with the same geometric image as $\gamma$, the orbit  $\gamma=(\gamma,T)$ is said to be \textit{simple}.  A periodic orbit $\gamma = (\gamma,T) \in \P(\alpha)$ is called \textit{non-degenerate} if the transverse Floquet multipliers of $\gamma$ in period $T$ are not equal to $1$. The contact form $\alpha$ is called \textit{non-degenerate} if all elements in $\mathcal{P}(\alpha)$ are non-degenerate.

\subsection{Complex structures} Given a contact form $\alpha$ on $(Y,\xi)$, the 2-form $d\alpha$ restricts to a symplectic form on the vector bundle $\xi\to Y$. A complex structure $j$ on $\xi$ is said to be compatible with $d\alpha$, if $d\alpha (\cdot, j \cdot)$ is a positive-definite inner product on $\xi$ and hence 
\begin{equation}\label{metric} 
\left<u,v\right>_j:=\alpha(u)  \alpha(v) + d\alpha(\pi_\xi u, j \pi_\xi v), \qquad u,v \in TY,
\end{equation} 
is a Riemannian metric on $Y$. Here $\pi_\xi:TY \to \xi$ stands for the projection  along the Reeb vector field. We denote the set of $d\alpha$-compatible complex structures on $\xi$ by  $\mathfrak{j}(\alpha)$. The set $\mathfrak{j}(\alpha)$ is well known to be non-empty and contractible in the $C^\infty$-topology. Note that $\mathfrak{j}(\alpha)$ depends only on the co-oriented contact structure $\xi$, whose orientation is induced by $d\alpha$.

\subsection{A class of exact symplectic cobordisms}
\subsubsection{Symplectizations}\label{subsec:symplectization}
Let $\alpha$ be a contact form on $(Y,\xi)$. The symplectization of $\alpha$ is the product $\mathbb{R} \times Y$ equipped with the symplectic form $\omega = d\lambda$, where $\lambda = e^s \alpha$ and where $s$ denotes the $\R$-coordinate on $\mathbb{R} \times Y$. We think of the symplectization as a trivial exact symplectic cobordism, where the ``boundary'' consists of two copies of $(Y,\xi)$ at $s=\pm \infty$. We adopt the convention that this cobordism goes from $+\infty$ to $-\infty$.

Any $j \in \mathfrak{j}(\alpha)$ can be extended to an $\mathbb{R}$-invariant almost complex structure $J$ on $\mathbb{R} \times Y$ by demanding that
\begin{equation} \label{eq16}
 J \cdot \partial_s = R_\alpha \mbox{ and }  J|_\xi = j.
\end{equation} 
One checks that $J$ is $d(e^s \alpha)$-compatible. We denote by $\mathcal{J}(\alpha)$ the space of almost complex structures $J$ on $\mathbb{R} \times Y$ which satisfy \eqref{eq16} for some $j \in \mathfrak{j}(\alpha)$.

Almost complex structures $J\in \mathcal{J}(\alpha)$ are invariant under shifts in the real direction, that is, it holds that for all $\mathbf{b} \in \R$ we have that
$\mathbf{T}_{\mathbf{b}}^* J = J$, 
where the \textit{translation map} $\mathbf{T_b}: \R \times Y \to \R \times Y $, $\mathbf{b} \in \R$,  is defined by  
\begin{equation}\label{eq:translation}
    \mathbf{T_b}(s,p) := (s+\mathbf{b},p).
\end{equation}

\subsubsection{Exact symplectic cobordism between two contact forms}\label{sec:cobordisms}
Let $\alpha^+$ and $\alpha^-$ be two contact forms on the contact manifold $(Y,\xi)$. There exists a smooth function $f:Y \to (0,+\infty)$ satisfying $\alpha^+ = f \alpha^-$. Assume that $f>1$ pointwise (indicating a ``monotone'' increase). In order to interpolate between these contact forms, we deform the symplectization as follows.

Choose a smooth interpolation function $\chi:\R \times Y \to \R$ satisfying
\begin{equation}\label{chi_symp_form}
\begin{aligned}
 &\chi  = e^sf \mbox{ on } [1,+\infty)\times Y,\\
&\chi  = e^s  \mbox{ on } (-\infty,0]\times Y,\\
 &\partial_s\chi  > 0 \mbox{ on } \R \times Y,
\end{aligned}
\end{equation}
where $s$ is the $\R$-coordinate.  Let $\lambda = \chi \alpha_-$.  It then follows that 
\begin{equation}\label{exact_symp_form_on_cob}
\varpi:= d\lambda
\end{equation}
is a symplectic form on $\R \times Y$.

We call $W= (\R \times Y, \lambda, \varpi)$ an \textit{exact symplectic cobordism from $\alpha^+$ to $\alpha^-$}. On such a cobordism we consider almost complex structures $\bar J$ defined as follows: fix any two choices $J^+ \in \mathcal{J}(\alpha^+)$, $J^- \in \mathcal{J}(\alpha^-)$ and choose $\bar J$ satisfying
$$
\begin{aligned}
\bar J & =  J^+   \mbox{ on }  [1,+\infty)\times Y,  \\
\bar J & =  J^-   \mbox{ on }  (-\infty,0]\times Y,  \\
\bar J & \mbox{ is compatible with } \varpi \mbox{ on } [0,1]\times Y.
\end{aligned}
$$
We say that $\bar J$ is compatible with the cobordism $W$ and denote the space of such almost complex structures by $\mathcal{J}_\varpi(J^+,J^-)$. Finally we note that standard arguments will show that $\mathcal{J}_\varpi(J^+,J^-)$ is non-empty and contractible in the $C^\infty$-topology.

\subsubsection{Splitting families} \label{splitting}

Let $\alpha^+$, $\alpha$ and $\alpha^-$ be contact  forms on $(Y,\xi)$. 
Let $(\R\times Y,\lambda^+, \varpi^+)$ be an exact symplectic cobordism from $\alpha^+$ to $\alpha$ and let $(\R\times Y,\lambda^-,\varpi^-)$ be an exact symplectic cobordism from $\alpha$ to $\alpha^-$, as defined in the previous paragraph.

Fix $J^{\pm} \in \mathcal{J}(\alpha^{\pm}), J_{\alpha} \in \mathcal{J}(\alpha)$.  Next we choose $\bar J^+ \in \mathcal{J}_{\varpi^+}(J^+,J_{\alpha})$ and $\bar J^- \in \mathcal{J}_{\varpi^-}(J_{\alpha},J^-)$ and then define for any $R>0$ the almost complex structure $\widehat J^R$ by
\begin{equation}\label{split}
\begin{aligned}
\widehat J^R &= (\mathbf{T}_{-R})^*\bar J^+\mbox{ on } [0,+\infty)\times Y,\\
\widehat J^R &= (\mathbf{T}_{R+1})^*\bar J^-\mbox{ on }  (-\infty,0]\times Y,\\
\end{aligned}
\end{equation} where $\mathbf{T}_R: \R \times Y \to \R \times Y$ are the translation maps defined in \eqref{eq:translation}.
    By definition, $\widehat J^R$ is smooth and agrees with $J_{\alpha}$ on the neck region $[-R,R]\times Y$. 
Any such almost complex structure $\widehat{J}^R$  is compatible with the symplectic form $\widehat{\varpi}^R = d\widehat{\lambda}^R$, where 
\begin{equation}
\begin{aligned}
\widehat{\lambda}^R &= (\mathbf{T}_{-R})^*\lambda^+\mbox{ on } [0,+\infty)\times Y,\\
\widehat{\lambda}^R &= (\mathbf{T}_{R+1})^* \lambda^-\mbox{ on }  (-\infty,0]\times Y.\\
\end{aligned}
\end{equation}
We refer to $(\R \times Y, \widehat\lambda^R, \widehat\varpi^R)$ as {splitting symplectic cobordisms}.

\subsection{Finite energy holomorphic curves in exact symplectic cobordisms}\label{subsec:HolomorphicCurves} 
The goal is to connect periodic orbits of (possibly equal) Reeb vector fields. The standard way this is done is to consider pseudoholomorphic curves with singularities (called punctures) in the symplectic cobordism. If the punctures are not removable but have finite energy, then it turns out that their divergence behaviour is controlled: the $\R$-coordinate of the symplectization will tend to $\pm \infty$ and the curve approximates in the limit a trivial cylinder over a Reeb orbit. We may think of the curve as a cobordism from the link of Reeb orbits covered by punctures that tend to $ +\infty$ to the link of Reeb orbits covered by punctures that tend to $ -\infty$.

\subsubsection*{In trivial cobordisms} Let $(S,j)$ be a compact boundaryless Riemann surface and let $J \in \mathcal{J}(\alpha)$. A \textit{finite energy (pseudo-)holomorphic curve in $(\R \times Y,J)$} is a smooth map $$ \wtil=(a,w): S \setminus \Gamma \to \mathbb{R} \times Y, $$ where $\Gamma\subset S$ is a finite set of punctures, that satisfies
$$\begin{aligned}
& \bar \partial_J \wtil := \frac{1}{2}\left(d\wtil + J(\wtil) \circ d\wtil \circ j\right)=0,
\end{aligned}$$
and has finite Hofer energy
\begin{equation}
0 < E(\widetilde{w}):= \sup_{q \in \mathcal{E}} \int_{S \setminus \Gamma} \widetilde{w}^*d(q\alpha)<+ \infty,
\end{equation}
where
\begin{equation}\label{energia} 
\mathcal{E}= \{ q: \mathbb{R} \to [0,1] \mbox{ smooth }\, |\,  q' \geq 0\}.
\end{equation}

If $\widetilde{w}$ is a cylinder, that is, if $S$ is the Riemann sphere and $\Gamma$ consists of two points, then in suitable cylindrical coordinates $\varphi:\R \times S^1 \to S\setminus \Gamma$, the holomorphic equation for $\widetilde{u} = \widetilde{w} \circ \varphi:\R \times S^1 \to \R \times Y$ becomes 
$\partial_s \widetilde{u} + J(\widetilde{u}) \circ \partial_t \widetilde{u}= 0$. Special cylinders are the  \textit{trivial cylinders over a Reeb orbit} $(\gamma,T)$ of $\alpha$, cylinders $\widetilde{u} = \widetilde{w} \circ \varphi$ of the form  
$$\widetilde{u}(s,t) = (Ts,\gamma(Tt)).$$

\textit{Regular almost complex structures in trivial cobordisms.} Assume that $\alpha$ is non-degenerate. An almost complex structure $J\in \mathcal{J}(\alpha)$ will be called \textit{regular} if the linearization of $\overline{\partial}_J$ on any simply covered finite energy holomorphic curve $\widetilde{w}$ on $(\mathbb{R}\times Y, J)$ (which is in fact a Fredholm operator between suitable function spaces (see~\cite{Dragnev})) is surjective. The regular almost complex structures in $\mathcal{J}(\alpha)$ form a $C^{\infty}$-generic subset. 
\color{black}
\subsubsection*{In exact symplectic cobordisms} Let $\alpha_+,\alpha_-$ be contact forms on $(Y,\xi)$ and consider an exact symplectic cobordism $W=(\mathbb{R} \times Y, \lambda, \varpi)$ from $\alpha^+$ to $\alpha^-$ of the kind defined in Section~\ref{sec:cobordisms}. Consider $\bar J \in \mathcal{J}_\varpi(J^+,J^-)$, where $J^+ \in \mathcal{J}(\alpha^+)$ and $J^- \in \mathcal{J}(\alpha^-)$. A \textit{finite energy (pseudo-)holomorphic curve  in $(W,\bar{J})=(\R\times Y, \lambda, \varpi, \bar{J})$} is a smooth map $$ \wtil=(a,w): S \setminus \Gamma \to \R \times Y, $$ where $\Gamma \subset S$ is a finite set, that satisfies
$$
\begin{aligned}
& \bar \partial_{\bar J} \wtil := \frac{1}{2}\left(d\wtil + \bar J(\wtil) \circ d\wtil \circ j\right)=0,
\end{aligned}
$$
and has finite energy
\begin{equation}\label{encob}
0<E_{\alpha^+}(\wtil) + E_c(\wtil) + E_{\alpha^-} (\wtil) < +\infty,
\end{equation}
where
$$
\begin{aligned}
&E_{\alpha^+}(\wtil) = \sup_{q \in \mathcal{E}} \int_{\wtil^{-1}([1,+\infty)\times Y)} \wtil^*d(q\alpha^+), \\
&E_c(\wtil) = \int_{\wtil^{-1}([0,1]\times Y)} \wtil^*\varpi, \\
&E_{\alpha^-}(\wtil) = \sup_{q \in \mathcal{E}} \int_{\wtil^{-1}((-\infty,0]\times Y)} \wtil^*d(q\alpha^-).
\end{aligned}
$$

Let $\widehat J^R$ be a splitting almost complex structure as in Section~\ref{splitting}. These are defined by~\eqref{split}, where $J^\pm\in\J(\alpha^\pm)$, $J_{\alpha}\in\J(\alpha)$, $\bar J^+ \in \mathcal{J}_{\varpi^+}(J^+,J_{\alpha})$ and $\bar J^- \in \mathcal{J}_{\varpi^-}(J_{\alpha},J^-)$. Finite energy holomorphic curves $\wtil$ in $(W, \widehat{J}^R)$ are defined analogously, but we need a modified version of the finite energy condition:
\begin{equation}
0<E_{\alpha^+} (\wtil)+ E_{\alpha,\alpha^+}(\wtil) + E_{\alpha}(\wtil)+ E_{\alpha^-,\alpha}(\wtil)+ E_{\alpha^-}(\wtil)  < +\infty,
\end{equation}
where
$$
\begin{aligned}
&E_{\alpha^+}(\wtil) = \sup_{q \in \mathcal{E}} \int_{\wtil^{-1}([R+1,+\infty)\times Y)} \wtil^*d(q\alpha^+), \\
&E_{\alpha,\alpha^+}(\wtil) =\int_{\wtil^{-1}([R,R+1]\times Y)} \wtil^*(T_{-R})^*\varpi^+ ,\\
&E_{\alpha}(\wtil) = \sup_{q \in \mathcal{E}} \int_{\wtil^{-1}([-R,R]\times Y)} \wtil^*d(q\alpha), \\
&E_{\alpha^-,\alpha}(\wtil) = \int_{\wtil^{-1}([-R-1,-R]\times Y)} \wtil^*(T_{R+1})^*\varpi^- ,\\
&E_{\alpha^-}(\wtil) = \sup_{q \in \mathcal{E}} \int_{\wtil^{-1}((-\infty,-R-1]\times Y)} \wtil^*d(q\alpha^-) .\\
\end{aligned}
$$ 
See~\eqref{energia} and Section~\ref{splitting} for definitions.

The elements of $\Gamma \subset S$ are called punctures of $\wtil$. Assume that all punctures are non-removable. We only describe the behaviour of finite energy pseudoholomorphic curves near non-removable punctures on exact symplectic cobordisms, since the behaviour in the cases of symplectizations and splitting symplectic cobordisms is the same.

According to \cite{Hofer93,HWZ} the punctures are classified in two different types. Before presenting this classification we introduce some notation:
\begin{itemize}
\item If $z\in \Gamma $ then a neighbourhood of $z$ in $(S,j)$ is biholomorphic to the unit disk $\mathbb{D}\subset\C$ in such a way that $z\equiv 0$. Then $\mathbb{D}\setminus \{0\}$ is biholomorphic to $[0,+\infty) \times \R / \Z$ via the map $s+it \mapsto e^{-2\pi(s+it)}$. We always write $\wtil=\wtil(s,t)$ near $z\in \Gamma$ using these exponential holomorphic coordinates $(s,t)$ in $[0,+\infty)\times \R/ \Z$.

\end{itemize}
The non-removable punctures of a finite energy pseudoholomorphic curve $\wtil=(a,w)$ are classified as follows:
\begin{itemize}
\item $z \in \Gamma $ is a positive puncture, i.e., $a(z')\to + \infty$ as $z' \to z$ and given a sequence $s_n \to +\infty$ there exists a subsequence, still denoted $s_n$, and a periodic Reeb orbit $\gamma^+$ of $\alpha^+$ with period $T^+$ such that $w(s_n,\cdot)$ converges in $C^\infty$ to $\gamma^+(T^+\cdot)$ as $n\to +\infty$.
\item $z \in \Gamma $ is a negative puncture, i.e., $a(z')\to - \infty$ as $z' \to z$ and given a sequence $s_n \to +\infty$ there exists a subsequence, still denoted $s_n$, and a periodic Reeb orbit $\gamma^-$ of $\alpha^-$ with period $T^-$ such that $w(s_n,\cdot)$ converges in $C^\infty$ to $\gamma^-(-T^-\cdot)$ as $n\to +\infty$.

\end{itemize}

\begin{defn}
For a  puncture $z$ we say that a periodic Reeb orbit~$\gamma \in \P(\alpha)$ is an \textbf{asymptotic limit} of $\wtil$ at $z$ if there is a sequence $s_n$ such that  $w(s_n,\cdot)$ converges to~$\gamma$. If all asymptotic limits at $z$ coincide with $\gamma$, then we say that $\wtil$ is asymptotic to $\gamma$.
\end{defn}

\begin{rem}
If a non-degenerate periodic Reeb orbit $\gamma$ is an asymptotic limit of $\wtil$ at a puncture $z$ then $\wtil$ is asymptotic at $z$ to $\gamma$. This follows from results of~\cite{Ab,HWZ} where a much more detailed description of the convergence is given. 
\end{rem}

A punctured disk map $u:\D\setminus \{z\}\to \R \times Y$ to $(\mathbb{R} \times Y, \varpi)$  with  $\overline{\partial}_{\bar {J}} u = 0$ that has exactly one positive or negative puncture $z$ in its interior and such that  $u$ is asymptotic at $z$ to a periodic orbit $\gamma$ in the sense above, is called a \textit{holomorphic half-cylinder} positively resp.\ negatively  asymptotic to $\gamma$.

\subsection{Holomorphic buildings} \label{sec:buildings}

Holomorphic buildings appear in the compactifications of moduli space of holomorphic curves in symplectic cobordisms. They consist of multiple holomorphic curves that are "glued together" in coinciding pairs of positive and negative asymptotic limits of punctures. For the Riemann surfaces, these pairs are called nodes, for the pseudoholomorphic curves they are called breaking points. We recall here the definition of holomorphic buildings in the two settings that we need.

\subsubsection{Nodal Riemann surfaces}

\begin{defn}
    A {\bf nodal Riemann surface} with $D$ nodes and $m$ marked points is a tuple $(S,j,\Gamma, \Delta)$ consisting of:
\begin{itemize}
    \item a closed Riemann surface $(S,j)$,
    \item a set $\Gamma$ of $m$ punctures in $S$,
    \item an unordered set $\Delta$ of $2D$ points in $S\setminus \Gamma$ equipped with a fixed-point free involution $\upsigma: \Delta \to \Delta$. Each pair $\{z,\upsigma(z)\}$ is called a node.  
\end{itemize}
\end{defn}

Let $\widehat{\mathbf{S}}$ be the surface obtained by performing connected sum at each node. We then say that $(S,j,\Gamma,\Delta)$ is {\bf connected} if $\widehat{\mathbf{S}}$ is connected. The genus of $\widehat{\mathbf{S}}$ is called the arithmetic genus of $(S,j,\Delta)$.
Two nodal surfaces $(S,j,\Gamma,\Delta)$ and $(S',j',\Gamma',\Delta')$ are equivalent if there is a biholomorphism from $(S,j)$ to $(S',j')$ sending $\Delta$ to $\Delta'$ and $\Gamma$ to $\Gamma'$.

If a punctured Riemann surface $\dot{S}$ is obtained from a closed Riemann surface $S$ by removing the points in a finite set $Z \subset S$, we define the circle compactification 
$$\overline{S}:= \dot{S} \cup (\cup_{z\in Z}\updelta_z),$$
where for each $z\in Z$ the circle $\updelta_z$ is defined by 
$$\updelta_z := (T_z S \setminus \{0\})  /_{\R^+},$$
where $\R^+$ is the set of positive real numbers and its action on $T_z S$ is given by multiplication. There is a natural topology on $\overline{S}$ which makes it compact, and we refer the reader to Chapter 9 of \cite{Wendl-lectures} for its definition.

\subsubsection{Holomorphic buildings in exact symplectic cobordisms.}\label{sec:holomorphic_buildings_exact_cobordisms}

\

Let $\alpha^{\pm} \in \mathcal{R}(Y,\xi)$, and let  $W=(\R \times Y,\lambda, \varpi)$ be an exact symplectic cobordism from $\alpha^+$ to $\alpha^-$ as consider in Section \ref{sec:cobordisms}. 
Choose almost complex structures $J^+ \in \mathcal{J}(\alpha^+)$, $J^- \in \mathcal{J}(\alpha^-)$, and  ${\bar J} \in \mathcal{J}_\varpi(J^+,J^-)$.

\begin{defn}\label{def:building} (See Figure~\ref{fig:building}). Given numbers $\kappa^+\geq 0$ and $\kappa^-\geq 0$, a holomorphic building of type $\kappa^+|1|\kappa^-$ and of arithmetic genus $0$ in $(W,\bar{J}) =(\R \times Y, \lambda, \varpi,\bar{J})$ is a tuple
$$\mathbf{u}=(S,j,\Gamma, \Delta,\upphi,L,(\overline{u}^i)_{-\kappa^- \leq i\leq \kappa^+}),$$
with the various data defined as follows:
\begin{itemize}
    \item $(S,j,\Gamma, \Delta)$ is a connected nodal Riemann surface of arithmetic genus $0$. The nodes in $\Delta$ are called the breaking pairs of $\mathbf{u}$, and the points of $\Gamma$ are called the punctures of $\mathbf{u}$.
    \item The set $\Gamma$ is partitioned in $\Gamma^+$ and $\Gamma^-$. The set $\Gamma^+$ is called the set of positive punctures of $\mathbf{u}$ and the set $\Gamma^-$ is called the set of negative punctures of $\mathbf{u}$.
    \item The function $L$ on $S$ is called the \textbf{level structure}. It satisfies
        \begin{itemize}
        \item[-] $L: S \to \{-\kappa^-,\ldots,-1,0,1,\ldots,\kappa^+\}$ is locally constant and attains every value in $\{-\kappa^-,\ldots,-1,0,1,\ldots,\kappa^+\}$.
        \item[-] Each breaking pair in $\Delta$ can be labelled as $\{z^+,z^-\}$ such that $L(z^+)-L(z^-)=1$.
        \item[-] $L(\Gamma^+)= \kappa^+$ and $L(\Gamma^-)= -\kappa^-$.
        \end{itemize}
    We denote by $\dot{S}_i:= (S \setminus (\Gamma\cup \Delta)) \cap L^{-1}(i)$ the levels of the building.
    \item The decoration $\upphi$ is a choice of orientation reversing orthogonal map $$\upphi: \updelta_{z^+} \to \updelta_{z^-},$$
    for each breaking pair $\{z^+,z^-\}$  in $\Delta$.
    \item For each $i \in \{-\kappa^-,\ldots,-1,0,1,\ldots,\kappa^+\}$ there is a finite energy holomorphic curve $$\overline{u}^i: (\dot{S}_i,j) \to (W_i,J_i)$$ called the $i$-th level of $\mathbf{u}$, where 
    \begin{itemize}
        \item[]  $W_i$ is the symplectization of $\alpha^+$ and $J_i = J^+$ if $i>0$,
        \item[] $(W_0,J_0) = (W,\bar{J})$,
        \item[]  $W_i$ is the symplectization of $\alpha^-$ and $J_i=J^-$ if $i<0$.
    \end{itemize}
    \item  It holds that $\Gamma^+$ are the positive punctures of $\overline{u}^{\kappa^+}$ and $\Gamma^-$ are negative punctures of $\overline{u}^{\kappa^-}$.  For each breaking pair $\{z^+,z^-\}$  in $\Delta$ labelled such that $L(z^+)-L(z^-)=1$, $z^+$ is a negative puncture of $\overline{u}^{L(z^+)}$ and $z^-$ is a positive puncture of $\overline{u}^{L(z^-)}$. Moreover,  $\overline{u}^{L(z^+)}$ at $z^+$ and $\overline{u}^{L(z^-)}$ at $z^-$ detect the same Reeb orbit $\gamma$ (of $\alpha^+$, $\alpha$ or $\alpha^-$), and if $\overline{u}_+: \delta_{z^+} \to \gamma$ and $\overline{u}_-: \delta_{z^-} \to \gamma$ are the parametrizations of $\gamma$ induced by $\overline{u}^{L(z^+)}$ at $z^+$ and by $\overline{u}^{L(z^-)}$ at $z^-$, respectively, we have $\overline{u}_{-}\circ \upphi ( \delta_{z^+}) = \overline{u}_{+}$.
\end{itemize}
\end{defn}

\begin{figure}
    \centering
    \includegraphics[width=1\linewidth]{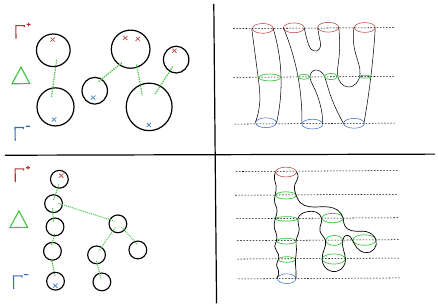}
    \caption{On the top, there is a building that fits Definition~\ref{def:building}. Left is the representation as surfaces with marked points and pairs, the images may look like the curve on the right. On the bottom, there is a cylinder with a downwards directed bubble tree branching off (Section~\ref{sec:defbreakingmain}), which is the kind of curve that arises as SFT-limit of holomorphic cylinders in our setting (Lemma~\ref{lem:v_no_topbot_level}). }
    \label{fig:building}
\end{figure}
Given a holomorphic building $\mathbf{u}=(S,j,\Gamma, \Delta, \upphi,L,(\overline{u}^i)_{-\kappa^- \leq i\leq \kappa^+})$, we let $\dot{S}:= S\setminus (\Gamma \cup \Delta)$ and consider $\overline{S}= \dot{S} \cup (\cup_{z\in \Gamma}\updelta_z) (\cup_{\{z^+,z^-\} \in \Delta} \delta_{z^+} \sim \delta_{z^-} ))$. As our notation suggests, in $\overline{S}$ for each pair $\{z^+,z^-\} \in \Delta$, the circles $ \delta_{z^+}$ and  $ \delta_{z^-}$ are identified, and the identification is done via the map $\upphi: \delta_{z^+} \to \delta_{z^-}$ appearing in the definition of $\mathbf{u}$. We endow $\overline{S}$ with the natural topology that makes it compact; see Chapter 9 of \cite{Wendl-lectures} for the precise definition. We call $\dot{S}$ the domain of  $\mathbf{u}$ and $\overline{S}$ the compactified domain of $\mathbf{u}$.

We will use the following terminology:
\begin{itemize}
    \item We say that $\mathbf{u}$ is a \textit{broken cylinder} if $\dot{S}_i$ is a cylinder for each \linebreak $i \in \{-\kappa^-,\ldots,-1,0,1,\ldots,\kappa^+\}$.
    \item We say that $\mathbf{u}$ is a \textit{cylinder with branches} if the compactified domain $\overline{S}$ is homeomorphic to $[0,1]\times S^1$.
\end{itemize}

\begin{rem}
    If $\mathbf{u}$ is a holomorphic building as in this section and \linebreak$(\overline{u}^i)_{i \in \{-\kappa^-,\ldots,-1,0,1,\ldots,\kappa^+\}}$ denote the levels of $\mathbf{u}$, we will sometimes write $$\mathbf{u}=(\overline{u}^{\kappa^+},\ldots,\overline{u}^1,\overline{u}^0,\overline{u}^{-1},\ldots,\overline{u}^{-\kappa^{-}}),$$ keeping in mind that this notation does not contain all the data needed to characterize the building. We call $\overline{u}^{\kappa^+}$ the top level and $\overline{u}^{-\kappa^{-}}$ the bottom level of the building. 
\end{rem}

\subsubsection{Holomorphic buildings in pairs of exact symplectic cobordisms.}\label{sec:holom_buildings_pairs}

\

Let $\alpha^{\pm}, \alpha \in \mathcal{R}(Y,\xi)$. Let  $W_{\rm top}= (\R \times Y,\lambda_{\rm top}, \varpi_{\rm top})$ be an exact symplectic cobordism from $\alpha^+$ to $\alpha$, and $W_{\rm bot}= (\R \times Y,\lambda_{\rm bot}, \varpi_{\rm bot})$ an exact symplectic cobordism from $\alpha$ to $\alpha^-$. Choose almost complex structures $J^+ \in \mathcal{J}(\alpha^+)$, $J^- \in \mathcal{J}(\alpha^-)$, $J_{\alpha} \in \mathcal{J}(\alpha)$,  $J_{\rm top}\in \mathcal{J}_{\varpi_{\rm top}}(J^+,J_{\alpha})$, and $J_{\rm bot}\in \mathcal{J}_{\varpi_{\rm bot}}(J_{\alpha},J^- )$.

Given numbers $\kappa^+\geq 0$, $\kappa^-\geq 0$, and $\kappa\geq 0$, a holomorphic building of type $\kappa^+|1|\kappa|1|\kappa^-$ and  arithmetic genus $0$ in the pair $[(W_{\rm top},J_{\rm top}),(W_{\rm bot},J_{\rm bot})]$ is a tuple
$$\mathbf{u}=(S,j,\Gamma, \Delta,\upphi, L,(\overline{u}_i)^{-\kappa^- \leq i\leq \kappa +  \kappa^+ +1}),$$
with the various data defined as follows:
\begin{itemize}
    \item $(S,j,\Gamma, \Delta)$ is a connected nodal Riemann surface of arithmetic genus $0$. The pairs in $\Delta$ are called the breaking pairs of $\mathbf{u}$, and the points of $\Gamma$ are called the punctures of $\mathbf{u}$.
    \item The set $\Gamma$ is partitioned in $\Gamma^+$ and $\Gamma^-$. The set $\Gamma^+$ is called the set of positive punctures of $\mathbf{u}$ and the set $\Gamma^-$ is called the set of negative punctures of $\mathbf{u}$.
    \item The level structure is a locally constant function \linebreak $L: S \to \{-\kappa^- ,\ldots,-1,0,1,\ldots,\kappa^+ + \kappa +1 \}$ that attains every value in $\{-\kappa^- ,\ldots,-1,0,1,\ldots,\kappa^+ + \kappa +1\}$ and satisfies:
    \begin{itemize}
        \item[-] Each breaking pair $\{z^+,z^-\}$  in $\Delta$ can be labelled such that $L(z^+)-L(z^-)=1$.
        \item[-] $L(\Gamma^+)= \kappa^+ + \kappa +1$ and $L(\Gamma^-)= -\kappa^- $.
    \end{itemize}
    We let $\dot{S}_i:= (S \setminus (\Gamma\cup \Delta)) \cap L^{-1}(i)$.
    \item The decoration is a choice of orientation reversing  orthogonal map $$\upphi: \updelta_{z^+} \to \updelta_{z^-},$$
    for each breaking pair $\{z^+,z^-\}$  in $\Delta$.
    \item For each $i \in \{-\kappa^- ,\ldots,-1,0,1,\ldots,\kappa^+ + \kappa +1\}$ we have a finite energy holomorphic curve $$\overline{u}^i: (\dot{S}_i,j) \to (W_i,J_i)$$ called the $i$-th level of $\mathbf{u}$, where 
    \begin{itemize}
        \item[]  $W_i$ is the symplectization of $\alpha^+$ and $J_i = J^+$ if \linebreak $\kappa +1 < i \leq \kappa^+ + \kappa + 1$,
        \item[] $(W_{\kappa+1},J_{\kappa+1}) = (W_{\rm top},J_{\rm top})$,
        \item[]  $W_i$ is the symplectization of $\alpha$ and $J_i=J_{\alpha}$ if $0 < i \leq \kappa$,
        \item[] $(W_{0},J_{0}) = (W_{\rm bot},J_{\rm bot})$,
        \item[]  $W_i$ is the symplectization of $\alpha^-$ and $J_i = J^-$ if $-\kappa^- \leq i < 0$. 
    \end{itemize}
    \item It holds that $\Gamma^+$ are positive punctures of $\overline{u}^{\kappa^+}$ and $\Gamma^-$ are positive punctures of $\overline{u}^{\kappa^-}$. For each breaking pair $\{z^+,z^-\}$  in $\Delta$ labelled such that $L(z^+)-L(z^-)=1$, $z^+$ is a negative puncture of $\overline{u}^{L(z^+)}$ and $z^-$ is a positive puncture of $\overline{u}^{L(z^-)}$. Moreover,  $\overline{u}^{L(z^+)}$ at $z^+$ and $\overline{u}^{L(z^-)}$ at $z^-$ detect the same Reeb orbit $\gamma$ (of $\alpha^+$, $\alpha$ or $\alpha^-$), and if $\overline{u}_+: \delta_{z^+} \to \gamma$ and $\overline{u}_-: \delta_{z^-} \to \gamma$ are the parametrizations of $\gamma$ induced by $\overline{u}^{L(z^+)}$ at $z^+$ and by $\overline{u}^{L(z^-)}$ at $z^-$, respectively, we have $\overline{u}_{+}\circ \upphi ( \delta_{z^+}) = \overline{u}_{-}$.
\end{itemize}
We will use the following terminology:
\begin{itemize}
    \item We say that $\mathbf{u}$ is a \textit{broken cylinder} if $\dot{S}_i$ is a cylinder for each \linebreak $i \in \{-\kappa^-,\ldots,-1,0,1,\ldots,\kappa^+ + \kappa + 1 \}$.
    \item We say that $\mathbf{u}$ is a \textit{cylinder with branches} if the compactified domain $\overline{S}$ is homeomorphic to $[0,1]\times S^1$.
\end{itemize}

\begin{rem}
    If $\mathbf{u}$ is a holomorphic building as in this section and \linebreak $(\overline{u}^i)_{i \in \{-\kappa^-,\ldots,-1,0,1,\ldots,\kappa^+ + \kappa+1\}}$ denote the levels of $\mathbf{u}$, we will sometimes write $$\mathbf{u}=(\overline{u}^{\kappa^+ + \kappa + 1},\ldots,\overline{u}^{\kappa+2},\overline{u}^{\kappa+1},\overline{u}^{\kappa},\ldots,\overline{u}^1,\overline{u}^0,\overline{u}^{-1},\ldots,\overline{u}^{-\kappa^- }),$$ keeping in mind that this notation does not contain all the data needed to characterize the building.
\end{rem}

\subsubsection{Breaking orbits in cylinder with branches.} \label{sec:defbreakingmain}
If a holomorphic building $\mathbf{u}=(S,j,\Gamma, \Delta, \upphi,L,\overline{u})$ is a cylinder with branches then the set $\Gamma$ always contains exactly two points. 

The building separates into a main cylinder and branches. We make the distinction as follows: let $\{z_-,z_+\} \in \Delta_{\rm main}$ if the circle in the compactified domain $\overline{S}$ of $\mathbf{u}$ associated to the pair $\{z_-,z_+\}$ is non-contractible. The elements of $\Gamma$ together with the punctures belonging to pairs in $\Delta_{\rm main}$ will be called the main punctures of $\mathbf{u}$. The orbits detected by the main punctures of $\mathbf{u}$ will be called the main orbits of $\mathbf{u}$. It follows directly from the definitions that each level of $\mathbf{u}$ contains exactly two main punctures.

\subsection{Convergence to a holomorphic building}
\subsubsection{In exact symplectic cobordisms between two contact forms.}\label{sec:SFT1}

Let $\alpha^{\pm}$ be two contact forms on $(Y,\xi)$, and let $W =(\R \times Y, \lambda, \varpi)$ be a symplectic cobordism from $\alpha^+$ to $\alpha^-$. Let $J^{\pm}\in \mathcal{J}(\alpha^{\pm})$, and choose  $\bar{J}\in \mathcal{J}_{\varpi}(J^+, J^-)$ as defined in Section~\ref{sec:cobordisms}.  
Let $\mathbf{u} = (S,j, \Gamma, \Delta, \upphi, L, (\overline{u}^i)_{-\kappa^-\leq i\leq \kappa^+})$ be a holomorphic building of type $\kappa^-|1|\kappa_+$ in $(W,\bar{J})$ and of arithmetic genus $0$, and denote by $\widehat{\mathbf{S}}$ the surface obtained from $S$ by performing connected sum at each node, and let $\dot{S}$ be defined as in Section~ \ref{sec:holomorphic_buildings_exact_cobordisms}.

Let $(\Sigma,j_{\Sigma})$ be the Riemann sphere.  We say that a sequence of holomorphic curves $\widetilde{w}_k: \Sigma\setminus \Gamma_k \to \R \times Y$ in $(W, \bar{J})$  \textit{converges in the SFT-sense} to $\mathbf{u}$ if the following holds. 

There are homeomorphisms $\varphi_k: \widehat{\mathbf{S}} \to \Sigma$ that are smooth in $S\setminus \Delta$ viewed as a subset in $\widehat{\mathbf{S}}$, that 
map $\Gamma^{\pm}$ to $\Gamma^{\pm}_k$, and satisfy 
$\varphi^*_k j_{\Sigma} \to j$ in $C^{\infty}_{\loc}(\widehat{\mathbf{S}}\setminus \Delta)$.
Moreover, for any $-\kappa^- \leq i \leq  \kappa^+ $,
let  $\widetilde{w}^i_k :=\widetilde{w}_k \circ \varphi_k|_{\dot{S}_i}$. Then, 
\begin{itemize}
\item $\widetilde{w}^0_k \to \overline{u}^0$ in  $C^{\infty}_{\loc}(\dot{S}_0, W)$, 
\item for $\pm i >0$, there is a sequence ${\textbf b}^i_k \to  \pm \infty$ such that 
$T_{-{\textbf b}_k^i} \circ \widetilde{w}^i_k \to \overline{u}^i$ in $C^{\infty}_{\loc}(\dot{S}_i, \R \times Y)$,
\item it holds that ${\textbf b}_k^{i+1} - {\textbf b}_k^i \to + \infty$ as $k\to \infty$ for all $i< \kappa^{+}$.
\end{itemize}
Finally $\widetilde{w}_k$ at $\Gamma_k^+$ and $\overline{u}^{\kappa^+}$ at $\Gamma^+$ detect the same Reeb orbit, and $\widetilde{w}_k$ at $\Gamma^-_k$ and $\overline{u}^{-\kappa^-}$ at $\Gamma^-$ detect the same Reeb orbit. 

We have the following compactness  theorem, \cite{CPT}, see also \cite{Wendl-lectures}.

\begin{thm}\label{thm:SFT1}
Assume $\alpha^{\pm}$ are non-degenerate. Let $(\Sigma,j_{\Sigma})$ be the Riemann sphere.  Then, for any sequence of holomorphic curves $\widetilde{w}_k: S\setminus \Gamma_k \to \R \times Y$ in $(W, \bar{J})$ with uniformly bounded energy, there is a  holomorphic building $\mathbf{u}$ of type $\kappa^-|1|\kappa^+$ of arithmetic genus $0$ in $(W,\bar{J})$ such that after restricting to a subsequence, $\widetilde{w}_k$ converges to $\mathbf{u}$ in the SFT-sense. Moreover, if $\# \Gamma_k^+=\# \Gamma_k^-=1$ for all $k\in \N$, then $\mathbf{u}$ is a cylinder with branches.  
\end{thm}

\begin{rem}
A similar result holds if we consider holomorphic curves as above in a sequence $W_k$ of exact symplectic cobordisms from $\alpha^+$ to $\alpha^-$ that converge to such a cobordim $W$ in the $C^{\infty}$-topology. 
\end{rem}

\subsubsection{In splitting families}\label{sec:SFT2}

\ 

Let $\alpha^{\pm}, \alpha \in \mathcal{R}(Y,\xi)$. Let  $W_{\rm top}= (\R \times Y,\lambda_{\rm top}, \varpi_{\rm top})$ be an exact symplectic cobordism from $\alpha^+$ to $\alpha$, and $W_{\rm bot}= (\R \times Y,\lambda_{\rm bot}, \varpi_{\rm bot})$ an exact symplectic cobordism from $\alpha$ to $\alpha^-$. Choose almost complex structures $J^{\pm} \in \mathcal{J}(\alpha^{\pm})$, $J_{\alpha} \in \mathcal{J}(\alpha)$, $J_{\rm top}\in \mathcal{J}_{\varpi_{\rm top}}(J^+,J_{\alpha})$, and $J_{\varpi_{\rm bot}}\in \mathcal{J}_{\varpi_{\rm bot}}(J_{\alpha},J^- )$.
As constructed in Section \ref{splitting} with $\varpi^+ = \varpi_{\rm top}$, $\varpi^- = \varpi_{\rm bot}$, $\bar{J}^+ =J_{\rm top}$, $\bar{J}^- = J_{\rm bot}$, let $W_R= (\R \times Y, \widehat\lambda^R, \widehat\varpi^R)$ be a family of splitting symplectic cobordisms and $\widehat{J}^R$ a family of splitting almost complex structures  on $W_R$.  
Moreover, let $\mathbf{u} = (S,j, \gamma, \Delta, \upphi, L, (\overline{u}^i)_{-\kappa^-\leq i\leq \kappa + \kappa^+})$ be a holomorphic building of type $\kappa^+|1|\kappa|1|\kappa^-$ and arithmetic genus $0$ in the pair $[(W_{\rm top}, J_{\rm top}), (W_{\rm bot}, J_{\rm bot})]$. 

Let $(\Sigma, j_{\Sigma})$ be the Riemann sphere. We say that a sequence 
of holomorphic curves $\widetilde{w}_k: \Sigma \setminus \Gamma_k \to \R \times Y$ in $(W^{R_k}, \widehat{J}^{R_k})$ with $R_k \to + \infty$ \textit{converges in the SFT-sense} to the holomorphic building $\mathbf{u}$ if the following holds. 

There are homeomorphisms $\varphi_k: \widehat{\mathbf{S}} \to \Sigma$ that are smooth in $S\setminus \Delta$, 
map $\Gamma^{\pm}$ to $\Gamma^{\pm}_k$, and satisfy 
$\varphi^*_k j_{\Sigma} \to j$ in $C^{\infty}_{\loc}(\widehat{\mathbf{S}}\setminus \Delta)$.
Moreover, for any $-\kappa^- \leq i \leq \kappa + \kappa^+ $ let $\widetilde{w}^i_k :=\widetilde{w}_k \circ \varphi_k|_{\dot{S}_i}$. Then, 
\begin{itemize}
\item there exists a sequence ${\textbf b}_k^i$ of real numbers  such that
$\mathbf{T}_{-{\textbf b}_k^i} \circ \widetilde{v}_k^i \to \overline{u}^i$ in $C^{\infty}_{\loc}(\dot{S}_i, \R \times Y)$, 
\item it holds that
\begin{itemize}
\item[] ${\textbf b}_k^i \to -\infty$ if $-\kappa^-\leq i\leq 0$,
\item[]  ${\textbf b}_k^i \to +\infty$ if $\kappa\leq i\leq \kappa^+$,
\item[]  ${\textbf b}_k^i = -R_k$ if $i=0$,  and  ${\textbf b}_k^i = R_k$ if $i=\kappa$, 
\end{itemize}
\item it holds that ${\textbf b}_k^{i+1} - {\textbf b}_k^i \to + \infty$ for all $i< \kappa^{+}$.
\end{itemize}
Finally $\widetilde{w}_k$ at $\Gamma_k^+$ and $\overline{u}^{\kappa^+}$ at $\Gamma^+$ detect the same Reeb orbit, and $\widetilde{w}_k$ at $\Gamma^-_k$ and $\overline{u}^{-\kappa^-}$ at $\Gamma^-$ detect the same Reeb orbit. 

Analogous to Theorem \ref{thm:SFT1} the following holds,  \cite{CPT}. 

\begin{thm}\label{thm:SFT2}
Assume $\alpha^{\pm}$, and $\alpha$ are non-degenerate. Let $(\Sigma,j_{\Sigma})$ be the Riemann sphere.  Then, for any sequence of holomorphic curves $\widetilde{w}_k: S\setminus \Gamma_k \to \R \times Y$ in $(W^{R_k}, \widehat{J}^{R_k})$ with $R_k \to \infty$ and with uniformly bounded energy, there is a  holomorphic building $\mathbf{u}$ of type $\kappa^+|1|\kappa|1|\kappa^-$ and arithmetic genus $0$ in the pair $[(W_{\rm top}, J_{\rm top}), (W_{\rm bot}, J_{\rm bot})]$ such that after restricting to a subsequence, $\widetilde{w}_k$ converges to $\mathbf{u}$ in the SFT-sense. Moreover, if $\# \Gamma_k^+=\# \Gamma_k^-=1$ for all $k\in \N$, then $\mathbf{u}$ is a cylinder with branches.  
\end{thm}

\color{black}

\section{The stability result and the proof of Theorem \ref{thm:main}} \label{sec:robustproof}
In this section, we will recall some relevant definitions concerning links of closed Reeb orbits and define (two notions) of exponential homotopical growth rate in the complement of such a link.  
We will then formulate our main stability result. It states that under some conditions on such a link $\mathcal{L}_0$, for any sufficiently  $C^0$-small pertubation of the contact form, a new link of Reeb orbits can be found (which might have more, or less components than $\mathcal{L}_0$) with certain lower bounds on the exponential homotopical growth rate in its complement. 
That stability result then yields, together with recent results 
 from \cite{Hryniewicz_generic,Contreras_generic} and \cite{Meiwes_horseshoes}, our main theorem,  Theorem~\ref{thm:main}, as a consequence.

\subsection{Contact forms that are hypertight in the complement of a link and the stability result}
 
We will start with some definitions. In the following, let $(Y,\xi)$ be a closed co-oriented contact $3$-manifold. 
\begin{defn}
A link $\mathcal{L}\subset Y$ is a closed embedded oriented $1$-dimensional manifold in $Y$. 
A link $\mathcal{L}$ is called a \textit{transverse link} in $(Y,\xi)$ if every component is transverse to $\xi$. 
\end{defn}

If $\Upsilon\subset \mathcal{P}(\alpha_0)$ is a finite collection of periodic orbits of a Reeb flow $\phi_{\alpha_0}$ and $\mathcal{L}$ is the link whose components are the closed curves defined by $\Upsilon$, we write $\mathcal{L} = \mathcal{L}(\Upsilon)$. 
In this case,  every component of $\mathcal{L}$ corresponds to a simple periodic Reeb orbit $\gamma \in \mathcal{P}(\alpha)$  (taken into account the natural identifications between distinct orbits explained in Section~\ref{subsec:HolomorphicCurves}); we say that $\mathcal{L}$ is a collection of (closed) Reeb orbits of~$\alpha_0$.

\begin{defn}
Let $\mathcal{L}$ be a link in $Y$. Let $\Pi$ be a surface with boundary. We say that a continuous map $h:\Pi\to Y$ \textit{has an interior intersection with $\mathcal{L}$} if there is an interior point $x\in\Pi$ with $h(x)\in \mathcal{L}$. We call  a continuous disk map $\mathfrak{h}:\overline{\D} \to Y$ with $\gamma(t) = \mathfrak{h}(e^{2\pi i t})$ a \textit{disk filling} of $\gamma$. We say that a loop $\beta:S^1 \to Y$ is \textit{contractible in the complement of $\mathcal{L}$} 
if there is a disk filling of $\beta$ with no interior intersection with $\mathcal{L}$.   Moreover, we say that two loops  $\beta_1:S^1 \to Y$ and $\beta_2:S^1 \to Y$ 
are \textit{homotopic in the complement of $\mathcal{L}$}
if there is a homotopy between $\beta_1$ and $\beta_2$ that does not have an interior intersection with $\mathcal{L}$. 
\end{defn}
\begin{defn}\label{defn:properclass}
We say that a free homotopy class $\rho$ of closed loops in $Y\setminus \mathcal{L}$ is a \textit{proper link class} if no loop $\gamma:S^1 \to Y$ that is a multiple of a component of $\mathcal{L}$ is homotopic in the complement of $\mathcal{L}$ to a loop $\beta:S^1 \to Y$ with $[\beta] = \rho$.
\end{defn}
\begin{defn}\label{defn:hypertight}
Let $\alpha_0 \in \mathcal{R}(Y,\xi)$ be a contact form, and $\mathcal{L}_0$ a transverse link. We say that $\alpha_0$ is \textit{hypertight in the complement of $\mathcal{L}_0$} if $\mathcal{L}_0$ is a collection of Reeb orbits of $\alpha_0$   and if there exists no periodic Reeb orbit of $\alpha_0$ (possibly a multiple of a component of $\mathcal{L}_0$) that is (as a loop) contractible in the complement of  $\mathcal{L}_0$. 
\end{defn}
\begin{defn} \label{def:homotopicalgrowth}
Let $\alpha_0\in \mathcal{R}(Y,\xi)$ be a contact form on $(Y,\xi)$, and let $\mathcal{L}_0$ be a link of Reeb orbits of $\alpha_0$.
Let $\Lambda^T_{\alpha_0}(\mathcal{L}_0)$ be the set whose elements are free homotopy classes $\rho$ of loops in $Y \setminus \mathcal{L}_0$ with the following properties
\begin{itemize}
    \item $\rho$ is a proper link class, 
    \item all Reeb orbits of $\alpha_0$ representing $\rho$ are non-degenerate, 
        \item $\min_{\gamma} \mathcal{A}_{\alpha_0}(\gamma)<T$, where the minimum is taken over all Reeb orbits $\gamma$ that represent $\rho$. 
\end{itemize}
Moreover, let $\Omega^T_{\alpha_0}(\mathcal{L}_0)$ be the set of those $\rho \in \Lambda^T_{\alpha_0}(\mathcal{L}_0)$ such that 
\begin{itemize}
\item there is only one Reeb orbit $\gamma_{\rho}$ of $\alpha_0$ that represents $\rho$.
\end{itemize}
 We let  $\Omega_{\alpha_0}(\mathcal{L}_0) := \bigcup_{T>0}\Omega^T_{\alpha_0}(\mathcal{L}_0)$.
\end{defn}
If $\alpha_0$ is hypertight in the complement of $\mathcal{L}_0$ and $\rho$ is a proper link class, then the cylindrical contact homology $\CH_{\rho,\mathcal{L}_0}(\alpha_0)$ in the complement of $\mathcal{L}_0$, restricted to $\rho$, which is due to Momin \cite{Momin} (see also \cite{AlvesPirnapasov,HMS}), is well defined. Moreover it is 
 clear from its definition that for $\rho \in \Omega^T_{\alpha_0}$, $\CH^T_{\rho,\mathcal{L}_0}(\alpha_0) \neq 0$.

\begin{defn}\label{defn:sys_h}
Let $\alpha_0$ be a contact form, and let 
 $\mathcal{L}_0$ be a link of Reeb orbits of $\alpha_0.$  
 Given an  almost complex structure $\widetilde{J}_0 \in \mathcal{J}(\alpha_0)$, we denote by $\mathfrak{h}(\alpha_0, \mathcal{L}_0, \widetilde{J}_0)$ the minimal $d\alpha_0$-energy, $\int_u d\alpha_0$, of a finite energy holomorphic cylinder $u$ in  $(\R \times Y, \widetilde{J}_0)$ that is not a trivial cylinder 
and that is positively or negatively asymptotic to a component of $\mathcal{L}_0$. We set $$\mathfrak{h}(\alpha_0,\mathcal{L}_0) = \sup\{\mathfrak{h}(\alpha_0, \mathcal{L}_0, \widetilde{J}_0)\, |\, \widetilde{J}_0\in \mathcal{J}(\alpha_0) \text{ regular}\}.$$ \color{black} 
We denote by $\mathrm{sys}(\alpha_0)$ the period of the shortest Reeb orbit of $\alpha_0$.
\end{defn}
The following is the main technical result of this paper.   
\begin{thm} \label{thm:stability}
Let $\alpha_0$ be a non-degenerate contact form on $(Y,\xi)$ and $\mathcal{L}_0$ a link of  Reeb orbits of $\alpha_0$ such that $\alpha_0$ is hypertight in the complement of $\mathcal{L}_0$. For $T >0$, let $\Omega^T_{\alpha_0}(\mathcal{L}_0)$ be the set defined in Definition~\ref{def:homotopicalgrowth}, and assume that
\begin{equation*}
    \Gamma:=\limsup_{T \to +\infty } \frac{\log\left(\# \Omega^T_{\alpha_0}(\mathcal{L}_0)\right)}{T} >0.
\end{equation*}

Then for every $\epsilon>0$ there exists $\delta>0$,  depending only on $\epsilon$, $\mathrm{sys}(\alpha_0)$, and $\mathfrak{h}(\alpha_0, \mathcal{L}_0)$,
such that for every non-degenerate contact form $\alpha$ with $d_{C^0}(\alpha, \alpha_0) < \delta$,  there is a link $\mathcal{L}(\alpha)$ of  Reeb orbits of $\alpha$ satisfying: 
\begin{equation*}
    \limsup_{T \to +\infty } \frac{\log\left(\# \Lambda^T_{\alpha}(\mathcal{L}({\alpha}))\right)}{T} > \Gamma -\epsilon.
\end{equation*}

Moreover, 
\begin{equation*}
    h_{\rm top}(\phi_\alpha) > \Gamma -\epsilon.
\end{equation*}
\end{thm}

\subsection{Approximation results and consequences}\label{sec:approximation_results}
Let us explain how our main theorem follows from Theorem~\ref{thm:stability}. 
First, we recall the following approximation result by the third author.  
\begin{thm}\label{thm:CH_recover}\cite[Thm.~3 and Rmk.~1.2]{Meiwes_horseshoes}
Let $(Y,\xi)$ be a closed contact $3$-manifold, $\alpha_0\in \mathcal{R}(Y,\xi)$. Assume that $h_{\topo}(\phi_{\alpha_0})>0$. For any $0< \epsilon< h_{\topo}(\phi_{\alpha_0})$,  there exists a link $\mathcal{L}^{\alpha_0}_\epsilon$ defined by hyperbolic periodic orbits of $\phi_{\alpha_0}$  such that 
\begin{equation*}
\limsup_{T \to +\infty } \frac{\log\left(\# \Omega^T_{\alpha_0}(\mathcal{L}_{\epsilon}^{\alpha_0})\right)}{T} > h_{\topo}(\phi_{\alpha_0})-\epsilon.
\end{equation*}
\end{thm}

Our main theorem can now be obtained from Theorem 
 \ref{thm:stability}, Theorem \ref{thm:CH_recover}, and some recent results on the existence of global surfaces of section for generic contact forms in dimension $3$, \cite{Hryniewicz_generic, Contreras_generic}.
\begin{proof}[Proof of Theorem \ref{thm:main}]  
By the work of \cite{Hryniewicz_generic} and \cite{Contreras_generic}, there is a $C^{\infty}$-open and dense set\footnote{In fact there is a $C^{2}$-open and $C^{\infty}$-dense set with these properties.} $\mathcal{V}_0(Y,\xi)$ in $\mathcal{R}(Y,\xi)$ such that the Reeb flow $\phi_{\alpha_0}$ of any $\alpha_0 \in \mathcal{V}_0(Y,\xi)$ admits a global surface of section with a finite collection of  non-degenerate periodic orbits~$\Upsilon_b$ forming its binding; we refer the reader to the above references for the precise definitions of these terms. 
We will show that $h_{\topo}$ is lower semi-continuous on $\mathcal{V}_0(Y,\xi)$  with respect to $d_{C^0}$.
To that end, let 
 $\alpha_0$ be any contact form in $\mathcal{V}_0(Y,\xi)$ with $h_{\topo}(\alpha_0)>0$, and fix any $0<\epsilon<h_{\topo}(\alpha_0)/2$. 
By Theorem \ref{thm:CH_recover}, there is a link 
$\mathcal{L}^{\alpha_0}_{\epsilon}=\mathcal{L}^{\alpha_0}_{\epsilon}(\Upsilon_{\epsilon})$ formed by a collection of periodic  orbits $\Upsilon_{\epsilon}$ of $\phi_{\alpha_0}$ such that 
$$\limsup_{T\to \infty} \frac{\log\left(\# \Omega^T_{\alpha_0}(\mathcal{L}^{\alpha_0}_{\epsilon})\right)}{T} \geq h_{\topo}(\phi_{\alpha_0})-\epsilon.$$
Choose a locally maximal hyperbolic set $K_{\epsilon}$ that contains $\mathcal{L}^{\alpha_0}_{\epsilon}$, which means that  there exists a compact neighbourhood $U_{\epsilon}$ of $K_{\epsilon}$ such that any orbit of $\phi_{\alpha}$ that is contained in $U_{\epsilon}$ is already contained in $K_{\epsilon}$; let us fix such a neighbourhood~$U_{\epsilon}$.  
We put $\Upsilon_0 := \Upsilon_{b} \cup \Upsilon_{\epsilon}$  and $\mathcal{L}_0:= \mathcal{L}_b(\Upsilon_{b}) \cup \mathcal{L}^{\alpha_0}_{\epsilon}$. It holds that $\alpha_0$ is hypertight in the complement of $\mathcal{L}_0$: in the case that $\mathcal{L}_b= \mathcal{L}_b(\{\gamma\})$ bounds a disk like global surface of section, all periodic orbits except the single binding orbit $\gamma$ is non-contractible in the complement of $\mathcal{L}_{b}\subset \mathcal{L}_0$, and also 
 $\gamma$ is non-contractible in the complement of $\mathcal{L}^{\alpha_0}_\epsilon\subset \mathcal{L}_0$; in all other cases, no periodic Reeb orbit is contractible in the complement of $\mathcal{L}_b \subset \mathcal{L}_0$, cf.\ \cite[Lemma 6.6]{Momin}.
By well-known arguments we can find a sequence $\alpha_k$ of non-degenerate contact forms such that
\begin{itemize}
\item $(\alpha_k)_{k\in \N}$ converges to $\alpha$ in the $C^{\infty}$-topology,
\item $\alpha_k|_{U_{\epsilon}} = \alpha_0|_{U_\epsilon}$,
\item $\alpha_k$ coincides with $\alpha$ in a sufficiently small neighbourhood $U_b$ of $\mathcal{L}_b$, 
\item $\alpha_k$ is hypertight in the complement of $\mathcal{L}_0$, 
\item $\mathrm{sys}(\alpha_k) = \mathrm{sys}(\alpha_0)$.
\end{itemize}
We then still have, for all $k\in \N$, 
$$\limsup_{T\to \infty} \frac{\log\left(\# \Omega^T_{\alpha_k}(\mathcal{L}_{0})\right)}{T} \geq h_{\topo}(\phi_{\alpha_0})-\epsilon.$$

Choose  almost complex structures on symplectizations  $\widetilde{J}_0$ in $\mathcal{J}(\alpha_0)$ and  $\widetilde{J}_k$ in
$\mathcal{J}(\alpha_k)$ such that for all $k\in \N$, $\widetilde{J}_k$ coincides with $\widetilde{J}_0$ on $\R \times  (U_\epsilon \cup U_b)$. By well-known  arguments, see e.g.\ the proof of Theorem~3.3 in \cite{HoferSalamon}, one can show that the minimal $d\alpha_k$-energy of finite energy holomorphic cylinders in  $(\R \times Y, J_k)$ that are not trivial cylinders and that are positively or negatively asymptotic to a component in $\mathcal{L}_0$, is uniformly (in $k$) bounded from below.
It follows then from Theorem~\ref{thm:stability} that we can choose $\delta>0$ such that 
for every $k\in \N$ and for all non-degenerate contact forms  $\alpha$ with $d_{C^0}(\alpha, \alpha_k) < \delta$, 
\begin{align}\label{ineq:alpha_alpha0}
h_{\topo}(\phi_{\alpha}) > h_{\topo}(\phi_{\alpha_k}) - \epsilon \geq  h_{\topo}(\phi_{\alpha_0}) - 2\epsilon.
\end{align}
Since $h_{\topo}$ is $C^{\infty}$-upper semi-continuous, \cite{N89}, and  $\mathcal{V}_0(Y,\xi)$ is $C^{\infty}$-dense in $\mathcal{R}(Y,\xi)$, we have that \eqref{ineq:alpha_alpha0} holds for all $\alpha\in \mathcal{R}(Y,\xi)$ with $d_{C^0}(\alpha,\alpha_0)<\delta$. 
\end{proof}
\begin{rem}\label{rem:corollaries}
Given the proof of Theorem \ref{thm:main} above,  it is straightforward to deduce Corollary~\ref{cor:righthanded} from the introduction. In the situation of a right-handed Reeb flow,  the link $\mathcal{L}_0$ of closed Reeb orbits that is considered in the proof can be chosen to be identical to the link $\mathcal{L}_{\epsilon}^{\alpha_0}$ of hyperbolic periodic orbits obtained from the approximation theorem, Theorem~\ref{thm:CH_recover}: this follows from Ghys' result in \cite{Ghys} that any closed orbit is the binding of a global surfaces of section, which was mentioned in the introduction.
\end{rem}
\color{black}
\section{Proof of the stability result (Theorem \ref{thm:stability})} \label{sec:proofmainresult}

Throughout this section, we fix a closed contact manifold $(Y,\xi)$ of dimension $3$, a non-degenerate contact form $\alpha_0$ for $(Y,\xi)$, and a transverse link $\mathcal{L}_0$ such that $\alpha_0$ is hypertight in the complement of $\mathcal{L}_0$.

\subsection{Symplectic cobordisms and homotopies of symplectic cobordisms} \label{sec:sympcobordisms}
Let $\alpha$ be a contact form on $(Y,\xi)$. 
We write the contact form $\alpha$ of $(Y,\xi)$ as $\alpha = f_\alpha \alpha_0$, and choose any real numbers $C^+
> \max(f_\alpha)$, $C^- < \min(f_{\alpha})$.  We will construct special homotopies of exact symplectic cobordisms from $\mathrm{C^+} \alpha_0$ to $\mathrm{C^-}\alpha_0$.

We start by fixing once and for all a regular cylindrical almost complex structure $\widetilde{J}_0$ in $\mathcal{J}(\alpha_0)$. Recall that it is regular in the sense that the linearized Cauchy-Riemann operator associated to it is surjective on any simply covered finite energy holomorphic curve in $(\mathbb{R}\times Y, \widetilde{J}_0)$. We also fix a cylindrical almost complex structure $J_\alpha$ in $\mathcal{J}(\alpha)$.

\subsubsection*{A dilation cobordism} Let $0<\mu<<<1$. Given any $\mathrm{C^+}$ and $\mathrm{C^-}$ as above and for any $R>0$, we first construct a diffeomorphism
\begin{equation*}
    F_R: \R \to \R
\end{equation*}
satisfying 
\begin{itemize}
    \item $F'_R(a) =  \frac{1}{C^+}$ for $a \geq R+1 - \mu$,
    \item $F'_R(a) =  \frac{1}{C^-}$ for $a \leq -R-1 + \mu$.
\end{itemize}

If we define $\Upsilon_R: \R \times Y \to \R \times Y$ by 
$$\Upsilon_R(a,q) :=  (F_R(a),q),  $$
and we define 
\begin{equation}
    J^R_0:= (\Upsilon_R)^* \widetilde{J}_0,
\end{equation}
then the almost complex structure $J^R_0$ satisfies: 
\begin{itemize}
    \item $J^R_0$ coincides with an almost complex structure $J^+_0 \in \mathcal{J}(C^+ \alpha_0)$ on \linebreak $[R+1- \mu, +\infty) \times Y$,
    \item $J^R_0$ coincides with an almost complex structure $J^-_0 \in \mathcal{J}(C^- \alpha_0)$ on \linebreak $(-\infty, -R-1 + \mu] \times Y$.
\end{itemize}
Note that $J^{\pm}_0 = {(\Upsilon^{\pm})}^*\widetilde{J}_0$, 
where $\Upsilon^{\pm}:\R \times Y \to \R \times Y$ is defined by 
\begin{equation}\label{eq:Upsilon}
\Upsilon^{\pm}(a,q) := (\frac{1}{C^{\pm}}a,q).
\end{equation}
Let now $h^R_0: \R \times Y$ be a function such that
\begin{itemize}
    \item $h^R_0 \equiv C^+$ on $[R+1,+\infty) \times Y$,
    \item $h^R_0 \equiv C^-$ on $(-\infty,-R-1] \times Y$,
    \item $\partial_a h^R_0\geq 0$,
    \item $\partial_a h^R_0 > 0$ on $(-R-1,R+1) \times Y$.
\end{itemize}
We define 
\begin{equation}
    \lambda^R_0:= h^R_0\alpha_0,
\end{equation}
and 
\begin{equation}
    \omega^R_0:= d\lambda^R_0.
\end{equation}

Then $(\R \times Y, \lambda^R_0, \omega^R_0)$ is an exact symplectic cobordism from $C^+\alpha_0$ to $C^-\alpha_0$, and the almost complex structure $J^R_0$ is compatible with $\omega^R_0$ on the region \linebreak $(-R-1,R+1) \times Y$, where this form is symplectic and preserves $\omega^R_0$ everywhere.

\subsubsection*{Stretching exact symplectic cobordisms} We now consider exact symplectic cobordisms $W_R= (\R \times Y, \lambda^R_1,\omega^R_1)$ of the form considered in Section \ref{sec:cobordisms} and compatible almost complex structures $J^R_1$ on these cobordisms.
For this, let
$$h^R_1: \R \times Y \to \R $$
be a function satisfying
\begin{itemize}
    \item $h^R_1 \equiv C^+ $ on $[R+1, +\infty) \times Y$,
    \item $h^R_1 \equiv C^- $ on $(-\infty, -R-1] \times Y$,
    \item $h^R_1 \geq 0$,
    \item $\partial_a h^R_1 >0$ on $(-R-1,R+1) \times Y$,
    \item $h^R_1$ depends only on the real coordinate and coincides with $h^R_0$ on $[R+1-\mu,R+1] \times Y$ and on $[-R-1,-R-1+\mu]\times Y$,
    \item $h^R_1 = e^{ \frac{a}{\mathrm{K}}}f_\alpha$ on $[-R+2\mu,R-2\mu] \times Y$, where $\mathrm{K}= \max \Big\{ \frac{R}{\frac{\log(C^+)}{\max (f_\alpha)}},\frac{R}{\frac{-\log(C^-)}{\max (f_\alpha)}} \Big\},$
    \item $ h^R_1 \equiv \mathrm{h}^+_R f_\alpha$ on $[R,R+\mu] \times Y$, where $\mathrm{h}^+_R:[R,R+\mu] \to R$ is a strictly increasing smooth function,
    \item $ h^R_1 \equiv \mathrm{h}^-_R(a+R) f_\alpha$ on $[-R-\mu,-R] \times Y$, where $\mathrm{h}^-_R:[-R-\mu,-R] \to \R$ is a strictly increasing smooth function.
\end{itemize}

We assume that for any $R >0 $ and any $c\geq 0$ we have 
\begin{equation} \label{eq:hconvergence1}
    h^{R+c}_1(a+c) = h^R_1(a) \mbox{ for all } a \in [R-\mu,+\infty) \times Y,
\end{equation}
and 
\begin{equation} \label{eq:hconvergence2}
h^{R+c}_1(a-c) = h^R_1(a) \mbox{ for all } a \in (-\infty, -R +\mu ] \times Y.  
\end{equation}
We then set 
\begin{equation}
    \lambda^R_1 :=  h^R_1 \alpha_0
\end{equation}
and 
\begin{equation}
    \omega^R_1 := d(h^R_1 \alpha_0).
\end{equation}

Recall that for any real number $\mathbf{b}$ we have  the translation map $\mathbf{T_b}: \R \times Y \to \R \times Y $, 
\begin{equation}
    \mathbf{T_b}(a,p) := (a+\mathbf{b},p).
\end{equation}

By equations \eqref{eq:hconvergence1} and \eqref{eq:hconvergence2} it follows that $    (\mathbf{T}_R)^*\lambda^R_1  $ converges as $R \to +\infty$ in $C^\infty_{\rm loc}$ to a one-form $\lambda^+_1$ on $\R \times Y$. Similarly,  $    (\mathbf{T}_R)^*\omega^R_1  $ converges as $R \to +\infty$ in $C^\infty_{\rm loc}$ to a two-form $\omega^+_1$ on $\R \times Y$, and $d(\lambda^+_1)=\omega^+_1$. The triple $(\R \times Y,\lambda^+_1,\omega^+_1)$ is an exact symplectic cobordism from $C^+ \alpha_0$ to $\alpha$. 

Again by equations \eqref{eq:hconvergence1} and \eqref{eq:hconvergence2} it follows that $    (\mathbf{T}_{-R})^*\lambda^R_1  $ converges as $R \to +\infty$ in $C^\infty_{\rm loc}$ to a one-form $\lambda^-_1$ on $\R \times Y$. Similarly,  $    (\mathbf{T}_{-R})^*\omega^R_1  $ converges as $R \to +\infty$ in $C^\infty_{\rm loc}$ to a two-form $\omega^-_1$ on $\R \times Y$, and $d(\lambda^-_1)=\omega^-_1$. The triple $(\R \times Y,\lambda^-_1,\omega^-_1)$ is an exact symplectic cobordism from $\alpha$ to $C^- \alpha_0$. 

We will now consider an almost complex structure $J^R_1$ on $\R \times Y$ that is compatible with $\omega^R_1$ on $(-R-1, R+1) \times Y$, that leaves it invariant everywhere, and satisfies:
\begin{itemize}
    \item $J^R_1 \equiv J^-_0$ on $(-\infty,-R-1+\mu] \times Y$,
    \item $J^R_1 \equiv J^+_0$ on $[R+1-\mu,+\infty) \times Y$,
    \item $J^R_1 \equiv J_\alpha$ on $[-R,R] \times Y$ (recall that $J_\alpha$ was fixed earlier and is independent of $R$),
    \item $J^R_1$ is compatible with $\omega^R_1$ in $(-R-1,R+1)\times Y$ and preserves it everywhere, 
    \item $(\mathbf{T}_{R})^*J^R_1$ converges in $C^\infty_{\rm loc}$ as $R \to +\infty$ to an almost complex structure $J^{+\infty}_1$ on $\R \times Y$ compatible with the cobordism $(\R \times Y,\lambda^+_1,\omega^+_1)$,
    \item $(\mathbf{T}_{-R})^*J^R_1$ converges in $C^\infty_{\rm loc}$ as $R \to +\infty$ to an almost complex structure $J^{-\infty}_1$ on $\R \times Y$ compatible with the cobordism $(\R \times Y,\lambda^-_1,\omega^-_1)$,
    \item $J^R_1$ is regular in the sense that the linearized Cauchy-Riemann operator of $J^R_1$ over each simply covered finite energy holomorphic curve in $(W_R,J^R_1)$ is surjective.    
\end{itemize}
The last condition is $C^\infty$-generic (see~\cite{Wendl-lectures}). It is immediate that the almost complex structure $J^{+\infty}_1$ is asymptotically cylindrical to $J^+_0$ at $+\infty$ and to $J_\alpha$ at $-\infty$. Similarly, $J^{-\infty}_1$ is asymptotically cylindrical to $J^-_0$ at $-\infty$ and to $J_\alpha$ at $+\infty$.

The energy of a holomorphic curve in a splitting exact symplectic cobordism \linebreak $(\R \times Y,\lambda^R_1,\omega^R_1,J^R_1)$ is defined in Section \ref{subsec:HolomorphicCurves}. We can choose $J^R_1$ to be regular in the sense that moduli spaces of simply covered holomorphic curves in \linebreak $(\R \times Y,\lambda^R_1,\omega^R_1,J^R_1)$ are transversely cut out.

Let $(h^R_s)_{s \in [0,1]}$ be a smooth homotopy of functions 
\begin{equation}
    h^R_s: \R \times Y \to \R,
\end{equation}
starting at the function $h^R_0$ and ending at the function $h^R_1$ defined above. We require that, for every $s \in [0,1]$, $h^R_s$ satisfies:
\begin{itemize}
    \item $\partial_a h^R_s \geq 0$,
    \item $\partial_a h^R_s >0$ on $(-R-1,R+1) \times Y$,
    \item $h^R_s$ coincides with $h^R_0$ and $h^R_1$ on $[R+1-\mu,+\infty) \times Y$ and $(-\infty, -R-1+\mu] \times Y$.
    \end{itemize}
This implies in particular that $h^R_s \equiv C^+$ on $[R+1,+\infty) \times Y$ and $h^R_s \equiv C^-$ on $(-\infty,-R-1] \times Y$. It is easy to construct homotopies $(h^R_s)_{s \in [0,1]}$  satisfying these assumptions.

We define
\begin{equation}
    \lambda^R_s:= h^R_s \alpha_0
\end{equation}
and 
\begin{equation}
    \omega^R_s:= d(\lambda^R_s).
\end{equation}
Then, $(\R \times Y, \lambda^R_s,\omega^R_s )_{s \in [0,1]}$ is a smooth homotopy of exact symplectic cobordisms from $C^+ \alpha_0$ to $C^- \alpha_0$ starting at $(\R \times Y, \lambda^R_0,\omega^R_0 )$ and ending at $(\R \times Y, \lambda^R_1,\omega^R_1 )$.

We consider then smooth homotopies $(J^R_s)_{s \in [0,1]}$, where each $J^R_s$ will be an asymptotically cylindrical almost complex structure on $(\R \times Y, \lambda^R_s,\omega^R_s )$.
We assume that $J^R_s$ satisfies for each $s \in [0,1]$:
\begin{itemize}
    \item $J^R_s \equiv J^+_0$ on $[R+1 -\mu, +\infty) \times Y$,
    \item $J^R_s \equiv J^+_0$ on $(-\infty,-R-1+\mu] \times Y$,
    \item $J^R_s$ is compatible with $\omega^R_s$ on $(-R-1,R+1) \times Y$ and preserves $\omega^R_s$ everywhere, 
    \item $J^R_s$ is a smooth homotopy of almost complex structures on $\R \times Y$ starting at $J^R_0$ and ending at $J^R_1$.
\end{itemize}

A smooth homotopy $(\R \times Y, \lambda^R_s,\omega^R_s,J^R_s )_{s \in [0,1]}$ of exact symplectic cobordisms endowed with almost complex structures as constructed in this section will be referred to from now on as an \textit{admissible homotopy of exact symplectic cobordisms endowed with almost complex structures}.

\subsection{Moduli spaces associated to components of the link $\mathcal{L}_0$}\label{sec:moduli}

\
\subsubsection{The moduli space} 
We start by introducing some notation. Given a simple  closed Reeb orbit $\gamma_0$ of $\alpha_0$ which is a component of $\mathcal{L}_0$ we denote by 
$$\mathcal{M}(\gamma_0,\gamma_0,\omega^R_s,J^R_s;s\in [0,1])=\{(s,\widetilde{u}_s)\}/\sim$$
the moduli space of pairs $(s,\widetilde{u}_s)$, where $s\in [0,1]$ and $\widetilde{u}_s$ is a finite energy holomorphic cylinder in $(\R \times Y,\lambda_s^R, \omega_s^R,J^R_s)$ with one positive and one negative puncture both asymptotic to $\gamma_0$. Two such cylinders $\widetilde{u}_s$ and $\widetilde{u}'_s$ are considered equivalent if they can be obtained one from another by holomorphic reparametrizations of their common domain $(\R \times S^1, i)$. Recall that holomorphic reparametrizations of a holomorphic cylinder are precisely combinations of translations  in the $\R$-factor  and rotations in the $S^1$-factor.

\subsubsection{Some recollections}
\begin{rec}[Bubbling of holomorphic planes] \label{rec:bubbling}
We recall some parts of the bubbling analysis of \cite{Hofer93,HoferWysockiZehnder} that we need. Let $(s_n,\widetilde{u}_{s_n})$ be a sequence of elements of $\mathcal{M}(\gamma_0,\gamma_0,\omega^R_s,J^R_s;s \in [0,1])$, where we assume that $s_n$ converges to an element $s\in [0,1]$ as $n \to +\infty$. If the gradient of the sequence $\widetilde{u}_{s_n}$ is not uniformly bounded, then we can find a subsequence $s'_n$ of $s_n$, a sequence of points $z_n \in \R\times S^1$, a sequence of real numbers $\mathbf{b}_n$, strictly monotone sequences of positive numbers $\epsilon_n \to 0$ and $K_n \to +\infty$, and a sequence of holomorphic charts\footnote{Here $B_{K_n}(0) \subset \mathbb{C}$ is the open ball of size $K_n$ around the origin for the flat metric on $\mathbb{C}$ and $B_{\epsilon_n}(z_n) \subset \R \times S^1$ is the open ball of size $\epsilon_n$ around $z_n$ for the flat metric on $\R \times S^1$. } $$\uppsi_n: B_{K_n}(0) \subset \mathbb{C} \to B_{\epsilon_n}(z_n) \subset \R \times S^1 $$ such that $\widetilde{v}_n:= \mathbf{T}_{-\mathbf{b}_n} \circ \widetilde{u}_{s'_n} \circ \uppsi_n$ converges in $C^\infty_{\rm loc}$ to $\widetilde{v}$, where $\widetilde{v}$ is either a non-constant finite energy plane in the symplectization $(\R \times Y, J_0^+)$ of $C^+ \alpha_0$, a non-constant finite energy plane in the exact symplectic cobordism  $(\R \times Y, \lambda^R_s,\omega^R_s, J^R_s)$, or a non-constant finite energy plane in the symplectization $(\R \times Y, J_0^-)$ of $C^- \alpha_0$.
\end{rec}

\begin{rec}[Breaking orbits] \label{rec:breaking}
Let $(s_n,\widetilde{u}_{s_n})$ be a sequence of elements of $\mathcal{M}(\gamma_0,\gamma_0,\omega^R_s,J^R_s;s \in [0,1])$, where we assume that $s_n$ converges to an element $s\in [0,1]$ as $n \to +\infty$. A closed Reeb orbit $\gamma^\pm$ of $C^\pm \alpha_0$ is a breaking orbit for the sequence $s_n$ if there exists a subsequence $s'_n$ of $s_n$, a family of embedded circles $C_n \subset \R \times S^1$, a family of annuli $\mathfrak{A}_n \subset \R \times S^1$ with $C_n \subset \mathfrak{A}_n$, a strictly monotone sequence of positive real numbers $L_n \to +\infty$, and a sequence of biholomorphisms
\begin{align*}
   \upphi_n: [-L_n, L_n] \times S^1 \to \mathfrak{A}_n, \\
   \upphi_n(\{0\} \times S^1) = C_n,
\end{align*}
such that $\widetilde{v}_n := \widetilde{u}_{s'_n} \circ \upphi_n$ converges in $C^\infty_{\rm loc}$ to the trivial holomorphic cylinder over $\gamma^\pm$ in the symplectization $(\R \times Y, J_0^\pm)$. 

\end{rec}

\subsubsection{Compactness of the moduli space} \label{sec:compactness}

The next proposition will be one of the crucial ingredients in the proof of Theorem \ref{thm:stability}: if $C^{\pm}$ are  sufficiently close to $1$, then the moduli space $\mathcal{M}(\gamma_0,\gamma_0,\omega^R_s,J^R_s;s\in [0,1])$ is compact.

Let us first introduce some notation. In the following let us write 
$$\mathrm{Spec}({\alpha_0}) :=\{ \mathcal{A}_{\alpha_0}(\gamma) \ | \ \gamma \text{ periodic Reeb orbit of } \phi_{\alpha_0}\}.$$
We also consider the subset $\mathrm{Spec}_{\alpha_0}(\mathcal{L}_0)\subset \mathrm{Spec}({\alpha_0})$ defined by  
$\mathrm{Spec}_{\alpha_0}(\mathcal{L}_0):= \{ \mathcal{A}_{\alpha_0}(\gamma) \ | \ \gamma \text{ component of } \mathcal{L}_0\}, \text{ and we put}$ 
$\mathcal{A}_{\alpha_0}(\mathcal{L}_0):=\max\mathrm{Spec}_{\alpha_0}(\mathcal{L}_0).$

Recall that $\mathrm{sys}(\alpha_0)$ denotes the period of the shortest Reeb orbit of $\alpha_0$, and  
 that  $\mathfrak{h}(\alpha_0, \mathcal{L}_0, \widetilde{J}_0)$ is the infimum of the numbers $\int_u d\alpha_0 = \mathcal{A}_{\alpha_0}(\gamma^+)-\mathcal{A}_{\alpha_0}(\gamma^-)$, where $u$ is any finite energy holomorphic cylinder in the symplectization $(\R \times Y, \widetilde{J}_0)$ that is not a trivial cylinder and for which its positive asymptotic orbit $\gamma^+$ or its negatively asymptotic orbit $\gamma^-$ is a component of $\mathcal{L}_0$, cf.\ Definition \ref{defn:sys_h}. 

In the following let us choose $\mu_{\alpha_0,\mathcal{L}_0}>0$ such that 
\begin{equation}\label{eq:mu_sys}
\mu_{\alpha_0,\mathcal{L}_0} < \frac{\mathrm{sys}(\alpha_0)}{3} \quad \text{and }
\end{equation}
\begin{equation}\label{eq:munew}
\mu_{\alpha_0,\mathcal{L}_0}<\mathfrak{h}(\alpha_0, \mathcal{L}_0, \widetilde{J}_0). 
\end{equation}
Since $\alpha_0$ is non-degenerate, a number $\mu_{\alpha_0,\mathcal{L}_0}$ as above exists.

This holds in particular if for every $c \in \mathrm{Spec}_{\alpha_0}(\mathcal{L}_0)$,   
\begin{equation} \label{eq:defdelta}
(c-\mu_{\alpha_0,\mathcal{L}_0},c+\mu_{\alpha_0,\mathcal{L}_0}) \cap \mathrm{Spec}(\alpha_0) = \{c\}.  \end{equation}

\begin{assump} \label{assump:delta}
Let $\delta_{\alpha_0,\mathcal{L}_0}>0$ be small enough so that the following conditions hold:

Condition A):  for every $c \in \mathrm{Spec}_{\alpha_0}(\mathcal{L}_0)$ we have that 
    \begin{equation*}
       e^{ -2\delta_{\alpha_0,\mathcal{L}_0}} c= \frac{e^{ -\delta_{\alpha_0,\mathcal{L}_0} }}{e^{ \delta_{\alpha_0,\mathcal{L}_0} }} c > c- \mu_{\alpha_0,\mathcal{L}_0},
    \end{equation*}
    and 
    \begin{equation*}
        e^{ 2\delta_{\alpha_0,\mathcal{L}_0}} c=\frac{e^{ \delta_{\alpha_0,\mathcal{L}_0} }}{e^{- \delta_{\alpha_0,\mathcal{L}_0} }} c < c + \mu_{\alpha_0,\mathcal{L}_0}.
    \end{equation*}

\ 

Condition B): for every $c \in \mathrm{Spec}_{\alpha_0}(\mathcal{L}_0)$ we have that $$(e^{ \delta_{\alpha_0,\mathcal{L}_0}}- e^{- \delta_{\alpha_0,\mathcal{L}_0} } ) c < e^{- \delta_{\alpha_0,\mathcal{L}_0} } \mathrm{sys}({\alpha_0}).$$
\end{assump}

\begin{prop} \label{prop:compactness}
Let $\alpha_0$, $(Y,\xi)$, and  $\mathcal{L}_0$ be as in Theorem \ref{thm:stability}. Let $\alpha$ be a contact form on $(Y,\xi)$ such that $d_{C^0}(\alpha,\alpha_0)<\delta_{\alpha_0,\mathcal{L}_0}$ and $C^+$ and $C^-$ be positive real numbers satisfying $\max (f_\alpha)< C^+ < e^{\delta_{\alpha_0,\mathcal{L}_0}}$ and $e^{-\delta_{\alpha_0,\mathcal{L}_0}} < C^- < \min (f_\alpha)$, respectively. Let $(\R \times Y, \lambda^R_s,\omega^R_s,J^R_s )_{s \in [0,1]}$ be an admissible homotopy of exact symplectic cobordisms endowed with almost complex structures as constructed in Section \ref{sec:sympcobordisms}, and $\gamma_0$ be a simple orbit contained in $\mathcal{L}_0$. 

Then, any sequence $(s_n,\widetilde{u}_{s_n})$ of elements of  $\mathcal{M}(\gamma_0,\gamma_0,\omega^R_s,J^R_s;s\in [0,1])$ has a subsequence which converges in the SFT-sense to an  element of the moduli space  $\mathcal{M}(\gamma_0,\gamma_0,\omega^R_s,J^R_s;s\in [0,1])$. Denoting by $\overline{\mathcal{M}}(\gamma_0,\gamma_0,\omega^R_s,J^R_s;s\in [0,1])$ the SFT-compactification of ${\mathcal{M}}(\gamma_0,\gamma_0,\omega^R_s,J^R_s;s\in [0,1])$, it follows that $$\overline{\mathcal{M}}(\gamma_0,\gamma_0,\omega^R_s,J^R_s;s\in [0,1])=\mathcal{M}(\gamma_0,\gamma_0,\omega^R_s,J^R_s;s\in [0,1]),$$
 i.e. $\mathcal{M}(\gamma_0,\gamma_0,\omega^R_s,J^R_s;s\in [0,1])$ is compact.
\end{prop}

\textit{Proof.} 
The proof is divided into steps. Clearly, there is no loss of generality in assuming that $s_n$ converges to $s \in [0,1]$, so we make this assumption from now on.

\textit{Step 1: Ruling out bubbling.}

Let $(s_n,u_{s_n})$ be a sequence of elements of $\mathcal{M}(\gamma_0,\gamma_0,\omega^R_s,J^R_s;s\in [0,1])$ such that $s_n \to s \in [0,1]$ as $n\to +\infty$. We assume that the gradients of $u_{s_n}$ are not uniformly bounded and search a contradiction. This will then rule out bubbling for the sequence $u_{s_n}$.  

By Recollection \ref{rec:bubbling}, we can find a subsequence $s'_n$ of $s_n$, a sequence of real numbers $\mathbf{b}_n$, a sequence of points $z_n \in \R\times S^1$, strictly monotone sequences of positive numbers $\epsilon_n \to 0$ and $K_n \to +\infty$, and a sequence of holomorphic charts  $$\uppsi: B_{K_n}(0) \subset \mathbb{C} \to B_{\epsilon_n}(z_n) \subset \R \times S^1 $$ such that $\widetilde{v}_n:= \mathbf{T}_{-\mathbf{b}_n} \circ \widetilde{u}_{s'_n} \circ \uppsi_n$ converges in $C^\infty_{\rm loc}$ to a finite energy plane $\widetilde{v}$. 
There are three possibilities for the finite energy plane $\widetilde{v}$:
\begin{itemize}
    \item[1)] $\widetilde{v}$ is a finite energy plane in the symplectization $(\R \times Y, J_0^+)$ of $C^+ \alpha_0$,
    \item[2)] $\widetilde{v}$ is a finite energy plane in the exact symplectic cobordism \linebreak $(\R \times Y, \lambda^R_s,\omega^R_s, J^R_s)$,
    \item[3)] $\widetilde{v}$ is a finite energy plane in the symplectization $(\R \times Y, J_0^-)$ of $C^- \alpha_0$.
\end{itemize}
We show that each of these three possibilities leads to a contradiction. For this, we notice that by Condition B) on $\delta_{\alpha_0,\mathcal{L}_0}$ and the choices of $C^{\pm}$ we have
\begin{equation} \label{eq:actionbub0}
    (C^+ - C^-)\mathcal{A}_{\alpha_0}(\mathcal{L}_0) < C^-\mathrm{sys}(\alpha_0).
\end{equation}

\textit{Ruling out 1).} Assume that 1) holds. Then $\widetilde{v}: \mathbb{C} \to (\R \times Y, J_0^+)$ is asymptotic to a  non-degenerate closed Reeb orbit $\gamma$ of $\alpha_0$. We then have:
\begin{equation} \label{eq:actionbub1}
\int_{\R \times S^1}\widetilde{v}^*d(C^+\alpha_0) = \mathcal{A}_{C^+\alpha_0}(\gamma) \geq C^+\mathrm{sys}(\alpha_0).
\end{equation}

At the same time 
\begin{align} \label{eq:actionbub2}
& \int_{\R \times S^1}\widetilde{v}^*d(C^+\alpha_0) = \lim_{n \to +\infty} \int_{B_{K_n}(0)}\widetilde{v}_n^*(\omega_{s'_n}^R) \leq \\ \nonumber & \leq \lim_{n\to +\infty} \int_{\R \times S^1}\widetilde{v}_{n}^*(\omega_{s'_n}^R)  = (C^+ - C^-)\mathcal{A}_{\alpha_0}(\gamma_0) \leq (C^+ - C^-)\mathcal{A}_{\alpha_0}(\mathcal{L}_0).
\end{align}

Combining \eqref{eq:actionbub1} and \eqref{eq:actionbub2} we obtain $(C^+ - C^-)\mathcal{A}_{\alpha_0}(\mathcal{L}_0) \geq C^+ \mathrm{sys}(\alpha_0)$. But this contradicts \eqref{eq:actionbub0}, since $C^-\leq C^+$.

\textit{Ruling out 2).} Assume that 2) holds. Then $\widetilde{v}: \mathbb{C} \to (\R \times Y, J^R_s)$ is a finite energy plane in $(\R \times Y, \lambda^R_s,\omega^R_s, J^R_s)$, which is positively asymptotic to a Reeb orbit $\gamma$ of $C^+ \alpha_0$. 
We then have
\begin{equation} \label{eq:actionbub3}
\int_{\R \times S^1}\widetilde{v}^*\omega^R_s = \mathcal{A}_{C^+\alpha_0}(\gamma) \geq C^+\mathrm{sys}(\alpha_0).
\end{equation}  
and 
\begin{align} \label{eq:actionbub4}
& \int_{\R \times S^1}\widetilde{v}^*\omega^R_s = \lim_{n \to +\infty} \int_{B_{K_n}(0)}\widetilde{v}_n^*(\omega_{s'_n}^R) \leq \\ \nonumber & \leq \lim_{n\to + \infty}\int_{\R \times S^1}\widetilde{v}_n^*(\omega^R_{s'_n})  = (C^+ - C^-)\mathcal{A}_{\alpha_0}(\gamma_0) \leq (C^+ - C^-)\mathcal{A}_{\alpha_0}(\mathcal{L}_0).
\end{align}
Combining \eqref{eq:actionbub3} and \eqref{eq:actionbub4} we again obtain $(C^+ - C^-)\mathcal{A}_{\alpha_0}(\mathcal{L}_0) \geq C^+ \mathrm{sys}(\alpha_0)$, which contradicts \eqref{eq:actionbub0}.

\textit{Ruling out 3).} Assuming 3) holds, $\widetilde{v}: \mathbb{C} \to (\R \times Y, J_0^-)$ is a finite energy plane in the symplectization of $C^- \alpha_0$, positively asymptotic to a Reeb orbit $\gamma$ of $C^- \alpha_0$. 

We then have
\begin{equation} \label{eq:actionbub5}
\int_{\R \times S^1}\widetilde{v}^*d(C^-\alpha_0) = \mathcal{A}_{C^-\alpha_0}(\gamma) \geq C^-\mathrm{sys}(\alpha_0), 
\end{equation}  
and 
\begin{align} \label{eq:actionbub6}
& \int_{\R \times S^1}\widetilde{v}^*d(C^-\alpha_0) = \lim_{n \to +\infty} \int_{B_{K_n}(0)}\widetilde{v}_n^*(\omega_{s'_n}^R) \leq \\ \nonumber & \leq\lim_{n\to + \infty} \int_{\R \times S^1}\widetilde{v}_n^*(\omega_{s'_n}^R)  = (C^+ - C^-)\mathcal{A}_{\alpha_0}(\gamma_0) \leq (C^+ - C^-)\mathcal{A}_{\alpha_0}(\mathcal{L}_0).
\end{align}

Combining \eqref{eq:actionbub5} and \eqref{eq:actionbub6}, we again obtain $(C^+ - C^-)\mathcal{A}_{\alpha_0}(\mathcal{L}_0) \geq C^- \mathrm{sys}(\alpha_0)$, which contradicts \eqref{eq:actionbub0}.

\

\textit{Step 2: Ruling out breaking.}

Let $(s_n,\widetilde{u}_{s_n})$ be a sequence of elements of $\mathcal{M}(\gamma_0,\gamma_0,\omega^R_s,J^R_s;s\in [0,1])$. 
The SFT-compactness Theorem of \cite{CPT}, see Section \ref{sec:SFT1}, then implies that there exists a subsequence $(s'_n,\widetilde{u}_{s'_n})$ of $(s_n,\widetilde{u}_{s_n})$ such that $s'_n \to s \in [0,1] $ and $\widetilde{u}_{s'_n}$ converges to a finite energy holomorphic building $\overline{u}$ in  $(\R \times Y,\lambda^R_s, \omega^R_s,J^R_s)$. 
By Step 1, we know that the gradients of the holomorphic curves $\widetilde{u}_{s'_n}$ are uniformly bounded, and the results of \cite{CPT,Hofer93} imply that every level of the building $\overline{u}$ is formed by exactly one non-trivial finite energy holomorphic cylinder. We thus have that $\overline{u}$ is a broken holomorphic cylinder and has the following structure:
\begin{itemize}
    \item there are non-negative integers $\kappa^+$ and $\kappa^-$ such that \linebreak $\overline{u}= (\overline{u}^{\kappa^+},\overline{u}^{\kappa^+-1},\ldots,\overline{u}^1,\overline{u}^0,\overline{u}^{-1},\ldots,\overline{u}^{-\kappa^-} )$,
    \item for each $0<j \leq \kappa^+$, $\overline{u}^j$ is a finite energy holomorphic cylinder in the symplectization $(\R \times Y,J^+_0)$ of $C^+\alpha_0$ that is not a trivial cylinder and that has one positive and one negative puncture,
    \item $\overline{u}^0$ is finite energy holomorphic cylinder in the exact symplectic cobordism $(\R \times Y, \lambda^R_s, \omega^R_s,J^R_s)$ with one positive and one negative puncture,
    \item for each $0>j \geq -\kappa^-$, $\overline{u}^j$ is a finite energy holomorphic cylinder in the symplectization $(\R \times Y,J^-_0)$ of $C^-\alpha_0$ that is not a trivial cylinder and that has one positive and one negative puncture,
    \item for each $ -\kappa^- < j \leq \kappa^+ $, the Reeb orbit $\gamma^{j-1}_+$ detected by the positive puncture of $\widetilde{u}^{j-1}$ coincides with the Reeb orbit $\gamma^j_-$ detected  by the negative puncture of $\overline{u}^j$,
    \item the positive puncture of $\overline{u}^{\kappa^+}$ and the negative puncture of $\overline{u}^{-\kappa^-}$ detect the Reeb orbit $\gamma_0$.
\end{itemize}

We first remark that since every level of $\overline{u}$ is a finite energy holomorphic cylinder in an exact symplectic manifold and since all the levels in symplectizations are not trivial cylinders, the following inequalities are satisfied by the actions of the breaking orbits of the building $\overline{u}$:
\begin{align*}
& C^+ \mathcal{A}_{\alpha_0}(\gamma^j_+)= \mathcal{A}_{C^+\alpha_0}(\gamma^j_+) >  \mathcal{A}_{C^+\alpha_0}(\gamma^j_-)= C^+ \mathcal{A}_{\alpha_0}(\gamma^j_-) \mbox{ for all } 1 \leq j \leq \kappa^+, \\ 
& C^+ \mathcal{A}_{\alpha_0}(\gamma^0_+) =\mathcal{A}_{C^+\alpha_0}(\gamma^0_+)  \geq \mathcal{A}_{C^-\alpha_0}(\gamma^0_-) = C^-\mathcal{A}_{\alpha_0}(\gamma^0_-), \\ 
& C^- \mathcal{A}_{\alpha_0}(\gamma^j_+)= \mathcal{A}_{C^-\alpha_0}(\gamma^j_+) >  \mathcal{A}_{C^-\alpha_0}(\gamma^j_-)= C^- \mathcal{A}_{\alpha_0}(\gamma^j_-) \mbox{ for all } -1 \geq j \leq -\kappa^-. 
\end{align*}
Combining this with the fact that for each $  -\kappa^- < j \leq \kappa^+ $, we have $\gamma^{j-1}_+= \gamma^{j}_-$ and that $\gamma^{\kappa^+}_+ = \gamma^{\kappa^-}_- = \gamma_0$ we conclude that 
\begin{align} \label{eq:nobreakingaction}
  &  C^+\mathcal{A}_{\alpha_0}(\gamma_0) > C^+\mathcal{A}_{\alpha_0}(\gamma_{-}^j)  > C^-\mathcal{A}_{\alpha_0}(\gamma_0) \mbox{ for all }  1 \leq j \leq \kappa^+, \\
  & \nonumber   C^+\mathcal{A}_{\alpha_0}(\gamma_0) \geq C^+\mathcal{A}_{\alpha_0}(\gamma_{+}^0) \geq C^-\mathcal{A}_{\alpha_0}(\gamma_{-}^0) \geq C^-\mathcal{A}_{\alpha_0}(\gamma_0), \\
  & \nonumber C^+\mathcal{A}_{\alpha_0}(\gamma_0) > C^-\mathcal{A}_{\alpha_0}(\gamma_{+}^j)  > C^-\mathcal{A}_{\alpha_0}(\gamma_0) \mbox{ for all } -1 \geq j \geq -\kappa^-.
\end{align}

We show that $\kappa^+=\kappa^- = 0$.
We argue by contradiction and assume that this is not the case. 
If $\kappa^+>0$, we consider the top level $\overline{u}^{\kappa^+}$, which is a finite energy non-trivial cylinder in the symplectization $(\R \times Y,J^+_0)$ of $C^+\alpha_0$ positively asymptotic to $\gamma_0$ and negatively asymptotic to the Reeb orbit $\gamma^{\kappa^+}_-$ of $C^+ \alpha_0$.  
By \eqref{eq:nobreakingaction}, by our choices of $\max (f_\alpha)< C^+ < e^{\delta_{\alpha_0,\mathcal{L}_0}}$ and $e^{-\delta_{\alpha_0,\mathcal{L}_0}} < C^- < \min (f_\alpha)$,    we obtain that  $\mathcal{A}_{C^+\alpha_0}(\gamma^{\kappa^+}_-) = C^+\mathcal{A}_{\alpha_0}(\gamma^{\kappa^+}_-)$ belongs to the interval $$ (C^-\mathcal{A}_{\alpha_0}(\gamma_0), C^+\mathcal{A}_{\alpha_0}(\gamma_0)). $$ This implies that $$\mathcal{A}_{\alpha_0}(\gamma^{\kappa^+}_-) \in (\frac{C^-}{C^+}\mathcal{A}_{\alpha_0}(\gamma_0), \mathcal{A}_{\alpha_0}(\gamma_0)) .$$ Since $\frac{C^-}{C^+} > \frac{e^{-\delta_{\alpha_0,\mathcal{L}_0}}}{e^{\delta_{\alpha_0,\mathcal{L}_0}}}$ and by Condition A) of $\delta_{\alpha_0,\mathcal{L}_0} >0$ chosen in Assumption \ref{assump:delta}, we conclude that  
$$ \mathcal{A}_{\alpha_0}(\gamma^{\kappa^+}_-) \in  (\mathcal{A}_{\alpha_0}(\gamma_0) - \mu_{\alpha_0,\mathcal{L}_0}, \mathcal{A}_{\alpha_0}(\gamma_0)),$$
for $\mu_{\alpha_0,\mathcal{L}_0}$ satisfying 
\eqref{eq:munew}.
This is in  contradiction to the fact that $(\Upsilon^+)^{-1} \circ \overline{u}^{\kappa^+}$ is  a holomorphic cylinder in $(\R\times Y,\widetilde{J}_0)$ positively asymptotic to  $\gamma_0$ and negatively asymptotic to $\gamma_-^{\kappa^+}$, where $\Upsilon^+$ is defined in \eqref{eq:Upsilon}.

If $\kappa^- > 0$, we consider the bottom level $\overline{u}^{-\kappa^-}$ of $\overline{u}$. The holomorphic cylinder $\overline{u}^{-\kappa^-}$ is negatively asymptotic to $\gamma_0$ and positively asymptotic to a Reeb orbit $\gamma^{-\kappa^-}_+$ of $C^-\alpha_0$. Seeing $\gamma^{-\kappa^-}_+$ as a Reeb orbit of $\alpha_0$, using \eqref{eq:nobreakingaction} and that $\max (f_\alpha)< C^+ < e^{\delta_{\alpha_0,\mathcal{L}_0}}$ and $e^{-\delta_{\alpha_0,\mathcal{L}_0}} < C^- < \min (f_\alpha)$, and by Condition A) of $\delta_{\alpha_0,\mathcal{L}_0} >0$ chosen in Assumption \ref{assump:delta}, we obtain that  
 $\mathcal{A}_{\alpha_0}(\gamma^{\kappa^-}_+) $ belongs to the interval $$ (\mathcal{A}_{\alpha_0}(\gamma_0), \frac{C^+}{C^-}\mathcal{A}_{\alpha_0}(\gamma_0))\subset(\mathcal{A}_{\alpha_0}(\gamma_0), \mathcal{A}_{\alpha_0}(\gamma_0)+\mu_{\alpha_0,\mathcal{L}_0}),$$
for $\mu_{\alpha_0,\mathcal{L}_0}$ satisfying \eqref{eq:munew}. This is in  contradiction to the fact that $(\Upsilon^-)^{-1} \circ \overline{u}^{-\kappa^{-}}$ is a holomorphic cylinder in $(\R\times Y,\widetilde{J}_0)$ positively asymptotic to  $\gamma_-^{-\kappa^-}$ and negatively asymptotic to $\gamma_0$, where $\Upsilon^-$ is defined in \eqref{eq:Upsilon}.
This finishes the proof of Step 2.

\

\textit{Step 3: End of the proof.} 

By Step 1 and Step 2, we conclude that $\overline{u}$ has only one level. We have thus proved the first assertion of the proposition, namely that any sequence of elements $\mathcal{M}(\gamma_0,\gamma_0,\omega^R_s,J^R_s;s\in [0,1])$ has a subsequence that converges to an element of $\mathcal{M}(\gamma_0,\gamma_0,\omega^R_s,J^R_s;s\in [0,1])$. This combined with the SFT-compactness theorem of \cite{CPT} implies that $\mathcal{M}(\gamma_0,\gamma_0,\omega^R_s,J^R_s;s\in [0,1])$ coincides with its SFT-compactification and is therefore compact.
\qed

\textbf{Standing assumption for the rest of section~\ref{sec:proofmainresult}:}
\\
\textbf{From now on, we assume that  $\alpha$, $\delta_{\alpha_0,\mathcal{L}_0}>0$, $C^+$ and $C^-$ satisfy the hypothesis of Proposition~\ref{prop:compactness}. }

\subsection{Continuation of a link}\label{sec:isotopy}

In Proposition~\ref{prop:important} below we explain how to transport links of Reeb orbits along an admissible homotopy of exact symplectic cobordisms. We will start with the trivial cylinders over the Reeb orbits in $\mathcal{L}_0$ in the trivial cobordism and deform the cobordisms to the desired one. Proposition~\ref{prop:compactness} is the key ingredient that allows us to guarantee that the relevant moduli space of cylinders cannot ``stop'' during the homotopy, and must continue during the whole homotopy. We thus obtain components of these moduli spaces which are homeomorphic to closed intervals with one boundary point in the trivial cobordism, and one boundary point in the other end of the homotopy. Moreover, in Section \ref{sec:isotopies} we show that an appropriate parametrization of our moduli spaces can be realized by an ambient isotopy of $\R \times Y$. This ambient isotopy will be important in Section \ref{sec:moduli_linkcompl}, as it will allow us to preserve crucial linking properties of Reeb orbits with $\mathcal{L}_0$. We advise the reader to skip the proof of Proposition \ref{prop:important} on a first reading.

\subsubsection{A reparametrization of the moduli space}

Before stating and proving Proposition~\ref{prop:important} we introduce some notation. Let $\mathrm{n}_0$ be the number of connected components of $\mathcal{L}_0$. There is an obvious bijection between the connected components of $\mathcal{L}_0$ and simple Reeb orbits of $\alpha_0$ contained in $\mathcal{L}_0$. We choose an ordering $\{\gamma^1_0,\gamma_0^2,\ldots,\gamma_0^{\mathrm{n}_0}\}$ of the simple Reeb orbits in $\mathcal{L}_0$.

\begin{prop} \label{prop:important}
Assume that $\alpha$, $\delta_{\alpha_0,\mathcal{L}_0}>0$, $C^+$ and $C^-$ satisfy the hypothesis of Proposition \ref{prop:compactness}. 
Let $(\R \times Y, \overline{\lambda}^R_s,\overline{\omega}^R_s,\overline{J}^R_s )_{s \in [0,1]}$ be an admissible homotopy of exact symplectic cobordisms endowed with almost complex structures as constructed in Section \ref{sec:sympcobordisms}.
    
Then there exists an admissible homotopy of exact symplectic cobordisms endowed with almost complex structures $(\R \times Y, \lambda^R_s,\omega^R_s,J^R_s )_{s \in [0,1]}$  as constructed in Section \ref{sec:sympcobordisms}, with the following properties:
\begin{itemize}
    \item the endpoints of the new homotopy coincide with the endpoints of the given homotopy: 
    $$(\overline{\lambda}^R_0,\overline{\omega}^R_0,  \overline{J}^R_0) =({\lambda}^R_0,{\omega}^R_0, {J}^R_0) \mbox{  and  }(\overline{\lambda}^R_1,\overline{\omega}^R_1,  \overline{J}^R_1) =({\lambda}^R_1,{\omega}^R_1, {J}^R_1),$$   
    \item the almost complex structure $J^R_s$ is regular for simple Reeb orbits not contained in $\mathcal{L}_0$, in the sense that if $\gamma$ is a simple Reeb orbit not contained in $\mathcal{L}_0$, then the linearized Cauchy-Riemann operator over any simply covered holomorphic cylinder with one positive and one negative puncture both asymptotic to $\gamma$ is surjective,
    \item for each $1 \leq j \leq \mathrm{n}_0$ there exists a smooth path $\mathcal{I}^R_j:= \{ (s,\widetilde{v}^R_j(s)) \ | \ s \in [0,1]\} $ in $[0,1] \times C^\infty(\R \times S^1 , \R \times Y)$ such that $\widetilde{v}^R_j(s)$ is a finite energy holomorphic cylinder in $(\R \times Y,\lambda^R_s, \omega^R_s, J^R_s)$ with one positive and one negative puncture both asymptotic to $\gamma^j_0$, and $\widetilde{v}^R_j(0)$ is the trivial cylinder $\widetilde{u}^0_{\gamma^j_0}$ over the Reeb orbit $\gamma^j_0$,
    \item for fixed $s \in [0,1]$ and $i \neq j$ the images of  $\widetilde{v}^R_i(s)$ and $\widetilde{v}^R_j(s)$ are disjoint.
\end{itemize}

\end{prop}

\textit{Proof.} 
We will prove the proposition by induction in the following way. Let $1 \leq i < \mathrm{n}_0$, and assume that we have constructed  an admissible homotopy of exact symplectic cobordisms endowed with almost complex structures $(\R \times Y, \lambda^R_s(i),\omega^R_s(i),J^R_s(i) )_{s \in [0,1]}$ such that:
\begin{itemize}
    \item for each $1 \leq l \leq i$ there exists a smooth path $\mathcal{I}^i_l:= \{ (s,\widetilde{v}^{R,i}_l(s)) \ | \ s \in [0,1]\} $ in $[0,1] \times C^\infty(\R \times S^1 , \R \times Y)$ such that $\widetilde{v}^{R,i}_l(s)$ is a finite energy holomorphic cylinder in $(\R \times Y,  \lambda^R_s(i),\omega^R_s(i), J^R_s(i))$ with one positive and one negative puncture both asymptotic to $\gamma^l_0$, and $\widetilde{v}^R_l(0)$ is the trivial cylinder $\widetilde{u}^0_{\gamma^l_0}$ over the Reeb orbit $\gamma^l_0$,
    \item for fixed $s \in [0,1]$ and $1 \leq  l' \neq l \leq i$ the images of  $\widetilde{v}^{R,i}_{l'}(s)$ and $\widetilde{v}^{R,i}_l(s)$ are disjoint,
     \item the homotopies satisfy $(\overline{\lambda}^R_0,\overline{\omega}^R_0,  \overline{J}^R_0) =({\lambda}^R_0(i),{\omega}^R_0(i), {J}^R_0(i))$ and  \linebreak $(\overline{\lambda}^R_1,\overline{\omega}^R_1, \overline{J}^R_1) =({\lambda}^R_1(i),{\omega}^R_1(i), {J}^R_1(i))$, 
     \item the almost complex structure $J^R_s$ is regular for simple Reeb orbits not contained in $\mathcal{L}_0$, in the sense that if $\gamma$ is a simple Reeb orbit not contained in $\mathcal{L}_0$, then the linearized Cauchy-Riemann operator over any simply covered holomorphic curved which is asymptotic to $\gamma$ at one of its punctures is surjective.

\end{itemize}

The inductive step will then consist in showing that assuming this, it is possible to construct an admissible homotopy of exact symplectic cobordisms endowed with almost complex structures  $$(\R \times Y, \lambda^R_s(i+1),\omega^R_s(i+1),J^R_s(i+1) )_{s \in [0,1]}$$ satisfying the analogous conditions with $i$ replaced by $i+1$.

\

\textit{Step 1: We establish the induction hypothesis for $i=1$.} 

Since  $(\R \times Y, \overline{\lambda}^R_s,\overline{\omega}^R_s,\overline{J}^R_s )_{s \in [0,1]}$ is  an admissible homotopy of exact symplectic cobordisms endowed with almost complex structures, we know that $\overline{J}^R_0$ and $\overline{J}^R_1$ are regular almost complex structures in the sense that the linearized Cauchy-Riemann operators over simply covered finite energy holomorphic curves are surjective. By standard transversality techniques it is then possible, for a sufficiently small $\updelta>0$, to perturb the homotopy $(\overline{J}^R_s)_{s \in [0,1]}$ of almost complex structures inside the open set $(\delta,1 -\delta) \times (-R-1,R+1) \times  Y $, to obtain a homotopy $(\overline{J}^R_s(1))_{s \in [0,1]}$ of almost complex structure such that:
\begin{itemize}
    \item $(\R \times Y, \overline{\lambda}^R_s,\overline{\omega}^R_s,\overline{J}^R_s(1) )_{s \in [0,1]}$ is still an admissible homotopy of exact symplectic cobordisms endowed with almost complex structures,
    \item the moduli space $\mathcal{M}(\gamma^1_0,\gamma^1_0,\omega^R_s,\overline{J}^R_s(1);s\in [0,1])$ is a one-dimensional manifold with boundary,
    \item the almost complex structure $\overline{J}^R_s(1)$ is regular for simple Reeb orbits not contained in $\mathcal{L}_0$ in the sense explained above.
\end{itemize}
It is clear that $(\overline{\lambda}^R_0,\overline{\omega}^R_0,  \overline{J}^R_0) =(\overline{\lambda}^R_0,\overline{\omega}^R_0,  \overline{J}^R_0(1))$ and  \linebreak $(\overline{\lambda}^R_1,\overline{\omega}^R_1,  \overline{J}^R_1) =(\overline{\lambda}^R_1,\overline{\omega}^R_1,  \overline{J}^R_1(1))$.

By Proposition \ref{prop:compactness} the moduli space $\mathcal{M}(\gamma^1_0,\gamma^1_0,\omega^R_s,\overline{J}^R_s(1);s\in [0,1])$ is compact, and we conclude that it is a compact one-dimensional manifold with boundary. We let $\mathcal{C}^1_1$ be the connected component of this moduli space that contains $(0,\widetilde{u}^0_{\gamma^1_0})$, where $\widetilde{u}^0_{\gamma^1_0}$ is the trivial cylinder over $\gamma^1_0$. It is immediate that $\mathcal{C}^1_1$ is  diffeomorphic to an interval. Since there are no broken holomorphic buildings in $\mathcal{C}^1_1$, we know that the two boundary points of $\mathcal{C}^1_1$ must have first coordinate equal to $0$ or $1$. Since $J^R_0$ is diffeomorphic to an $\R$-invariant almost complex structure on a symplectization, we know that the trivial cylinder $(0,\widetilde{u}^0_{\gamma^1_0})$ is the only element in $\mathcal{M}(\gamma^1_0,\gamma^1_0,\omega^R_s,\overline{J}^R_s(1);s\in [0,1])$ with first coordinate $0$. It follows that the boundary of $\mathcal{C}^1_1$ is formed by two elements $(0,\widetilde{u}^0_{\gamma^1_0})$ and $(1,\widetilde{u}^{R}_{\gamma^1_0})$, where $\widetilde{u}^1_{\gamma^1_0}$ is a finite energy cylinder of $(\lambda^R_1,J^R_1)$ with one positive puncture and one negative puncture both asymptotic to $\gamma^1_0$.

Let $\Psi^0_1:[0,1] \to [0,1] \times C^\infty(\R \times S^1 , \R \times Y)$
be a smooth parametrization of  $\mathcal{C}^1_1$ satisfying $\Psi^0_1(0) = (0,\widetilde{u}^0_{\gamma^1_0})$ and $\Psi^0_1(1)=(1,\widetilde{u}^{R}_{\gamma^1_0})$.  We write $$\Psi^0_1(\tau)= (\sigma^0_1(\tau),\psi^0_1(\tau)).$$
The map $\sigma^0_1$ is a smooth map from $[0,1]$ to itself such that $(\sigma^0_1)^{-1}(0)=0$ and $(\sigma^0_1)^{-1}(1)=1$. Moreover, it is clear that by a suitable choice of the parametrization  $\Psi^0_1$ we can also assume that $\sigma^0_1$ coincides with the identity on small neighbourhoods of the points $0$ and $1$.

We define the admissible homotopy of exact symplectic cobordisms endowed with almost complex structures $(\R \times Y, \lambda^R_s(1),\omega^R_s(1),\check{J}^R_s(1) )_{s \in [0,1]}$ by 
\begin{align}
    \lambda^R_s(1) &:= \overline{\lambda}^R_{\sigma^0_1(s)}, \\
    \omega^R_s(1) &:= \overline{\omega}^R_{\sigma^0_1(s)}, \\
    \check{J}^R_s(1) &:= \overline{J}^R_{\sigma^0_1(s)}(1).
\end{align}
We define $\widetilde{v}^{R,1}_1(s) := \psi^0_1(s)$. It is immediate that $\widetilde{v}^{R,1}_1(s)$ is a finite energy holomorphic cylinder in $(\R \times Y, \lambda^R_s(1),\omega^R_s(1),\check{J}^R_s(1) )$ with one positive and one negative puncture both asymptotic to $\gamma^1_0$.
We define $$\mathcal{I}^{1}_1= \{ (s,\widetilde{v}^{R,1}_1(s)) \ | \ s \in [0,1] \}.$$

It is immediate to check that $(\R \times Y, \lambda^R_s(1),\omega^R_s(1),\check{J}^R_s(1) )_{s \in [0,1]}$ and $\mathcal{I}^1_1$ satisfy the first three assumptions of the inductive hypothesis for $i=1$. To satisfy the fourth assumption, i.e., to obtain $(\R \times Y, \lambda^R_s(1),\omega^R_s(1), {J}^R_s(1) )_{s \in [0,1]}$ such that the linearized Cauchy-Riemann operator of $J^R_s(1)$ over any simply covered finite energy holomorphic cylinder which has a positive and a negative puncture both asymptotic to a Reeb orbit not contained in $ \{\gamma_0^1\}$ is surjective, we have to perturb $(\check{J}^R_s(1) )_{s\in [0,1]}$. The techniques developed by Bourgeois and Dragnev \cite{Bourgeois,Dragnev} (see also \cite{Momin} and \cite{Wendl-lectures}) show that this transversality condition can be achieved by an arbitrarily small perturbation ${J}^R_s(1)$ of $\check{J}^R_s(1)$, with the support of this perturbation contained in  
$$\mathcal{O}^{R,1}=\Big((0,1) \times  \big( (-R-1,R+1) \times  Y \big)\Big) \setminus \left\{(s,\Im(\widetilde{v}^{R,1}_1(s)))\, |\, s\in [0,1]\right\}  $$
 with the property that  $(\R \times Y, \lambda^R_s(1),\omega^R_s(1),J^R_s(1) )_{s \in [0,1]}$ is still an admissible homotopy of exact symplectic cobordisms endowed with almost complex structures. The reason why we can obtain regularity with a perturbation whose support is in a set as above is that 
\begin{itemize}
    \item any simply covered finite energy holomorphic cylinder in \linebreak $(\R \times Y, \lambda^R_s(1),\omega^R_s(1),\check{J}_s^R(1))$, for some $s\in [0,1]$,  that has a positive and a negative puncture both asymptotic to a Reeb orbit not contained in $ \{\gamma_0^1\}$ has a somewhere injective point with image in $\mathcal{O}^{R,1} \cap (\{s\} \times (-R-1,R+1) \times Y)$, 
    \item the almost complex structures $\check{J}^R_0(1)$ and $\check{J}^R_1(1)$ are regular.
\end{itemize}
Since ${J}^R_s(1)$ and $\check{J}^R_s(1)$ coincide in a set containing the image of $\widetilde{v}^{R,1}_1(s)$,
we obtain that $\widetilde{v}^{R,1}_1(s)$ is still holomorphic for ${J}^R_s(1)$ and conclude that \linebreak $(\R \times Y, \lambda^R_s(1),\omega^R_s(1),J^R_s(1) )_{s \in [0,1]}$ and $\mathcal{I}^{R,1}_1$ satisfy the four assumptions of the inductive hypothesis for $i=1$.

\ 

\textit{Step 2: The inductive step.}

We assume that for $1 \leq i< \mathrm{n}_0$ we have $(\R \times Y, \lambda^R_s(i),\omega^R_s(i),J^R_s(i) )_{s \in [0,1]}$ and  a smooth path $\mathcal{I}^{i}_l=  \{ (s,\widetilde{v}^{R,i}_l(s)) \ | \ s \in [0,1]\}$ for $1 \leq l \leq i$ satisfying the four conditions of the inductive hypothesis.

We then consider the moduli space $\mathcal{M}(\gamma^{i+1}_0,\gamma^{i+1}_0,\omega^R_s(i),J^R_s(i);s\in [0,1])$. It follows from Proposition \ref{prop:compactness} that this moduli space is compact. By the regularity property of $J^R_s(i)$, we also know that $\mathcal{M}(\gamma^{i+1}_0,\gamma^{i+1}_0,\omega^R_s(i),J^R_s(i);s\in [0,1])$ is a one-dimensional manifold, and conclude that it is a compact one-dimensional manifold with boundary. We denote by $\mathcal{W}^{i}_{i+1}$ the subset of \linebreak $\mathcal{M}(\gamma^{i+1}_0,\gamma^{i+1}_0,\omega^R_s(i),J^R_s(i);s\in [0,1])$, formed by elements disjoint from $\mathcal{I}^{R,i}_l$ for $0 \leq l \leq i$. More precisely, $(s,\widetilde{w}_s) \in \mathcal{W}^{i}_{i+1}$ if it belongs to \linebreak $\mathcal{M}(\gamma^{i+1}_0,\gamma^{i+1}_0,\omega^R_s(i),J^R_s(i);s\in [0,1])$ and if there is no intersection between the image of $\widetilde{w}_s$ and the images of $\widetilde{v}^{R,i}_l(s)$ for all $1 \leq l \leq i$.

We claim that $\mathcal{W}^{i}_{i+1}$ is a one-dimensional manifold with boundary. To prove this, it suffices to show that $\mathcal{W}^{i}_{i+1}$ is open and closed in \linebreak $\mathcal{M}(\gamma^{i+1}_0,\gamma^{i+1}_0,\omega^R_s(i),J^R_s(i);s\in [0,1])$.

To see that $\mathcal{W}^{i}_{i+1}$ is closed, we take a sequence of elements $(s_n,\widetilde{w}_{s_n})$  in $\mathcal{W}^{i}_{i+1}$ that converges to an element $(s,\widetilde{w}_s)$ of  $\mathcal{M}(\gamma^{i+1}_0,\gamma^{i+1}_0,\omega^R_s(i),J^R_s(i);s\in [0,1])$ that does not belong to $\mathcal{W}^{i}_{i+1}$. Since $\mathcal{M}(\gamma^{i+1}_0,\gamma^{i+1}_0,\omega^R_s(i),J^R_s(i);s\in [0,1])$  is a compact one-dimensional manifold with boundary, it follows that for sufficiently large $n$ there exists a homotopy completely contained in \linebreak $\mathcal{M}(\gamma^{i+1}_0,\gamma^{i+1}_0,\omega^R_s(i),J^R_s(i);s\in [0,1])$ between $(s_n,\widetilde{w}_{s_n})$ and $(s,\widetilde{w}_s)$. By the homotopy invariance of Siefring's intersection number $I(\cdot,\cdot)$ for asymptotically cylindrical surfaces in symplectic cobordisms\footnote{Here we are using that  $\widetilde{v}^i_l(s_n)$ and $\widetilde{v}^i_l(s)$ are homotopic as asymptotically cylindrical surfaces, which is clear from the fact that both are contained in $\mathcal{I}^i_l$.} (see~\cite{Siefring}), we conclude that for every $1 \leq l \leq i$ we have $I(\widetilde{w}_{s_n},\widetilde{v}^i_l(s_n)) = I(\widetilde{w}_{s},\widetilde{v}^i_l(s))$. Since the images of $\widetilde{w}_{s_n}$ and $\widetilde{v}^i_l(s_n)$ are disjoint and those cylinders have no common asymptotic orbits,  it follows 
 that $I(\widetilde{w}_{s_n},\widetilde{v}^i_l(s_n))=0$. We conclude that $I(\widetilde{w}_{s},\widetilde{v}^i_l(s))=0$, and since $\widetilde{w}_{s}$ and $\widetilde{v}^i_l(s)$ have no common asymptotic orbits,  the results of \cite{Siefring} imply that the images of $\widetilde{w}_{s}$ and $\widetilde{v}^i_l(s)$ are disjoint. We thus conclude that  $(s,\widetilde{w}_s) \in \mathcal{W}^{i}_{i+1}$, and therefore $\mathcal{W}^{i}_{i+1}$ is closed.

To prove that $\mathcal{W}^{i}_{i+1}$ is open in $\mathcal{M}(\gamma^{i+1}_0,\gamma^{i+1}_0,\omega^R_s(i),J^R_s(i);s\in [0,1])$, we will show that $\mathcal{M}(\gamma^{i+1}_0,\gamma^{i+1}_0,\omega^R_s(i),J^R_s(i);s\in [0,1]) \setminus \mathcal{W}^{i}_{i+1}$ is closed. To show this, let $(s_n,\widetilde{u}_{s_n})$ be a sequence of elements of  $$\mathcal{M}(\gamma^{i+1}_0,\gamma^{i+1}_0,\omega^R_s(i),J^R_s(i);s\in [0,1]) \setminus \mathcal{W}^{i}_{i+1}$$ that converges to an element $(s,\widetilde{w}_s)$ of $\mathcal{M}(\gamma^{i+1}_0,\gamma^{i+1}_0,\omega^R_s(i),J^R_s(i);s\in [0,1])$. By passing to a subsequence, we can assume that there exists $1 \leq l_0\leq i$ such that the image of every $\widetilde{u}_{s_n}$ has non-empty intersection with the image of $\widetilde{v}^{R,i}_{l_0}(s_n)$. Since $\widetilde{u}_{s_n}$ and $\widetilde{v}^{R,i}_{l_0}(s_n)$ have no common asymptotic orbits, we conclude that the Siefring's intersection number of $\widetilde{u}_{s_n}$ and $\widetilde{v}^{R,i}_{l_0}(s_n)$ satisfies $I(\widetilde{v}^{R,i}_{l_0}(s_n),\widetilde{u}_{s_n})>0$. Since $\mathcal{M}(\gamma^{i+1}_0,\gamma^{i+1}_0,\omega^R_s(i),J^R_s(i);s\in [0,1])$ is a compact one-dimensional manifold with boundary, we know that for sufficiently large $n$ there exists a homotopy completely contained in   $\mathcal{M}(\gamma^{i+1}_0,\gamma^{i+1}_0,\omega^R_s(i),J^R_s(i);s\in [0,1])$ between $(s_n,\widetilde{u}_{s_n})$ and $(s,\widetilde{w}_s)$. By the homotopy invariance of Siefring's intersection number, we conclude that  $I(\widetilde{w}_s, \widetilde{v}^{R,i}_{l_0}(s))>0$. Since $\widetilde{w}_s$ and $\widetilde{v}^{R,i}_{l_0}(s)$ have no common asymptotic orbits, the fact that $I(\widetilde{w}_s, ,\widetilde{v}^{R,i}_{l_0}(s))>0$ implies that the images of $\widetilde{w}_s$ and $\widetilde{v}^{R,i}_{l_0}(s)$ intersect. This implies that $(s,\widetilde{w}_s)$ does not belong to $\mathcal{W}^{i}_{i+1}$ and we conclude that $\mathcal{M}(\gamma^{i+1}_0,\gamma^{i+1}_0,\omega^R_s(i),J^R_s(i);s\in [0,1]) \setminus \mathcal{W}^{i}_{i+1}$ is closed, and therefore $\mathcal{W}^{i}_{i+1}$ is open.

Hence, any connected component of  $\mathcal{M}(\gamma^{i+1}_0,\gamma^{i+1}_0,\omega^R_s(i),J^R_s(i);s\in [0,1])$ that intersects $\mathcal{W}^{i}_{i+1}$ is completely contained in $\mathcal{W}^{i}_{i+1}$. Let $\mathcal{C}^{R,i}_{i+1}$ be the connected component of  $\mathcal{M}(\gamma^{i+1}_0,\gamma^{i+1}_0,\omega^R_s(i),J^R_s(i);{s\in [0,1]})$ that contains $(0,\widetilde{u}^0_{\gamma_0^{i+1}})$,  where $\widetilde{u}^0_{\gamma_0^{i+1}}$ is the trivial cylinder over $\gamma_0^{i+1}$. To see that $(0,\widetilde{u}^0_{\gamma_0^{i+1}})$ is in $\mathcal{W}^{i}_{i+1}$, we notice that it  does not intersect $\widetilde{v}^{R,i}_l(0)$ for every $1 \leq l \leq i$,  since $\widetilde{v}^{R,i}_l(0)= \widetilde{u}^{0}_{\gamma^l_0}$. It is then easy to see that  $\mathcal{C}^{R,i}_{i+1}$ is an interval one of whose boundary components is $(0,\widetilde{u}^0_{\gamma_0^{i+1}})$. The other boundary component must be an element $(1,\widetilde{u}^R_{\gamma_0^{i+1}})$ with first coordinate $1$ since there are no broken cylinders in the compact moduli space  $\mathcal{M}(\gamma^{i+1}_0,\gamma^{i+1}_0,\omega^R_s(i),J^R_s(i);{s\in [0,1]})$ and the only element of that moduli space 
with first coordinate $0$ is $(0,\widetilde{u}^0_{\gamma_0^{i+1}})$.

Let $\Psi^i_{i+1}:[0,1] \to [0,1] \times C^\infty(\R \times S^1 , \R \times Y)$
be a smooth parametrization of  $\mathcal{C}^{R,i}_{i+1}$ satisfying $\Psi^i_{i+1}(0) = (0,\widetilde{u}^0_{\gamma^{i+1}_0})$ and $\Psi^i_{i+1}(1)=(1,\widetilde{u}^{R}_{\gamma^{i+1}_0})$.  We write $$\Psi^i_{i+1}(\tau)= (\sigma^i_{i+1}(\tau),\psi^i_{i+1}(\tau)).$$
The map $\sigma^i_{i+1}$ is a smooth map from $[0,1]$ to itself such that $(\sigma^i_{i+1})^{-1}(0)=0$ and $(\sigma^i_{i+1})^{-1}(1)=1$. Moreover, it is clear that by a suitable choice of the parametrization  $\Psi^i_{i+1}$ we can also assume that $\sigma^i_{i+1}$ coincides with the identity on small neighbourhoods of the points $0$ and $1$. 

We define the admissible homotopy of exact symplectic cobordisms endowed with almost complex structures $(\R \times Y, \lambda^R_s(i+1),\omega^R_s(i+1),\check{J}^R_s(i+1) )_{s \in [0,1]}$ by 
\begin{align}
    \lambda^R_s(i+1) := {\lambda}^R_{\sigma^i_{i+1}(s)}(i), \\
    \omega^R_s(i+1) := {\omega}^R_{\sigma^i_{i+1}(s)}(i), \\
    \check{J}^R_s(i+1) := {J}^R_{\sigma^i_{i+1}(s)}(i).
\end{align}

For $1 \leq l \leq i$, we define $$\widetilde{v}^{R,i+1}_l(s) := \widetilde{v}^{R,i}_{l}\circ \sigma^i_{i+1} (s ).$$ It is immediate that $\widetilde{v}^{R,i+1}_l(s)$ is a finite energy holomorphic cylinder in $(\R \times Y, \lambda^R_s(i+1),\omega^R_s(i+1),\check{J}^R_s(i+1) )$ with one positive  and one negative puncture both asymptotic to $\gamma^l_0$. 
We then let $\mathcal{I}^{i+1}_l= \{ (s,\widetilde{v}^{R,i+1}_l(s)) \ | \ s \in [0,1] \}$. Lastly, we define $\widetilde{v}^{i+1}_{i+1}(s):=\psi^i_{i+1}(s)$ and let $\mathcal{I}^{i+1}_{i+1}= \{(s,\widetilde{v}^{R,i+1}_{i+1}(s)) \ | \ s\in [0,1] \}$. Finally, it is straightforward to see that for each $1 \leq l \leq i+1$, the cylinder $\widetilde{v}^{R,i+1}_{l}(s) $ is holomorphic for $\check{J}^R_s(i+1)$.
We obtain that $(\R \times Y, \lambda^R_s(i+1),\omega^R_s(i+1),\check{J}^R_s(i+1) )$ and the smooth paths $\mathcal{I}^{i+1}_l$ for $1 \leq l \leq i+1$ satisfy the first three conditions of the inductive hypothesis.

To satisfy the fourth assumption, i.e.,  to obtain $(\R \times Y, \lambda^R_s(i+1),\omega^R_s(i+1), {J}^R_s(i+1) )_{s \in [0,1]}$ such that the linearized Cauchy-Riemann operator of $J^R_s(i+1)$ over any simply covered finite energy holomorphic cylinder that has a positive and a negative puncture both asymptotic to a Reeb orbit not contained in $ \{\gamma_0^1,\ldots,\gamma_0^{i+1}\}$ is surjective, we have to perturb $(\check{J}^R_s(1) )_{s\in [0,1]}$.
Again the methods of Bourgeois and Dragnev \cite{Bourgeois,Dragnev} 
show that this transversality  can be achieved by an arbitrarily small perturbation ${J}^R_s(i+1)$ of $\check{J}^R_s(i+1)$, with the support of the perturbation contained in  
$$\mathcal{O}^{R,i+1}=\Big((0,1 ) \times  \big( (-R-1,R+1) \times  Y \big)\Big) \setminus \bigcup_{l=1}^{i+1}\left\{(s,\Im(\widetilde{v}^{R,i}_l(s)))\, |\, s\in [0,1]\right\}  $$
with the property that $(\R \times Y, \lambda^R_s(i+1),\omega^R_s(i+1),J^R_s(i+1) )_{s \in [0,1]}$ is still an admissible homotopy of exact symplectic cobordisms endowed with almost complex structures. The reason why we can obtain regularity with a perturbation whose support is in a set as above is that 
\begin{itemize}
    \item any simply covered finite energy holomorphic cylinder in \linebreak $(\R \times Y, \lambda^R_s(i+1),\omega^R_s(i+1),\check{J}^R_s(i+1) )$, for some $s\in [0,1]$, that has a positive and a negative puncture both asymptotic to a Reeb orbit not contained in $\mathcal{L}_0$ has a somewhere injective point with image in $\mathcal{O}^{R,i+1} \cap (\{s\}\times (-R-1,R+1)\times Y)$, 
    \item the almost complex structures $\overline{J}^R_0(i+1)$ and $\overline{J}^R_1(i+1)$ are regular.
\end{itemize}

Since ${J}^R_s(i+1)$ and $\check{J}^R_s(i+1)$ coincide in a set containing the images of $\widetilde{v}^{R,i+1}_l(s)$ for $1 \leq l \leq i+1$,
we obtain that $\widetilde{v}^{R,i+1}_l(s)$ is still holomorphic for ${J}^R_s(1)$. It is then immediate to see that $(\R \times Y, \lambda^R_s(i+1),\omega^R_s(i+1),{J}^R_s(i+1) )$ and the smooth paths $\mathcal{I}^{i+1}_l$ for $1 \leq l \leq i+1$ satisfy all the four conditions of the inductive hypothesis. This completes the proof of the inductive step, and of the proposition. 
\qed


\subsubsection{From the moduli spaces to ambient isotopies of cylinders}\label{sec:isotopies}
We assume that we have for each positive real number $R$ an admissible homotopy of exact symplectic cobordisms endowed with almost complex structures $(\R \times Y, \lambda^R_s,\omega^R_s,J^R_s )_{s \in [0,1]}$ and  $\mathcal{I}^R_j$ for $1 \leq j \leq \mathrm{n}_0$ satisfying the assumptions of Proposition \ref{prop:important}, and obtained from applying Proposition \ref{prop:important} to a homotopy of exact symplectic cobordisms endowed with almost complex structures  $(\R \times Y, \overline{\lambda}^R_s,\overline{\omega}^R_s,\overline{J}^R_s )_{s \in [0,1]}$ as constructed in  Section \ref{sec:sympcobordisms}.

Recall that, in this case, 
\begin{itemize}
    \item $(\R \times Y, (\mathbf{T}_{R})^*\lambda^R_1,(\mathbf{T}_{R})^*\omega^R_1,(\mathbf{T}_{R})^*J^R_1)$ converges in $C^\infty_{\rm loc}$ as \linebreak $R \to +\infty$ to $(\R \times Y, \lambda^+_1,\omega^+_1,J^{+\infty}_1 )$,
    \item $(\R \times Y, (\mathbf{T}_{-R})^*\lambda^R_1,(\mathbf{T}_{-R})^*\omega^R_1,(\mathbf{T}_{-R})^*J^R_1)$ converges in $C^\infty_{\rm loc}$ as \linebreak $R \to +\infty$ to $(\R \times Y, \lambda^-_1,\omega^-_1,J^{-\infty}_1 )$.
\end{itemize}

For every $1 \leq j\leq \mathrm{n}_0 $, we now recall some compactness properties of the homotopies $\mathcal{I}^R_j= \{ (s,\widetilde{v}^R_j(s)) \ | \ s \in [0,1]\} $.
It follows by the compactness results of \cite{CPT} that for any sequence $s_n \to s \in [0,1] $ we have that $\widetilde{v}^R_j(s_n)$ converges in $C^\infty$ to\footnote{The reason why we can improve the usual $C^\infty_{\rm loc}$-convergence to proper $C^\infty$-convergence is that the compact set of holomorphic curves $\mathcal{I}^R_j$ does not contain holomorphic buildings with multiple levels.} $\widetilde{v}^R_j(s)$.
For each $1 \leq j \leq \mathrm{n}_0$, let $\mathcal{U}_j$ be a small tubular neighbourhood of $\gamma^j_0$, and we assume moreover that the $\mathcal{U}_j$ are all disjoint. We fix diffeomorphisms $\varPhi_j: \mathcal{U}_j \to S^1 \times \D $ that take $\gamma^j_0$ to $S^1 \times \{0\}$. Let $\widetilde{\varPhi}_j := \mathrm{Id}_{\R} \times \varPhi_j : \R \times \mathcal{U}_j \to \R \times S^1 \times \D $.
Combining the $C^\infty$-convergence with the asymptotic behaviour of finite holomorphic curves near punctures we can  find a positive real number $\mathrm{N}_R$ such that for all $s \in [0,1]$ the intersections $\mathrm{U}^{\mathrm{N}_R, +}_{j,s}:= \mathrm{Im}(\widetilde{v}^R_j(s)) \cap \big( [\mathrm{N}_R,+\infty) \times Y \big)$ and $\mathrm{U}^{\mathrm{N}_R, -}_{j,s}:= \mathrm{Im}(\widetilde{v}^R_j(s)) \cap \big( (-\infty, -\mathrm{N}_R] \times Y \big)$ are cylinders, and moreover 
\begin{itemize}
    \item $\mathrm{U}^{\mathrm{N}_R, +}_{j,s} \subset \big( [\mathrm{N}_R, +\infty) \times \mathcal{U}_j \big)$,  $\mathrm{U}^{\mathrm{N}_R, -}_{j,s} \subset \big( (-\infty, -\mathrm{N}^R] \times \mathcal{U}_j \big)$,
    \item $\widetilde{\varPhi}_j(\mathrm{U}^{\mathrm{N}_R, +}_{j,s})= \{(\os,\ot,\eta^{\mathrm{N}_R, +}_{j,s}(\os,\ot) ) \ | \ (\os,\ot) \in [\mathrm{N}_R,+\infty) \times S^1 \}$, where $\eta^{\mathrm{N}_R, +}_{j,s}: [\mathrm{N}_R,+\infty) \times S^1 \to \D$ is a smooth function,
    \item $\widetilde{\varPhi}_j(\mathrm{U}^{\mathrm{N}_R, -}_{j,s})= \{(\os,\ot,\eta^{\mathrm{N}_R, -}_{j,s}(\os,\ot) ) \ | \ (\os,\ot) \in (-\infty,\mathrm{N}_R] \times S^1 \}$, where $\eta^{\mathrm{N}_R, -}_{j,s}: (-\infty,\mathrm{N}_R] \times S^1 \to \D$ is a smooth function,
    \item the one-parameter families of functions $\eta^{\mathrm{N}_R, +}_{j,s}$ and $\eta^{\mathrm{N}_R, -}_{j,s}$ depend smoothly on $s$.
\end{itemize}
Geometrically speaking, the conditions above mean that in the regions \linebreak $(-\infty, -\mathrm{N}_R] \times Y$ and $ [\mathrm{N}_R,+\infty) \times Y$ the images of the finite energy cylinders $\widetilde{v}^R_j(s)$ are graphs over $(-\infty, -\mathrm{N}_R] \times \gamma^j_0$ and $ [\mathrm{N}_R,+\infty) \times \gamma^j_0 $. 

Fix  an auxiliary Riemannian metric $g_0$ on $Y $. We fix the Riemannian metric $\overline{g}$ on $\R \times Y$ obtained as the Riemannian product of the usual metric on $\R$ with $g_0$. From now on, when we talk of the normal bundle of a surface embedded in $\R \times Y$, we mean the normal bundle with respect to $\overline{g}$. Similarly, when we talk about the exponential map of a vector field over a surface embedded in $\R \times Y$, we mean the exponential map with respect to $\overline{g}$.

For any $1 \leq j \leq \mathrm{n}_0$ and $s \in [0,1]$, we let $\mathrm{V}^R_{j,s}$ be the image of the finite energy holomorphic cylinder $\widetilde{v}^R_j(s)$. By our previous discussion, $\mathrm{V}^R_{j,s}$ is an embedded cylinder in $\R \times Y$, which contains the half cylinders $\mathrm{U}^{\mathrm{N}_R, +}_{j,s}$ and $\mathrm{U}^{\mathrm{N}_R, -}_{j,s}$. We consider the restriction $T(\R \times Y) |_{\mathrm{V}^R_{j,s}}$ of the tangent bundle of $\R \times Y$ to $\mathrm{V}^R_{j,s}$. The  tangent bundle $T \mathrm{V}^R_{j,s}$ of $\mathrm{V}^R_{j,s}$ is a rank $2$ sub-bundle of $T(\R \times Y) |_{\mathrm{V}^R_{j,s}}$. We say that a rank $2$ sub-bundle $\mathfrak{V}^R_{s,j}$ of $T(\R \times Y) |_{\mathrm{V}^R_{j,s}}$ is a \textit{complement of} $T \mathrm{V}^R_{j,s}$, if for every $(a,p) \in \mathrm{V}^R_{j,s}$ we have $$T_{(a,p)}(\R \times Y) = T_{(a,p)} (\mathrm{V}^R_{j,s}) \oplus  (\mathfrak{V}^R_{s,j})_{(a,p)}. $$
The normal bundle $\mathfrak{N}^R_{j,s}$ of $\mathrm{V}^R_{j,s}$ is a complement of 
$T \mathrm{V}^R_{j,s}$. Moreover, $\mathfrak{N}^R_{j,s}$ varies smoothly with $s \in [0,1]$, in the sense that it is a smooth  family of submanifolds of $T(\R \times Y)$. 

We introduce some notation. The tangent planes to the disks $\{t\} \times \D \subset S^1 \times \D$ give a smooth 2-dimensional distribution $H$ in $S^1 \times \D$. For each $1\leq j \leq \mathrm{n}_0$ we let $H_j:= \Phi_j^*H$ the induced 2-dimensional distribution in $\mathcal{U}_j$.

Let $\mathrm{U}^{\mathrm{2N}_R, +}_{j,s} \subset \mathrm{U}^{\mathrm{N}_R, +}_{j,s}$ be the closed half-cylinder $\mathrm{U}^{\mathrm{N}_R, +}_{j,s} \cap [2 \mathrm{N}_R,+\infty) \times Y$. Similarly, $\mathrm{U}^{2\mathrm{N}_R, -}_{j,s}$ is the closed half-cylinder $\mathrm{U}^{\mathrm{N}_R, -}_{j,s} \cap (-\infty,-2 \mathrm{N}_R] \times Y$. Let $\uppi_2: \R \times Y \to Y$ be the projection on the second coordinate. 
Let $\mathfrak{W}^{R,+}_{j,s}$ be the pullback $ (\uppi_2)^*(H_j)|_{\mathrm{U}^{\mathrm{2N}_R, +}_{j,s}}$. Notice that $\mathfrak{W}^{R,+}_{j,s}$ is a complement of $T (\mathrm{U}^{\mathrm{2N_R}, +}_{j,s})$, and that $\mathfrak{W}^{R,+}_{j,s}$ varies smoothly with $s$. We define the complement $\mathfrak{W}^{R,-}_{j,s}$ over $\mathrm{U}^{\mathrm{2N}_R, -}_{j,s}$ in an identical way.

It is not hard to construct a smooth (in the parameter $s$) family of complements $\mathcal{N}^R_{j,s}$ of $T\mathrm{V}^R_{j,s}$ which satisfy:
\begin{itemize}
    \item $(\mathcal{N}^R_{j,s})_{(a,p)} =(\mathfrak{W}^{R,-}_{j,s})_{(a,p)} $ if $(a,p) \in \mathrm{U}^{\mathrm{2N}_R, -}_{j,s}$,
    \item $(\mathcal{N}^R_{j,s})_{(a,p)} =(\mathfrak{W}^{R,+}_{j,s})_{(a,p)} $ if $(a,p) \in \mathrm{U}^{\mathrm{2N}_R, +}_{j,s}$,
    \item $(\mathcal{N}^R_{j,s})_{(a,p)} =(\mathfrak{N}^R_{j,s})_{(a,p)} $ if $(a,p) \in \big( \mathrm{V}_{j,s} \setminus (\mathrm{U}^{\mathrm{N}_R, +}_{j,s} \cup \mathrm{U}^{\mathrm{N}_R, -}_{j,s}))$.
\end{itemize}

We can identify cylinders which are sufficiently close to $\mathrm{V}^R_{j,s}$ in the $C^\infty$-sense with sections of  $\mathcal{N}^R_{j,s}$. To do this, we need to introduce some auxiliary structures.

We will also denote by $g$ the restriction of the Riemannian metric $g$ to the complements $\mathcal{N}^R_{j,s}$. From the fact that all $\mathrm{V}^R_{j,s}$ are asymptotically cylindrical, it follows that there exists $\upepsilon>0$ such that for every $j\in \{1,\ldots,\mathrm{n}_0\}$ and every $s\in [0,1]$ the map $\Upsilon^R_{j,s}: \mathcal{N}^R_{j,s}(2\upepsilon) \to \R \times Y$ defined by
\begin{equation}
    \Upsilon^R_{j,s}(p,\mathrm{v}) = \mathrm{exp}_p(\mathrm{v})
\end{equation}
is a diffeomorphism onto its image, where $\mathcal{N}^R_{j,s}(\upepsilon)$ is the disk bundle in $\mathcal{N}^R_{j,s}$ composed of vectors with $g$-norm  $<\upepsilon$, and $\mathrm{exp}$ is the exponential map with respect to the Riemannian metric $g$. 

Recall that by Proposition \ref{prop:important}, given $s \in [0,1]$, the cylinders $\mathrm{V}^R_{j,s}$ and $\mathrm{V}^R_{j',s}$ are disjoint for different $j\neq j'$. Because of the asymptotic behaviour of these cylinders and of the Riemannian metric $g$, and because of the  $C^\infty$-compactness of the spaces $(\mathrm{V}^R_{j,s})_{s \in [0,1]}$, we have that if $\upepsilon>0$ is small enough,  then:
\begin{itemize}
    \item for any $s\in [0,1]$ and any $j\neq j'$ we have that $d_g(\mathrm{V}^R_{j,s},\mathrm{V}^R_{j',s}) > 8 \upepsilon$. Here $d_g$ is the distance function induced by $g$. 
\end{itemize}
We are using crucially that for $j\neq j'$ the asymptotic limits of $\mathrm{V}^R_{j,s}$ and $\mathrm{V}^R_{j',s}$ are disjoint.
It follows that
\begin{itemize}
    \item for every $s \in [0,1]$ and  $j\neq j'$ the images $\Upsilon^R_{j,s}(\mathcal{N}^R_{j,s}(\upepsilon))$ and \linebreak $\Upsilon^R_{j',s}(\mathcal{N}^R_{j',s}(\upepsilon))$ are all disjoint. 
\end{itemize}

It follows by the choice of  $g$, the compactness of the spaces $(\mathrm{V}^R_{j,s})_{s \in[0,1]}$ and the fact that $\mathrm{V}^R_{j,s_n} \to \mathrm{V}^R_{j,s}$ in $C^\infty$ as $s_n \to s$, that there exists $\updelta>0$ such that for all $j\in \{1,\ldots,\mathrm{n}_0\}$ and $s \in [0,1]$, 
\begin{equation}
    \mathrm{V}^R_{j,s'} \subset \Upsilon^R_{j,s}(\mathcal{N}^R_{j,s}(\upepsilon))
\end{equation}
if $s' \in (s-\updelta,s+\updelta)$. 
Choose a set of numbers 
\begin{equation} \label{eq:sequence}
    0=s_0<s_1<s_2<\cdots<s_{\mathrm{m}}=1
\end{equation}
such that $|s_l - s_{l-1}| < \frac{\updelta}{4}$ for all $1\leq l \leq \mathrm{m} $.

Fix $l \in \{0,\ldots,\mathrm{m}-1\}$ and $j\in \{1,\ldots,\mathrm{n}_0\}$.
Under the conditions above, we know that for $s\in [s_l,s_{l+1}]$ there exists a unique section $\upsigma^R_{j,l,s}$ of $\mathcal{N}^R_{j,s_l}$ such that $\Upsilon^R_{j,s_l}(\upsigma^R_{j,l,s})= V_{j,s}$. The sections $\upsigma^R_{j,l,s}$ depend smoothly on $s$, and because of our choice of $g$ and the compactness of $(\mathrm{V}^R_{j,s})_{s \in[0,1]}$ in $C^\infty$, it follows that $\upsigma^R_{j,l,\overline{s}^k} \to \upsigma^R_{j,l,s}$ in $C^\infty$ if $\overline{s}^k $ is a sequence of elements in $ [s_l,s_{l+1}]$ converging to $s$. 

Moreover, given $\delta>0$,  we can find $\widetilde{N}^R_\delta$, independent of $j$ and $l$, such that for every $s\in [s_l,s_{l+1}]$, the section $\upsigma^R_{j,l,s}$ restricted to the sets $\mathrm{V}^R_{j,s_l} \cap ([\widetilde{N}^R_\delta,+\infty) \times Y)$ and $\mathrm{V}^R_{j,s_l} \cap ((-\infty,\widetilde{N}^R_\delta] \times Y)$ is $\delta$-close to the zero-section in the $C^\infty$-topology. Given $\widetilde{N}^R_\delta$,  we call the union of the sets $\mathrm{V}^R_{j,s_l} \cap ([\widetilde{N}^R_\delta,+\infty) \times Y)$ and $\mathrm{V}^R_{j,s_l} \cap ((-\infty,\widetilde{N}^R_\delta] \times Y)$ the $\widetilde{N}^R_\delta$-end of $\mathrm{V}_{j,s_l}$. We call the restriction of $\mathcal{N}^R_{j,s_l}(\upepsilon)$ to the $\widetilde{N}^R_\delta$-end of $\mathrm{V}^R_{j,s_l}$ the $\widetilde{N}^R_\delta$-end of $\mathcal{N}^R_{j,s_l}(\upepsilon)$.

Consider for each $p \in \mathrm{V}^R_{j,s_l}$ the path $\upgamma_{j,l}^{R,p}(s)=\upsigma^R_{j,l,s}(p)$,  $s\in [s_{l},s_{l+1}] $. We consider the partially defined vector field $\mathrm{X}^R_{j,s_l}$ on $\mathcal{N}^R_{j,s_l}(\upepsilon)$ defined by the vectors $\dot{\upgamma}^{R,p}_{j,l}(s)$ tangent to the paths $\upgamma_{j,l}^{R,p}$. Because the sections $\upsigma^R_{j,l,s}$ converge uniformly in $s$ to the zero section on the ends of the cylinder, the partially defined time-dependent\footnote{Here the time coordinate is $s$.} vector field $\mathrm{X}^R_{j,l}(s,p,v)$ converges to zero in $C^\infty$ on the ends
of $\mathcal{N}^R_{j,s_l}(\upepsilon)$. By a well-known procedure, we can extend $\mathrm{X}^R_{j,l}$  to a vector field defined on the whole of $\mathcal{N}^R_{j,s_l}(\upepsilon)$ which satisfies:
\begin{itemize}
    \item[a)] $\mathrm{X}^R_{j,l}(s,p,v)$ is tangent to the hypersurfaces $(\Upsilon_{j,s_l}^{R})^{-1}(\{a\} \times Y )$ for $a \geq \mathrm{N}$,
    \item[b)] $\mathrm{X}^R_{j,l}$ converges uniformly to $0$ in $C^\infty$ as $p$ converges to the ends of the cylinder,
    \item[c)] there exists a  compact set $\mathbf{K}^R_{s,l} \subset \mathcal{N}^R_{j,s_l}(\upepsilon)$ such that the support of $\mathrm{X}^R_{j,l}(s,\cdot,\cdot)$ is contained in $\mathbf{K}^R_{s,l}$ for every $s \in [s_{l},s_{l+1}]$.
\end{itemize}
It is clear that  $|\mathrm{X}^R_{j,l}|< \upepsilon$ for the metric $(\Upsilon_{j,s_l}^R)^*g$, and that $\mathrm{X}^R_{j,l}$ converges to zero in $C^\infty$ in the ends of $\mathcal{N}^R_{j,s_l}(\upepsilon)$. This means that for every $\delta'>0$ we can find $\widetilde{N}^R_{\delta'}$ such that $\mathrm{X}^R_{j,s_l} $ is $ \delta'$-close to the zero-section in $C^\infty$ when restricted to the $\widetilde{N}^R_{\delta'}$-end of $\mathcal{N}^R_{j,s_l}(\upepsilon)$. 


We define the time-dependent vector field $\widetilde{\mathrm{X}}^R_{j,l}(s,\cdot,\cdot)$ on $\R \times Y$ by $$\widetilde{\mathrm{X}}^R_{j,l}:=(D\Upsilon^R_{j,s_l})_*\widetilde{\mathrm{X}}^R_{j,l}.$$
In words, $\widetilde{\mathrm{X}}^R_{j,l}$ is obtained by transporting $\mathrm{X}^R_{j,l}$ to $\R \times Y$ using the diffeomorphism $\Upsilon^R_{j,s_l}$. A priori, $\widetilde{\mathrm{X}}^R_{j,l}$ is only defined on the image $\Upsilon^R_{j,s_l}(\mathcal{N}^R_{j,s_l}(\upepsilon))$ of the diffeomorphism $\Upsilon_{j,s_l}$, but since the support of $\mathrm{X}^R_{j,l}$ is contained in the compact set $\mathbf{K}^R_{s,l} \subset \mathcal{N}^R_{j,s_l}(\upepsilon)$, we can define $\widetilde{\mathrm{X}}^R_{j,l}$ to be $0$ in $(\R \times Y) \setminus (\Upsilon^R_{j,s_l}(\mathcal{N}^R_{j,s_l}(\upepsilon)))$, and with this,  $\widetilde{\mathrm{X}}^R_{j,l}$ is a well-defined $C^\infty$-smooth vector field on the whole of $\R \times Y$.
By the choice of the Riemannian metric $g$ and the asymptotic behaviour of the cylinders $\mathrm{V}^R_{j,s_l}$, the maps $\Upsilon^R_{j,s_l}$ are bounded in $C^\infty$. This together with the properties of $\mathrm{X}^R_{s,l}$ implies that the time-dependent vector field $\widetilde{\mathrm{X}}^R_{j,l}$ has the following properties:
\begin{itemize}
    \item[d)] the norm of $\widetilde{\mathrm{X}}^R_{j,l}$ with respect to $g$ is globally bounded by a constant $\mathbf{C}(R)>0$ independent of $s$,
    \item[e)] for every $\updelta'>0$ there is $N^R_{\updelta'}$ such that $\widetilde{\mathrm{X}}^R_{j,l}$ is $\delta'$-close to $0$ in $C^\infty$ when restricted to the $N^R_{\updelta'}$-ends of $\R \times Y$, where the $N^R_{\updelta'}$-ends of $\R \times Y$ are defined to be $\big( (-\infty,-N^R_{\updelta'}] \times Y\big) \cup \big( [N^R_{\updelta'},+\infty) \times Y\big) $, 
    \item[f)] there is $\widehat{\mathbf{N}}^R_{j,l}$ such that on the $\widehat{\mathbf{N}}^R_{j,l}$-ends of $\R \times Y$ the vector field $\widetilde{\mathrm{X}}^R_{s,l}$ is tangent to the hypersurfaces $\{a\} \times Y$.
\end{itemize}

We now define for each $1 \leq l \leq \mathrm{m}$ the vector field 
\begin{equation} \label{def:vf}
\mathbf{X}^R_l(s,\cdot,\cdot) := \sum_{j=1}^{\mathrm{n}_0} \widetilde{\mathrm{X}}^R_{j,l}(s,\cdot,\cdot). 
\end{equation}
Notice that $\mathbf{X}^R_l(s,\cdot,\cdot)$ is defined for $s\in [s_l,s_{l+1}]$, and that for fixed $l$, the images $\Upsilon^R_{j,s_l}(\mathcal{N}^R_{j,s_l}(\upepsilon))$ and $\Upsilon^R_{j',s_l}(\mathcal{N}^R_{j,s_l}(\upepsilon))$ are disjoint for $j\neq j'$. Since the support of $\widetilde{\mathrm{X}}^R_{j,l}$ is in the interior of  $\Upsilon^R_{j,s_l}(\mathcal{N}^R_{j,s_l}(\upepsilon))$, it follows that $\mathbf{X}^R_l$ coincides with each of the vector fields $\widetilde{\mathrm{X}}^R_{j,l} $
on the regions where these vector fields are defined. The flow of $\mathbf{X}^R_l$ is then the composition of the flows of $\widetilde{\mathrm{X}}^R_{j,l} $.

Let $\mathbf{F}^R_{l}:[s_{l},s_{l+1}] \times \R \times Y \to \R \times Y $ be the flow of the vector field $\mathbf{X}^R_l$. Because the norm of the vector field  $\mathbf{X}_l$ is globally bounded, the flow $\mathbf{F}^R_{l}$ is indeed defined for every $s \in [s_l,s_{l+1}]$ and $(a,q) \in \R \times Y $. If we denote by $\mathbf{F}^R_{j,l}$ the flow of $\widetilde{\mathrm{X}}^R_{j,l} $, we know that $\mathbf{F}^R_{l}$ coincides\footnote{Here and below we let $(\mathbf{F}^R_{l})^s(\cdot)=\mathbf{F}^R_{l}(s,\cdot)$ and $(\mathbf{F}^R_{j,l})^s(\cdot)=\mathbf{F}^R_{j,l}(s,\cdot)$ for each $s \in [s_l,s_{l+1}]$. } with $\mathbf{F}^R_{j,l}$ in $\Upsilon^R_{j',s_l}(\mathcal{N}^R_{j,s_l}(\upepsilon))$.

Now, it follows directly from the construction of the vector field $\widetilde{\mathrm{X}}^R_{j,l} $ that for every $s \in [s_l,s_{l+1}]$ we have
\begin{equation}
(\mathbf{F}_{j,l}^R)^s(\mathrm{V}_{j,s_l}) = \mathrm{V}_{j,s}.
\end{equation}
It then follows that for every $1 \leq j \leq \mathrm{n}_0$
\begin{equation} \label{eq:ambiisotop1}
(\mathbf{F}_{j}^R)^s(\mathrm{V}_{j,s_l}) = \mathrm{V}_{j,s}.
\end{equation}
In other words, $((\mathbf{F}^R_{l})^s)_{s \in [s_l,s_{l+1}]}$ is an ambient isotopy of $\R \times Y$ which deforms the cylinder $\mathrm{V}^R_{j,s_l}$ to the cylinder $\mathrm{V}^R_{j,s_{l+1}}$ via the family of cylinders $(\mathrm{V}^R_{j,s})_{s \in [s_l,s_{l+1}]}$.

From the properties d), e) and f) satisfied by the vector fields $\widetilde{\mathrm{X}}^R_{j,l} $ which we glued to form $\mathbf{X}^R_l$ we obtain that 
the maps $(\mathbf{F}_{l}^R)^s$ have the following properties:
\begin{itemize}
\item[g)] there exists $\mathbf{N}''(R)$ independent of $s$ such that $(\mathbf{F}_{l}^R)^s$ preserves the $\mathbf{N}''(R)$-ends of $\R \times Y$, and for $a \geq \mathbf{N}''(R)$ and $a\leq -\mathbf{N}''(R)$ the maps $(\mathbf{F}_{l}^R)^s$ preserve the hypersurfaces $\{a\} \times Y$,
\item[h)] given $\updelta'>0$, there exists $\widetilde{N}^R_{\updelta'}>\mathbf{N}''(R)$, again independent of $s$, such that $(\mathbf{F}^R_{l})^s$ preserves the $\widetilde{N}^R_{\updelta'}$-ends of $\R \times Y$, and $(\mathbf{F}^R_{l})^s$ restricted to these $\widetilde{N}^R_{\updelta'}$-ends of $\R \times Y$ is $\updelta'$-close to the identity in the $C^\infty$-topology.
\end{itemize}

It follows from h) that for each $s\in [s_{l},s_{l+1}]$, if $\mathrm{V}$ is an asymptotically cylindrical half-cylinder in $\R \times Y$ asymptotic positively/negatively to a Reeb orbit $\gamma$ of $\alpha_0$, then its image $(\mathbf{F}_{l}^R)^s(\mathrm{V})$ is an asymptotically cylindrical half-cylinder in $\R \times Y$ asymptotic positively/negatively to the same Reeb orbit $\gamma$ of $\alpha_0$.

Define $\mathbf{F}_R: [0,1] \times \R \times Y \to \R \times Y $ by 
\begin{equation}
    \mathbf{F}_R^s(\cdot) = (\mathbf{F}^R_l)^s(\cdot) \circ (\mathbf{F}^{R}_{l-1})^{s_l}(\cdot) \circ (\mathbf{F}^{R}_{l-2})^{s_{l-1}}(\cdot) \circ \cdots \circ (\mathbf{F}^{R}_{0})^{s_1}(\cdot) \mbox{ for } s \in [s_l,s_{l+1}], 
\end{equation}
where $\mathbf{F}_R^s(\cdot)=\mathbf{F}_R(s,\cdot)$. We think of $(\mathbf{F}_R^s)_{s \in [0,1]}$ as a family of diffeomorphisms of $\R \times Y$. It follows directly from equation \eqref{eq:ambiisotop1} that 
\begin{equation*} 
    \mathbf{F}_R^s(\mathrm{V}^R_{j,0})= \mathrm{V}^R_{j,s}.
\end{equation*}
Since $\mathrm{V}^R_{j,0}= \R \times \gamma^j_0$, what we obtain is 
\begin{equation}
\mathbf{F}_R^s(\R \times \gamma^j_0)= \mathrm{V}^R_{j,s}.
\end{equation}
Moreover, $\mathbf{F}_R^s$ has properties analogous to the properties g) and h) satisfied by $(\mathbf{F}^R_l)^s$. From this it follows that 
\begin{itemize}
    \item[i)]  if $\mathrm{V}$ is an asymptotically cylindrical half-cylinder in $\R \times Y$ positively (negatively) asymptotic to a Reeb orbit $\gamma$ of $\alpha_0$, then its image $\mathbf{F}_R^s(\mathrm{V})$ is an asymptotically cylindrical half-cylinder in $\R \times Y$  positively (negatively) asymptotic to the same Reeb orbit $\gamma$ of $\alpha_0$.
\end{itemize}

As a first application of the construction of the isotopy $(\mathbf{F}_R^s)_{s \in [0,1]}$ we establish the following Lemma. 
\begin{lem} \label{lemma:planesintersect}
 Fix $R>0$ and  let $\widetilde{w}$ be either
\begin{itemize}
    \item[$\heartsuit$] a finite energy plane in $(\R \times Y, \lambda^R_s,\omega^R_s,J^R_s)$ for some $s\in [0,1]$,
    \item[$\spadesuit$] or a finite energy cylinder in $(\R \times Y, \lambda^R_s,\omega^R_s,J^R_s)$ for some $s\in [0,1]$, with one positive puncture asymptotic to a Reeb orbit $\gamma^+ \not\subset \mathcal{L}_0$ and one negative puncture asymptotic to a Reeb orbit $\gamma^-\not\subset \mathcal{L}_0$ such that $\gamma^+ $ and $\gamma^-$ are not freely homotopic in $Y \setminus \mathcal{L}_0$. 
\end{itemize} 
 Then, there is an interior intersection point between $\widetilde{w}$ and $ \cup_{j=1}^{\mathrm{n}_m} \mathrm{V}^R_{j,s} $.
\end{lem}
\proof

We first treat $\heartsuit$.
Let then $\gamma$ be the Reeb orbit that is detected by the unique positive puncture of $\widetilde{w}$. 
The image $(\mathbf{F}_R^s)^{-1}(\widetilde{w})$ is an asymptotically cylindrical disk in $\R \times Y$ which is also asymptotic to $\gamma$. This implies that the projection of $(\mathbf{F}_R^s)^{-1}(\widetilde{w})$ to $Y$ is a disk $\mathcal{D}$ filling $\gamma$. 

Since $\alpha_0$ is hypertight in the complement of $\mathcal{L}_0$, we know that $\mathcal{D}$ must have an interior intersection point with $\mathcal{L}_0$, and it follows that $(\mathbf{F}_R^s)^{-1}(\widetilde{w})$ has an interior intersection point with $\R \times \mathcal{L}_0$. But since $$ \R \times \mathcal{L}_0 = (\mathbf{F}_R^s)^{-1} \big( \cup_{j=1}^{\mathrm{n}_0} \mathrm{V}^R_{j,s} \big),  $$
there is a bijection between the interior intersection points of $(\mathbf{F}_R^s)^{-1}(\widetilde{w})$ with $\R \times \mathcal{L}_0$ and the interior intersection points of $\widetilde{w}$ with $ \cup_{j=1}^{\mathrm{n}_0} \mathrm{V}^R_{j,s} $. We thus conclude that there must exist interior intersection points between $\widetilde{w}$ and $(\cup_{j=1}^{\mathrm{n}_m} \mathrm{V}^R_{j,s} )$.

The case $\spadesuit$ is treated in the same way, and we refer also to Lemma \ref{lem:homot1} where  a more general situation is treated. 
\qed

\subsubsection{Moduli spaces of holomorphic curves asymptotic to orbits in the link complement}\label{sec:moduli_linkcompl}
We define 
\begin{equation} \label{eq:fam-cylinders}
    \mathcal{Z}_s:= \cup_{j=1}^{\mathrm{n}_0} \mathrm{V}^R_{j,s}. 
\end{equation}
 Let $T \in (0,+\infty)$ and $\rho \in \Omega^T_{\alpha_0}(\mathcal{L})$, and let $\gamma_\rho$ be the unique Reeb orbit of $\alpha_0$ in $\rho $. We will now study the moduli space of holomorphic cylinders in the homotopy $(\R \times Y, \lambda^R_s,\omega^R_s,J^R_s)_{s \in [0,1]}$. Namely we let 
\begin{equation} \label{eq:moduli-rho}
  \mathcal{M}(\gamma_\rho, \gamma_\rho,\omega^R_s,J^R_s; \mathcal{Z}_s,s\in [0,1] )  
\end{equation}
be the moduli space whose elements are equivalence classes of pairs $(s,\widetilde{u})$ such that $s\in [0,1]$ and $\widetilde{u}$ is a finite energy holomorphic cylinder in $(\R \times Y, \lambda^R_s,\omega^R_s,J^R_s)$ with one positive puncture and one negative puncture, both asymptotic to $\gamma_\rho$, and that  does not intersect $\mathcal{Z}_s$. Two pairs $(s,\widetilde{u})$ and $(s',\widetilde{u}')$ represent the same element of this moduli space if $s=s'$ and $\widetilde{u}'=\widetilde{u}\circ \uppsi$ for some biholomorphism $\uppsi:\R \times S^1 \to \R \times S^1$.

By the regularity properties of the homotopy  $(J^R_s)_{s \in [0,1]}$,  we know that \linebreak $\mathcal{M}(\gamma_\rho, \gamma_\rho,\omega^R_s,J^R_s; \mathcal{Z}_s,s\in [0,1] ) $ is a one-dimensional manifold with boundary, the only boundary points being the elements in this space with first coordinate equal to $0$ or $1$. We proceed to show the following lemma. 
\begin{lem}\label{modulispace_rho_compact}
Let $\rho \in \Omega_{\alpha_0}(\mathcal{L})$, $\gamma_\rho$ be the unique Reeb orbit of $\alpha_0$ in  $\rho$, and \linebreak $\mathcal{M}=\mathcal{M}(\gamma_\rho, \gamma_\rho,\omega^R_s,J^R_s; \mathcal{Z}_s,s\in [0,1] )$  the moduli space defined above. Then $\mathcal{M}$ is compact, i.e., it coincides with its SFT-compactification.
\end{lem}

\textit{Proof.} We will establish the lemma by first showing that neither bubbling nor breaking can occur for a sequence $(s_n,\widetilde{u}_n)$ of elements of $\mathcal{M}(\gamma_\rho, \gamma_\rho,\omega^R_s,J^R_s; \mathcal{Z}_s,s\in [0,1] )$ that  converges to a holomorphic building. This implies that $(s_n,\widetilde{u}_n)$ converges to a pair $(s,\widetilde{u})$,  where $s=\lim_{n \to +\infty} s_n$ and $\widetilde{u}$ is a finite energy holomorphic cylinder in $(\R \times Y, \lambda^R_s, \omega^R_s, J^R_s)$ with one positive and one negative puncture, both asymptotic to $\gamma_\rho$. We will conclude the proof by showing that $\widetilde{u}$ does not intersect $\mathcal{Z}_s$, which implies that $(s,\widetilde{u}) \in \mathcal{M}(\gamma_\rho, \gamma_\rho,\omega^R_s,J^R_s; \mathcal{Z}_s,s\in [0,1] )$.

\textit{Step 1: Ruling out bubbling.}

We argue by contradiction assuming that bubbling occurs for a sequence $(s_n,\widetilde{u}_n)$ of elements of $\mathcal{M}(\gamma_\rho, \gamma_\rho,\omega^R_s,J^R_s; \mathcal{Z}_s,s\in [0,1] )$ with $s_n$ converging to $s \in [0,1]$, giving rise to a finite energy holomorphic plane $\widetilde{w}$. Then there are three possibilities:
\begin{itemize}
    \item[Case 1)]: $\widetilde{w}$ is a finite energy plane in the symplectization of $C^+ \alpha_0$ endowed with the almost complex structure
    $J^+_0$,
    \item[Case 2)]: $\widetilde{w}$ is a finite energy plane in  $(\R \times Y,\omega^R_s,J^R_s)$,
    \item[Case 3)]: $\widetilde{w}$ is a finite energy plane in the symplectization of $C^- \alpha_0$ endowed with the almost complex structure
    $J^-_0$.
\end{itemize}

\textit{Ruling out Case 1).} In this case, by Recollection \ref{rec:bubbling} we can find a subsequence $s'_n$ of $s_n$, a sequence of real numbers $\mathbf{b}_n$, a sequence of points $z_n \in \R\times S^1$, strictly monotone sequences of positive numbers $\epsilon_n \to 0$ and $K_n \to +\infty$, and a sequence of holomorphic charts  $$\uppsi: B_{K_n}(0) \subset \mathbb{C} \to B_{\epsilon_n}(z_n) \subset \R \times S^1 $$ such that $\widetilde{w}_n:= \mathbf{T}_{-\mathbf{b}_n} \circ \widetilde{u}_{s'_n} \circ \uppsi_n$ converges in $C^\infty_{\rm loc}$ to a finite energy plane $\widetilde{w}$ in the symplectization of $C^+ \alpha_0$. Because  $\widetilde{w}$ is in the symplectization of $C^+ \alpha_0$, we must have that $\mathbf{b}_n$ converges to $+\infty$. 

Since $\alpha_0$ is hypertight in the complement of $\mathcal{L}_0$, there must be an interior intersection point of  $\widetilde{w}$ with $\R \times \mathcal{L}_0$. We let $p \in \mathbb{C}$ be a point such that $\widetilde{w}(p) \in \R \times \mathcal{L}_0$. Since $J^+_0$ is invariant with respect to the $\R$-action on $\R \times Y$, we can assume that $\widetilde{w}(p) \in \{0\} \times \mathcal{L}_0$. For this, we might have to change the sequence $\mathbf{b}_n$, but the new sequence $\mathbf{b}_n$ will still satisfy $$\mathbf{b}_n \to +\infty.$$

Take $\upepsilon>0$ small enough so that $$\widetilde{w}(B_p(\upepsilon)) \subset [1,1] \times Y,$$ where  $B_p(\upepsilon)$ is the closed ball of radius $\upepsilon$ centered at $p$ in $\mathbb{C}$. Since the images of the finite energy holomorphic curves $\widetilde{w}$ and $\R \times \mathcal{L}_0$ do not coincide, we know that they have finitely many intersection points. Therefore, by choosing $\upepsilon>0$ even smaller we can guarantee that $$\widetilde{w}(\partial B_p(\upepsilon)) \cap [-1-\upepsilon,1+\upepsilon]\times \mathcal{L}_0 = \emptyset $$ and  $$\widetilde{w}( B_p(\upepsilon)) \cap \partial ([-1-\upepsilon,1+\upepsilon]\times \mathcal{L}_0) = \emptyset.$$ 
By positivity of intersections, the algebraic intersection number of the surfaces $\widetilde{w}|_{B_p(\upepsilon)}$ and $[-1-\upepsilon,1+\upepsilon]\times \mathcal{L}_0$ is positive. Since the boundaries of each of these two surfaces are isolated from the other surface, it follows that sufficiently small $C^\infty$-perturbations of  $\widetilde{w}|_{B_p(\upepsilon)}$ and $[-1-\upepsilon,1+\upepsilon]\times \mathcal{L}_0$ must intersect in a point that is in their interior.

Note that the asymptotic convergence of the elements of $\mathcal{Z}_s$ near the punctures is uniform. It follows that $$\mathbf{T}_{-\mathbf{b}_n}\big( (\cup_{j=1}^{\mathrm{n}_0} \mathrm{V}^R_{j,s_n} ) \cap ([-1-\upepsilon,1+\upepsilon]\times Y) \big)$$
converges in $C^\infty$ to $[-1-\upepsilon,1+\upepsilon]\times \mathcal{L}_0$. Similarly, we know that $\widetilde{w}_n|_{B_p(\upepsilon)}$ converges in $C^\infty$ to $\widetilde{w}|_{B_p(\upepsilon)}$. From this and our previous discussion, we conclude that there exists $n_0$ such that for all $n \geq n_0$ the surfaces $\mathbf{T}_{-\mathbf{b}_n}\big( (\cup_{j=1}^{\mathrm{n}_0} \mathrm{V}^R_{j,s_n} ) \cap ([-1-\upepsilon,1+\upepsilon]\times Y) \big)$ and $\widetilde{w}_n(B_p(\upepsilon))$ must have an  intersection point located at the interior of both surfaces. This immediately implies that $(\cup_{j=1}^{\mathrm{n}_0} \mathrm{V}^R_{j,s_n} ) \cap ([-1-\upepsilon,1+\upepsilon]\times Y)$ and $\mathbf{T}_{\mathbf{b}_n}\circ\widetilde{w}_n(B_p(\upepsilon))$ have an intersection point. 

Since $\mathbf{T}_{\mathbf{b}_n}\circ\widetilde{w}_n(B_p(\upepsilon))$ is contained in the image of $\widetilde{u}_n$, we conclude that the image of $\widetilde{u}_n$ intersects $(\cup_{j=1}^{\mathrm{n}_0} \mathrm{V}^R_{j,s_n} ) $. But, by assumption, the $\widetilde{u}_n$ are disjoint from $(\cup_{j=1}^{\mathrm{n}_0} \mathrm{V}^R_{j,s_n} ) $, since $(s_n,\widetilde{u}_n)$ is a sequence of elements of $\mathcal{M}(\gamma_\rho, \gamma_\rho,\omega^R_s,J^R_s; \mathcal{Z}_s,s\in [0,1] )$. This gives the promised contradiction.

\textit{Ruling out Case 2).} In this case, by Recollection \ref{rec:bubbling} we can find a subsequence $s'_n$ of $s_n$, a sequence of points $z_n \in \R\times S^1$, strictly monotone sequences of positive numbers $\epsilon_n \to 0$ and $K_n \to +\infty$, and a sequence of holomorphic charts  $$\uppsi: B_{K_n}(0) \subset \mathbb{C} \to B_{\epsilon_n}(z_n) \subset \R \times S^1 $$ such that $\widetilde{w}_n:=  \widetilde{u}_{s'_n} \circ \uppsi_n$ converges in $C^\infty_{\rm loc}$ to a finite energy plane $\widetilde{w}$ in  $(\R \times Y,\omega^R_s,J^R_s)$. Applying Lemma \ref{lemma:planesintersect}, it follows that $\widetilde{w}$ must intersect $\cup_{j=1}^{\mathrm{n}_0} \mathrm{V}^R_{j,s}$. Using positivity and stability of intersections for holomorphic curves, we obtain from this that for sufficiently large $n$ the images of $\widetilde{w}_n$ must intersect  $\cup_{j=1}^{\mathrm{n}_0} \mathrm{V}^R_{j,s_n}$. But this contradicts the fact that $(s_n,\widetilde{u}_n)$ is a sequence of elements of \linebreak $\mathcal{M}(\gamma_\rho, \gamma_\rho,\omega^R_s,J^R_s; \mathcal{Z}_s,s\in [0,1] )$. This rules out Case 2).

\textit{Ruling out Case 3).} The argument to rule out Case 3) is identical to the argument to rule out Case 1), except that in this case the sequence $\mathbf{b}_n$ converges to $-\infty$. We leave it as an exercise to the reader to do the necessary straightforward adaptations.

\medskip

\textit{Step 2: Ruling out breaking.} 

The SFT-compactness Theorem of \cite{CPT}, see Section \ref{sec:SFT1}, then implies that there exists a subsequence $(s'_n,\widetilde{u}_{s'_n})$ of $(s_n,\widetilde{u}_{s_n})$ such that $s'_n \to s \in [0,1] $ and $\widetilde{u}_{s'_n}$ converges to a finite energy holomorphic building $\overline{u}$ in  $(\R \times Y, \lambda^R_s, \omega^R_s,J^R_s)$. 
By Step 1, we know that the gradients of the holomorphic curves $\widetilde{u}_{s'_n}$ are uniformly bounded, and the results of \cite{CPT,Hofer93} imply that every level of the building $\overline{u}$ is formed by exactly one non-trivial finite energy holomorphic cylinder. We thus have that $\overline{u}$ is a broken holomorphic cylinder and has the following structure:
\begin{itemize}
    \item there are non-negative integers $\kappa^+$ and $\kappa^-$ such that \linebreak $\overline{u}= (\overline{u}^{\kappa^+},\overline{u}^{\kappa^+-1},\ldots,\overline{u}^1,\overline{u}^0,\overline{u}^{-1},\ldots,\overline{u}^{-\kappa^-} )$,
    \item for each $0<j \leq \kappa^+$, $\overline{u}^j$ is a finite energy holomorphic cylinder in the symplectization $(\R \times Y,J^+_0)$ of $C^+\alpha_0$ that is not a trivial cylinder and that has one positive and one negative puncture,
    \item $\overline{u}^0$ is finite energy holomorphic cylinder in the exact symplectic cobordism $(\R \times Y,\lambda^R_s, \omega^R_s,J^R_s)$ with one positive and one negative puncture,
    \item for each $0>j \geq -\kappa^-$, $\overline{u}^j$ is a finite energy holomorphic cylinder in the symplectization $(\R \times Y,J^-_0)$ of $C^-\alpha_0$ that  is not a trivial cylinder and that  has one positive and one negative puncture,
    \item for each $ -\kappa^- < j \leq \kappa^+ $ we have that the Reeb orbit $\gamma^{j-1}_+$ detected by the positive puncture of $\widetilde{u}^{j-1}$ coincides with the Reeb orbit $\gamma^j_-$ detected  by the negative puncture of $\overline{u}^j$,
    \item the positive puncture of $\overline{u}^{\kappa^+}$ and the negative puncture of $\overline{u}^{-\kappa^-}$ detect the Reeb orbit $\gamma_\rho$.
\end{itemize}
To rule out breaking we must show that $\kappa^+ = \kappa^-=0$.

We first assume that $\kappa^+>0$ and show that this leads to a contradiction. In this case,  the top level $\overline{u}^{\kappa^+}$ is a finite energy cylinder in the symplectization of  $(\R \times Y,J^+_0)$ of $C^+\alpha_0$ that is positively asymptotic to $\gamma_\rho$. Since $\overline{u}^{\kappa^+}$ is not a trivial cylinder, the negative puncture of $\overline{u}^{\kappa^+}$ must be asymptotic to a Reeb orbit $\gamma$ of $\alpha_0$ that is different from $\gamma_\rho$. We thus have two possibilities: either $\gamma$ is contained in $\mathcal{L}_0$ or it is not.

We claim that  $\overline{u}^{\kappa^+}$ does not intersect $\R \times \mathcal{L}_0$. Otherwise, positivity and stability of  isolated intersections of holomorphic curves would imply that for sufficiently large $n$ there is an intersection between the image of $\widetilde{u}_n$ and  $\cup_{j=1}^{\mathrm{n}_0} \mathrm{V}^R_{j,s_n}$. But this is prohibited by assumption.

Assume that $\gamma$ is not contained in $\mathcal{L}_0$. Then, since $\gamma_\rho$ is the unique orbit in the free homotopy class of loops $\rho$ in $Y \setminus \mathcal{L}_0$, the projection of the cylinder $\overline{u}^{\kappa^+}$ to $Y$ must intersect $\mathcal{L}_0$. But this would imply that $\overline{u}^{\kappa^+}$ intersects $\R \times \mathcal{L}_0$, contradicting our previous discussion.

We must then have that $\gamma$ is contained in $\mathcal{L}_0$. But the fact that $\rho$ is a proper link class implies that any cylinder whose boundary is the union of a loop in $\rho$ and a cover of a component
of $\mathcal{L}_0$ must have an interior intersection point with $\mathcal{L}_0$. The compactification of the  projection of $\overline{u}^{\kappa^+}$ to $Y$ is a cylinder with these properties and must therefore have an interior intersection point with $\mathcal{L}_0$. But this would imply that $\overline{u}^{\kappa^+}$ intersects $\R \times \mathcal{L}_0$, contradicting our previous discussion.
We thus conclude that $\kappa^+=0$.

The proof that $\kappa^-=0$ is essentially identical. One obtains a contradiction by analysing the bottom level $\overline{u}^{-\kappa^-}$ following the same strategy as above. We leave it to the reader to make the necessary straightforward adaptations.

\medskip 

\textit{Step 3: End of the proof.}

From steps 1 and 2 we conclude that any sequence of elements of \linebreak $\mathcal{M}(\gamma_\rho, \gamma_\rho,\omega^R_s,J^R_s; \mathcal{Z}_s,s\in [0,1] )$ has a subsequence $(s_n,\widetilde{u}_n)$ converging in the SFT-sense to a pair $(s,\widetilde{u})$, where $s=\lim_{n\to +\infty}s_n$ and $\widetilde{u}_n$ converges in $C^\infty$ to a finite energy cylinder  $\widetilde{u}$ in $(\R\times Y, \omega^R_s,J^R_s)$ with one positive and one negative puncture both asymptotic to $\gamma_\rho$. To prove that $\mathcal{M}(\gamma_\rho, \gamma_\rho,\omega^R_s,J^R_s; \mathcal{Z}_s,s\in [0,1] )$  is compact,  it suffices to prove that $(s,\widetilde{u})$ belongs to this moduli space, and this will follow if we show that the image of $\widetilde{u}$ does not intersect $\cup_{j=1}^{\mathrm{n}_0} \mathrm{V}^R_{j,s}$. We argue by contradiction to show that this is indeed the case.

If there are intersection points between $\cup_{j=1}^{\mathrm{n}_0} \mathrm{V}^R_{j,s}$ and the image of $\widetilde{u}$, then the number of intersection points is finite since the image of $\widetilde{u}$ and $\cup_{j=1}^{\mathrm{n}_0} \mathrm{V}^R_{j,s}$ are not contained one in the other. 
It then follows from positivity and stability of intersections for holomorphic curves that the image of  sufficiently $C^\infty$-small perturbations of $\widetilde{u}$ and $\cup_{j=1}^{\mathrm{n}_0} \mathrm{V}^R_{j,s}$ must intersect. But since $\widetilde{u}_n$ and $\cup_{j=1}^{\mathrm{n}_0} \mathrm{V}^R_{j,s_n}$ converge in $C^\infty$ to $\widetilde{u}$ and $\cup_{j=1}^{\mathrm{n}_0} \mathrm{V}^R_{j,s}$, respectively, this would imply that the image of $\widetilde{u}_n$ and $\cup_{j=1}^{\mathrm{n}_0} \mathrm{V}^R_{j,s_n}$ intersect for sufficiently large $n$. This contradicts the fact that $(s_n,\widetilde{u}_n)$ is a sequence of elements of $\mathcal{M}(\gamma_\rho, \gamma_\rho,\omega^R_s,J^R_s; \mathcal{Z}_s,s\in [0,1] )$ and finishes the proof of the lemma.
\qed

\subsection{Finite energy holomorphic buildings in the limit}\label{sec:buildings_limit}
We will now derive some properties of the finite energy holomorphic buildings obtained in the limit as $R\to \infty$ of the finite energy holomorphic cylinders in the exact symplectic cobordisms $(\R \times Y,\lambda^R_1, \omega^R_1,J^R_1)$ that are given by the second endpoints of the connected components of the moduli spaces considered above.  

\subsubsection{Buildings asymptotic to orbits in the link}
Recall that for $R>0$, the finite energy holomorphic cylinder $\widetilde{v}^R_i(1)$ in  $(\R \times Y,\lambda^R_1, \omega^R_1,J^R_1)$ is such that $(1,\widetilde{v}^R_i(1))$ is a boundary element of $\mathcal{I}^R_i$.

Because of the way in which the exact symplectic cobordisms\footnote{As $R \to +\infty$,  this is a stretching family along a contact hypersurface in the sense of \cite{CPT}.} \linebreak $(\R \times Y,\lambda^R_1, \omega^R_1,J^R_1)$  constructed in Section \ref{sec:sympcobordisms} behave as $R \to +\infty$,
we can apply the SFT-compactness theorem of \cite{CPT}, see  Section \ref{sec:SFT2}, and can choose a sequence $R_n \to +\infty$ such that for each $1 \leq i \leq \mathrm{n}_0$ the sequence of holomorphic cylinders $\widetilde{v}^R_i(1)$ converges to a finite energy holomorphic building $\mathbf{v}_i$ of type $|\kappa^+_i|1|\kappa_i|1|\kappa^-_i|$ and arithmetic genus $0$ in the pair of  exact symplectic cobordisms  $[(\R \times Y, \lambda^+_1,\omega^+_1,J^{+\infty}_1),(\R \times Y, \lambda^-_1,\omega^-_1,J^{-\infty}_1)]$. Recall that $(\R \times Y, \lambda^+_1,\omega^+_1,J^{+\infty}_1)$ is an exact symplectic cobordism from $C^+ \alpha_0$ to $ \alpha$ and that  $(\R \times Y, \lambda^-_1,\omega^-_1,J^{-\infty}_1)$ is an exact symplectic cobordism from $\alpha$ to $C^- \alpha_0$.

Since the holomorphic building $\mathbf{v}_i= (S,j,\Gamma, \Delta,L,(\overline{v}^l_i)_{-\kappa^-_i \leq l\leq \kappa^+_i  +\kappa_i +1})$ is the limit of cylinders, the SFT-compactness theorem implies that $\mathbf{v}_i$ is a cylinder with branches. As observed in Section \ref{sec:defbreakingmain},  each level $\overline{v}_i^l$ of $\mathbf{v}_i$ contains exactly two main punctures.
Moreover, the maximum principle implies that one of the two main punctures in each level $\overline{v}^l_i$ of $\mathbf{v}_i$ is positive and that one of them is negative. We then denote by $z^+_{i}(l)$ the unique positive main puncture of the level $\overline{v}^l_i$ and by  $z^-_{i}(l)$ the unique negative main puncture of the level $\overline{v}^l_i$. It is clear from the definition of $\Delta_{\rm main}$ that the Reeb orbits detected by $\overline{v}^l_i$ at  $z^-_{i}(l)$ and by $\overline{v}^{l-1}_i$ at  $z^-_{i}(l-1)$ coincide. We will denote the Reeb orbit detected by $\overline{v}^l_i$  at $z^+_{i}(l)$ by $\gamma^{l,+}_i$, and the Reeb orbit detected by $\overline{v}^l_i$  at $z^-_{i}(l)$ by $\gamma^{l,-}_i$. As we just observed, $$\gamma^{l,-}_i=\gamma^{l-1,+}_i.$$
The top level $\overline{v}^{\kappa^+_i  + \kappa_i +1}_i$ of $\mathbf{v}_i$ has one positive puncture asymptotic to $\gamma_0^i$ and the bottom level $\overline{v}_i^{-\kappa^-_i}$ of $\mathbf{v}_i$ has one negative puncture asymptotic to $\gamma_0^i$. 

Following the terminology introduced in Section \ref{sec:buildings} we will sometimes write
$$\mathbf{v}_i=(\overline{v}^{\kappa^+_i  + \kappa_i +1}_i,\ldots,\overline{v}^{\kappa_i+2}_i,\overline{v}^{\kappa_i+1}_i,\overline{v}^{\kappa_i}_i,\ldots,\overline{v}^{1}_i,\overline{v}^{0}_i,\overline{v}^{-1}_i,\ldots,\overline{v}_i^{-\kappa^-_i}).$$

\begin{lem}\label{lem:v_no_topbot_level}
Let $i \in \{1,\ldots,{\rm n}_0\}$, and let  
 $\mathbf{v_i}$ be the limit of cylinders above. Then, we have that  $\kappa^+_i =0$ and $\kappa^-_i=0$. Furthermore, the top level of $\mathbf{v}_i$ is a finite energy holomorphic curve $\overline{v}_i^{\kappa_i + 1}$ in $(\R \times Y, \lambda^+_1,\omega^+_1,J^{+\infty}_1)$ with one positive puncture asymptotic to $\gamma^i_0$ and at least one negative puncture, and the bottom level is $\overline{v}_i^{0}$  with  one negative puncture asymptotic to $\gamma^i_0$ and at least one positive puncture. 
\end{lem}
\textit{Proof.} The first statement of the lemma implies the second by the structure of the buildings $\mathbf{v}_i$.
We thus fix $i \in \{1,\ldots,{\rm n}_0\}$ and prove the first statement of the lemma for $i$.

We need the following observation regarding the actions of the orbits $\gamma^{l,+}_i$. We claim that for $l \in \{-\kappa^-_i,\ldots,\kappa^+_i + \kappa_i+1\} $ the action of the Reeb orbit $\gamma^{l,+}_i$ detected by the unique positive main puncture $z^+_{i}(l)$ of $\overline{v}^l_i$ satisfies:
\begin{align} \label{eq:ineqacbuilding}
& C^-\mathcal{A}_{\alpha_0}(\gamma_0^i) \leq C^+\mathcal{A}_{\alpha_0}(\gamma^{l,+}_i)   \leq C^+\mathcal{A}_{\alpha_0}(\gamma_0^i) \mbox{ for } \kappa_i^+ + \kappa_i+1 \geq l \geq \kappa_i+1, \\
\nonumber & C^-\mathcal{A}_{\alpha_0}(\gamma_0^i) \leq \mathcal{A}_{\alpha}(\gamma^{l,+}_i)   \leq C^+\mathcal{A}_{\alpha_0}(\gamma_0^i) \mbox{ for } \kappa_i \geq l\geq 0 , \\ 
\nonumber & C^-\mathcal{A}_{\alpha_0}(\gamma_0^i) \leq C^-\mathcal{A}_{\alpha_0}(\gamma^{l,+}_i)   \leq C^+\mathcal{A}_{\alpha_0}(\gamma_0^i) \mbox{ for } -1\geq l\geq -\kappa^-_i .
\end{align}
Indeed, we first observe that since the compactified domain of $\mathbf{v}_i$ is homeomorphic to a compact cylinder, we must have that $z^+_{i}(l)$ and $z^-_{i}(l)$ belong to the same connected component $\overline{v}^l_i({\rm main})$ of $\overline{v}^l_i$. Using the exactness of the symplectic manifolds on which the levels of $\mathbf{v}_i$ live, we can conclude that the action of $\gamma^{l,+}_i$ satisfies
\begin{align}\label{eq:eqactionsbuilding0}
    & C^+\mathcal{A}_{\alpha_0}(\gamma^{l,+}_i) \geq C^+\mathcal{A}_{\alpha_0}(\gamma^{l,-}_i) \mbox{ for } \kappa_i^+ + \kappa_i+1 \geq l > \kappa_i+1, \\
    \nonumber &  C^+\mathcal{A}_{\alpha_0}(\gamma^{\kappa_i + 1,+}_i) \geq \mathcal{A}_{\alpha}(\gamma^{\kappa_i + 1,-}_i), \\
    \nonumber & \mathcal{A}_{\alpha}(\gamma^{l,+}_i) \geq \mathcal{A}_{\alpha}(\gamma^{l,-}_i) \mbox{ for }  \kappa_i \geq l>0, \\
    \nonumber & \mathcal{A}_{\alpha}(\gamma^{0,+}_i) \geq C^-\mathcal{A}_{\alpha_0}(\gamma^{0,-}_i), \\
    \nonumber & C^-\mathcal{A}_{\alpha_0}(\gamma^{ l,+}_i) \geq C^-\mathcal{A}_{\alpha_0}(\gamma^{l,-}_i), \mbox{ for }  0 >l\geq -\kappa^-_i.
\end{align}
We introduce the names $(\uptheta_{l})_{l \in \{-\kappa^-_i-\kappa_i-1,\ldots,\kappa^+_i\} }$ for the modified actions 
\begin{align} \label{eq:eqactionsbuilding1}
  &  \uptheta_{l} = C^+\mathcal{A}_{\alpha_0}(\gamma^{l,+}_i) \mbox{ for } \kappa_i^+ + \kappa_i+1 \geq l \geq \kappa_i+1, \\
  & \nonumber   \uptheta_{l} = \mathcal{A}_{\alpha}(\gamma^{l,+}_i) \mbox{ for } \kappa_i \geq l\geq 0, \\
  & \nonumber \uptheta_{l} =  C^-\mathcal{A}_{\alpha_0}(\gamma^{ l,+}_i)  \mbox{ for } -1\geq l\geq -\kappa^-_i.
\end{align}
 It follows from \eqref{eq:eqactionsbuilding1} and the fact that $\gamma^{l,-}_i= \gamma^{l-1,+}_i,$ that $\uptheta_{l}$ is non-decreasing as $l$ increases.
For the top level $\overline{v}^{\kappa_i^+ + \kappa_i+1}_i$ of $\mathbf{v}_i$ the orbit detected by the main positive puncture is $\gamma_0^i$ and its action is $C^+\mathcal{A}_{\alpha_0}(\gamma_0^i)$. This implies that 
\begin{equation}\label{eq:eqactionbuilding2}
    \uptheta_{\kappa_i^+ + \kappa_i+1}  = C^+\mathcal{A}_{\alpha_0}(\gamma_0^i).
\end{equation}

For the bottom level $\overline{v}^{-\kappa^-_i }_i$ of $\mathbf{v}_i$ the orbit detected by the negative puncture is $\gamma_0^i$. Again, because of the exactness of the symplectic manifolds on which the levels of $\mathbf{v}_i$ live, it follows that the action $\uptheta_{-\kappa^-_i}$ of the main positive puncture of $\overline{v}^{-\kappa^-_i}_i$ satisfies 
\begin{equation} \label{eq:eqactionbuilding3}
 \uptheta_{-\kappa^-_i } \geq   C^-\mathcal{A}_{\alpha_0}(\gamma_0^i).
\end{equation}

Combining \eqref{eq:eqactionsbuilding0}, \eqref{eq:eqactionsbuilding1}, \eqref{eq:eqactionbuilding2}, and \eqref{eq:eqactionbuilding3}, we obtain that  $\uptheta_{l} \in [C^-\mathcal{A}_{\alpha_0}(\gamma_0^i),C^+\mathcal{A}_{\alpha_0}(\gamma_0^i)]$ for all $l \in \{-\kappa^-_i,\ldots,\kappa^+_i +\kappa_i+1\}$, which establishes \eqref{eq:ineqacbuilding}.

To prove that $\kappa_i^+=0$, we argue by contradiction and assume that $\kappa^+_i>0$. We then consider the top level $\overline{v}^{\kappa^+_i + \kappa_i +1}_i$ of $\mathbf{v}_i$ whose target is the symplectization of $C^+\alpha_0$. 
Since $\gamma^{l,-}_i= \gamma^{l-1,+}_i,$ we have that $$C^+\mathcal{A}_{\alpha_0}(\gamma^{\kappa^+_i + \kappa_i +1,-}_i)= \uptheta_{\kappa^+_i + \kappa_i } \geq C^-\mathcal{A}_{\alpha_0}(\gamma^i_0).$$
Moreover, by exactness of the symplectization, we have that $$C^+\mathcal{A}_{\alpha_0}(\gamma^{\kappa^+_i + \kappa_i +1,-}_i)\leq C^+\mathcal{A}_{\alpha_0}(\gamma^i_0).$$ 
We conclude that $\gamma^{\kappa^+_i + \kappa_i +1,-}_i$ is a Reeb orbit of $\alpha_0$ that  satisfies 
\begin{equation} \label{eq:endofproof0}
    \mathcal{A}_{\alpha_0}(\gamma^{\kappa^+_i + \kappa_i +1,-}_i) \in [\frac{C^-}{C^+}\mathcal{A}_{\alpha_0}(\gamma^i_0),\mathcal{A}_{\alpha_0}(\gamma^i_0)].
\end{equation}
Recall that by Condition A) on the numbers $\delta_{\alpha_0,\mathcal{L}_0}$ chosen in Assumption \ref{assump:delta},  we have that 
\begin{equation} \label{eq:endofproof1}
  \frac{e^{ -\delta_{\alpha_0,\mathcal{L}_0} }}{e^{ \delta_{\alpha_0,\mathcal{L}_0} }} \mathcal{A}_{\alpha_0}(\gamma^i_0) > \mathcal{A}_{\alpha_0}(\gamma^i_0)- \mu_{\alpha_0,\mathcal{L}_0},  
\end{equation}
for $\mu_{\alpha_0,\mathcal{L}_0}$ satisfying  \eqref{eq:munew}. Since $C^->e^{ -\delta_{\alpha_0,\mathcal{L}_0}} $ and  $C^+ < e^{ \delta_{\alpha_0,\mathcal{L}_0}} $ we clearly have 
$$\frac{C^-}{C^+}\mathcal{A}_{\alpha_0}(\gamma^i_0) > \frac{e^{ -\delta_{\alpha_0,\mathcal{L}_0} }}{e^{ \delta_{\alpha_0,\mathcal{L}_0} }} \mathcal{A}_{\alpha_0}(\gamma^i_0) > \mathcal{A}_{\alpha_0}(\gamma^i_0)- \mu_{\alpha_0,\mathcal{L}_0}.$$
Combining this with \eqref{eq:endofproof0},  we conclude that 
\begin{equation}
 \mathcal{A}_{\alpha_0}(\gamma^{\kappa^+_i + \kappa_i +1,-}_i) \in (\mathcal{A}_{\alpha_0}(\gamma^i_0) - \mu_{\alpha_0,\mathcal{L}_0},\mathcal{A}_{\alpha_0}(\gamma^i_0)].
 \end{equation}
 Since $(\Upsilon^+)^{-1} \circ \overline{v}_i^{\kappa^+_i + \kappa_i +1,-}$ is a holomorphic cylinder in the symplectization  $(\R \times Y,\widetilde{J}_0)$ asymptotic to $\gamma^i_0$ and $\gamma^{\kappa^+_i + \kappa_i +1,-}_i$, this implies by \eqref{eq:munew} that $\overline{v}_i^{\kappa^+_i + \kappa_i +1,-}$ has to be the  trivial cylinder, contradicting the fact the every level of $\mathbf{v}_i$ inside a symplectization is different from a trivial cylinder. This contradiction establishes that $\kappa_i^+=0$.

To prove that $\kappa_i^-=0$,  we argue by contradiction and assume that $\kappa^-_i>0$.  We consider the bottom level $\overline{v}^{-\kappa^-_i}_i$. Using again \eqref{eq:ineqacbuilding} we obtain that 
\begin{equation}
    \mathcal{A}_{\alpha_0}(\gamma^{-\kappa^-_i,+}_i) \in [\mathcal{A}_{\alpha_0}(\gamma^i_0),\frac{C^+}{C^-}\mathcal{A}_{\alpha_0}(\gamma^i_0)].
\end{equation}
Recall that by Condition A) on the numbers $\delta_{\alpha_0,\mathcal{L}_0}$ chosen in Assumption \ref{assump:delta}, we have that 
\begin{equation} \label{eq:endofproof5}
  \frac{e^{ \delta_{\alpha_0,\mathcal{L}_0} }}{e^{ -\delta_{\alpha_0,\mathcal{L}_0} }} \mathcal{A}_{\alpha_0}(\gamma^i_0) < \mathcal{A}_{\alpha_0}(\gamma^i_0)+ \mu_{\alpha_0,\mathcal{L}_0},  
\end{equation}
for $\mu_{\alpha_0,\mathcal{L}_0}$ satisfying \eqref{eq:munew}. Since $C^->e^{ -\delta_{\alpha_0,\mathcal{L}_0}} $ and  $C^+ < e^{ \delta_{\alpha_0,\mathcal{L}_0}}$, we clearly have 
$$\frac{C^+}{C^-}\mathcal{A}_{\alpha_0}(\gamma^i_0) < \frac{e^{ \delta_{\alpha_0,\mathcal{L}_0} }}{e^{- \delta_{\alpha_0,\mathcal{L}_0} }} \mathcal{A}_{\alpha_0}(\gamma^i_0) < \mathcal{A}_{\alpha_0}(\gamma^i_0)+ \mu_{\alpha_0,\mathcal{L}_0}.$$
This, together with \eqref{eq:endofproof5}, gives 
\begin{equation} \label{eq:endofproof6}
 \mathcal{A}_{\alpha_0}(\gamma^{-\kappa^-_i ,+}_i) \in [\mathcal{A}_{\alpha_0}(\gamma^i_0),\mathcal{A}_{\alpha_0}(\gamma^i_0) + \mu_{\alpha_0,\mathcal{L}_0}).
 \end{equation}

Since $\mathbf{v}_i$ is a cylinder with branches, we have that the bottom level $\overline{v}^{-\kappa^-_i}_i$ of $\mathbf{v}_i$ has a unique negative puncture asymptotic to $\gamma_0^i$, and that every connected component of it has genus $0$. Moreover,  the maximum principle implies that each connected component of $\overline{v}^{-\kappa^-_i}_i$ has exactly one positive puncture.

Consider the connected component $w^{-\kappa^-_i,i}_{\rm main}$ of $\overline{v}^{-\kappa^-_i}_i$ that  contains the two main punctures of $\overline{v}^{-\kappa^-_i}_i$. By \eqref{eq:endofproof6}, and since $(\Upsilon^{-})^{-1} \circ w^{-\kappa^-_i,i}_{\rm main}$ is negatively asymptotic to  $\gamma_0^i$ and positively asymptotic to $\gamma^{-\kappa^-_i ,+}_i$, the choice of $\mu_{\alpha_0,\mathcal{L}_0}$ implies that $w^{-\kappa^-_i,i}_{\rm main}$ is the trivial cylinder over $\gamma_0^i$.

For each $l \in \{\kappa^-_i,\ldots, \kappa^+_i +\kappa_i +1 \}$ let $\upsigma^+_l$ denote the sum of the actions\footnote{Recall that the action of a puncture of a finite energy holomorphic curve is defined as the action of the Reeb orbit detected by the curve at this puncture.} of the positive punctures of $\overline{v}^{l}_i$. Similarly, let $\upsigma^-_l$ denote the sum of the actions of the negative punctures of $\overline{v}^{l}_i$. By the exactness of the targets of the levels of $\mathbf{v}_i(1)$, we obtain that for all $l \in \{\kappa^-_i,\ldots, \kappa^+_i +\kappa_i +1 \}$ $$\upsigma^-_l\leq \upsigma^+_l.$$ For $l \in \{\kappa^-_i+1,\ldots, \kappa^+_i +\kappa_i +1 \}$ we have a one to one correspondence between the  negative punctures of $\overline{v}^{l}_i$ and the positive punctures of $\overline{v}^{l-1}_i$, and the Reeb orbits detected by the punctures identified by this correspondence coincide. This implies that $$\upsigma^-_l = \upsigma^+_{l-1}.$$ Moreover, we know that 
\begin{align}
    & \upsigma^-_{-\kappa^-_i}= C^-\mathcal{A}_{\alpha_0}(\gamma_0^i), \\
    & \upsigma^-_{\kappa^+_i=0}= C^+\mathcal{A}_{\alpha_0}(\gamma_0^i).
\end{align}
An argument similar as the one used to prove \eqref{eq:ineqacbuilding} then implies that the finite sequence $\upsigma^+_l$ is non-decreasing in $\in \{\kappa^-_i,\ldots, \kappa^+_i +\kappa_i +1 \}$, and that 
\begin{equation}\label{eq:endofproof7}
 \upsigma^+_l \in [C^- \mathcal{A}_{\alpha_0}(\gamma^i_0),C^+ \mathcal{A}_{\alpha_0}(\gamma^i_0)].   
\end{equation}

By our discussion above,  $w^{-\kappa^-_i,i}_{\rm main}$ is the trivial cylinder over the Reeb orbit $\gamma^i_0$.  
Since $\overline{v}^{-\kappa^-_i}_i$ is not a trivial cylinder, 
it must have at least one other component $w$ which must be a finite energy plane with one positive puncture. The Reeb orbit $\gamma'$ detected by $w$ must have action $\geq \mathrm{sys}(\alpha_0)$. It follows that $$\upsigma^+_{-\kappa^-_i}\geq C^- \mathcal{A}_{\alpha_0}(\gamma^i_0) + \mathrm{sys}(\alpha_0). $$ But since, by our choice of $\mu_{\alpha_0,\mathcal{L}_0}>0$, we have $\mathrm{sys}(\alpha_0) > 3\mu_{\alpha_0,\mathcal{L}_0} $, and \linebreak $C^- \mathcal{A}_{\alpha_0}(\gamma^i_0)  > \mathcal{A}_{\alpha_0}(\gamma^i_0)  - \mu_{\alpha_0,\mathcal{L}_0}$, we obtain that $$C^- \mathcal{A}_{\alpha_0}(\gamma^i_0) + \mathrm{sys}(\alpha_0) > \mathcal{A}_{\alpha_0}(\gamma^i_0) + 2\mu_{\alpha_0,\mathcal{L}_0},$$ and therefore  
$$\upsigma^+_{-\kappa^-_i} > \mathcal{A}_{\alpha_0}(\gamma^i_0)+ 2\mu_{\alpha_0,\mathcal{L}_0}> C^+\mathcal{A}_{\alpha_0}(\gamma^i_0).$$
This contradicts \eqref{eq:endofproof7}. This contradiction proves that  $\kappa^-_i=0$. This finishes the proof of the lemma.
\qed

\subsubsection{Buildings asymptotic to orbits in the link complement}
Having established Lemma \ref{lem:v_no_topbot_level}, we are now in the situation to show an analogous statement considering instead the moduli spaces $\mathcal{M}(\gamma_{\rho},\gamma_{\rho}, \omega_s^R, J_s^R;\mathcal{Z}_s,s\in [0,1])$, $\rho\in \Omega_{\alpha_0}(\mathcal{L}_0)$. 
Let $\rho$ be a class in $\Omega_{\alpha_0}(\mathcal{L}_0)$. 
By the choice of $(J^R_s)_{s\in [0,1]}$ and by Lemma \ref{modulispace_rho_compact}, we have for any $R>0$ that the moduli space $\mathcal{M}^R_{\rho}=\mathcal{M}(\gamma_{\rho},\gamma_{\rho}, \omega_s^R, J_s^R;\mathcal{Z}_s,s\in [0,1])$ is a compact one-dimensional manifold.  There exists a connected component of $\mathcal{M}^R_{\rho}$ that contains a boundary point of the form   $(0,\widetilde{u}^0_{\gamma_{\rho}})$, where $\widetilde{u}^0_{\gamma_{\rho}}$ is a trivial cylinder over $\gamma_{\rho}$. Let $(1,\widetilde{u}_{\rho}^R(1))$ be the second boundary point of that component. As in the discussion above for the holomorphic cylinders $\widetilde{v}_i^R(1)$, the holomorphic cylinders $\widetilde{u}_{\rho}^R(1)$ converge as $R\to +\infty$ in the SFT-sense to a finite energy holomorphic building $\mathbf{u}_{\rho}$ of type $|\kappa_{\rho}^+|1|\kappa_{\rho}|1|\kappa_{\rho}^-|$ and of arithmetic genus $0$ in the pair of exact symplectic cobordisms $[(\R\times Y,\lambda_1^+,\omega_1^+,J_1^{+\infty}),(\R\times Y,\lambda_1^-,\omega_1^-,J_1^{-\infty})]$. As in the discussion above, $\mathbf{u}_{\rho}$ is a cylinder with branches. Again we also  use the notation     
$$\mathbf{u}_{\rho}=(\overline{u}_{\rho}^{\kappa_{\rho}^+  + \kappa_{\rho} +1},\ldots,\overline{u}_{\rho}^{\kappa_{\rho}+2},\overline{u}_{\rho}^{\kappa_{\rho}+1},\overline{u}_{\rho}^{ \kappa_{\rho}},\ldots, \overline{u}_{\rho}^{1},\overline{u}_{\rho}^{0},\overline{u}_{\rho}^{-1},\ldots,\overline{u}_{\rho}^{-\kappa_{\rho}^-}).$$

Before we 
 proceed to study the building $\mathbf{u}_{\rho}$, we first collect some consequences of the construction of the family of  diffeomorphisms $\mathbf{F}_R^s:\R \times Y \to \R \times Y$, $s\in [0,1]$, in Section \ref{sec:isotopies}, partly generalizing Lemma \ref{lemma:planesintersect}. 
  Let $V_{1}^R$ be the union of $V_{j,1}^R$, $j=1, \ldots, {\rm n}_0$, where, as above, $V_{j,1}^R$ is the image of 
 $\widetilde{v}_j^R(1)$ in $\R \times Y$.  
Recall that we say that two loops $\beta_0, \beta_1:S^1 \to Y$ are homotopic in the complement of $\mathcal{L}_0$ if there is a homotopy between them such that no interior point maps to $\mathcal{L}_0$, and we say that a loop $\beta:S^1 \to Y$ is contractible in the complement of $\mathcal{L}_0$ if there is a disk map with boundary in $\beta_0$ with no interior intersection point with $\mathcal{L}_0$. 
\begin{lem}\label{lem:homot1} Let $R>0$. Let $\beta, \beta_0$, and $\beta_1$ be loops as above. If there exist a continuous map $H(s,t)=(a(s,t),f(s,t)):\R \times S^1 \to \R \times Y$ and  $s_-,s_+\in \R \cup\{\pm \infty\}$ with $s_-<s_+$ such that  (uniformly)
 $$\lim_{s\to -\infty} a(s,t), \lim_{s\to +\infty} a(s,t) \in \{\pm\infty\},  $$ 
 $$\lim_{s\to -\infty}f(s,t) = \beta_1(t) \, \,  \text{  and } \lim_{s\to +\infty} f(s,t) = \beta_2(t),  \text{ and}$$
 $$H(s,t) \in (\R \times Y) \setminus V^R_1 \text{ if and only if } s\in(s_-,s_+),$$ 
 then  $\beta_0$ and $\beta_1$ are homotopic in the complement of $\mathcal{L}_0$.
 
Moreover, if there exists a continuous map $H(x) = (a(x),f(x)):\D \to \R \times Y$ and $0<r_0\leq 1$ such that 
 $$\lim_{r\to 1}a(re^{2\pi it}) = \pm \infty, \, $$
 $$\lim_{r \to 1} f(re^{2\pi it}) = \beta(t), \, \text{ and }$$
 $$H(x)\in (\R\times Y) \setminus V^R_1 \, \text{ if and only if } |x| < r_0,$$
 then $\beta$ is contractible in the complement of $\mathcal{L}_0$. 
  \end{lem}
\begin{proof}
Define in both cases
$h:=\uppi_2\circ (\mathbf{F}_R^1)^{-1}\circ H$, 
where $\uppi_2: \R \times Y \to Y$ is the projection to the second coordinate. 
We recall that $(\mathbf{F}_R^1)^{-1}(V_1^R) = \R \times \mathcal{L}_0$, and that $(\mathbf{F}_R^1)^{-1}$ is asymptotic to the identity  on $\R \times Y$. 
It follows that $h$ induces a homotopy of $\beta_0$ and $\beta_1$ in $Y\setminus \mathcal{L}_0$ in the first case, and a homotopy in $Y\setminus \mathcal{L}_0$ between $\beta$ and a constant loop in the second case.
\end{proof}
The following Lemma is an analogue of Lemma \ref{lem:v_no_topbot_level} for the building $\mathbf{u}_{\rho}$.
\begin{lem}\label{lem:u_firstlevels} 
For any $\rho \in \Omega_{\alpha_0}(\mathcal{L}_0)$ we  have that $\kappa_{\rho}^+ = \kappa_{\rho}^- = 0$. It follows that  the top  level of $\mathbf{u}_{\rho}$ is $\overline{u}_{\rho}^{\kappa+1}$, a finite energy holomorphic curve with a unique positive puncture  asymptotic to $\gamma_{\rho}$, and that the bottom level of $\mathbf{u}_{\rho}$ is $\overline{u}_{\rho}^{0}$, which has a cylindrical component that has a unique negative puncture asymptotic to $\gamma_{\rho}$. 
\end{lem}
\begin{proof}
In the following we fix  $\rho \in \Omega_{\alpha_0}(\mathcal{L}_0)$, and shorten the notation and write $\widetilde{u}^R = \widetilde{u}_{\rho}^R(1)$, $\widetilde{v}^R_j = \widetilde{v}^R_j(1)$. 
We argue by contradiction and assume first that $\kappa_{\rho}^{+}>0$.  We consider the top level $\hat{u}$ of $\mathbf{u}_{\rho}$.  
Since $\hat{u}$ has only one positive puncture, it follows from the maximum principle that $\hat{u}$  is connected, and we recall that, by our convention, $\hat{u}$ is not a trivial cylinder. Let $m\geq 1$ be the number of the negative punctures of $\hat{u}$. We will now describe, for $R$ sufficiently large, how to glue a map defined on a negative half-cylinder to the map  $\widetilde{u}^{R}|_{Z}$, where $Z\subset \R \times S^1$ is a suitable positive half-cylinder. By this, we will obtain a map $H:\R \times S^1 \to Y$ to which we can apply Lemma \ref{lem:homot1} in order to get a contradiction. We write $\hat{u}: \hat{S}\setminus\{q_0,\ldots, q_{m}\} \to Y$, where $\hat{S}$ is the Riemann sphere,  $q_0\in \hat{S}$ is the unique positive puncture, and $q_1, \ldots, q_{m}\in \hat{S}$ are the negative punctures. We assume that $q_m$ is the main negative puncture and let $\gamma$ be the main Reeb orbit that corresponds to $q_m$.   
Since there is $R_k \to + \infty$ such that the sequence $\widetilde{u}^{R_k}$ converges in the SFT-sense to the building $\mathbf{u}$, there is a sequence of open disks ${D^k_0}$ and  ${D^k_m}$ around $q_0$ and $q_m$ in $\hat{S}$ with $\bigcap_{k\in \N} \overline{D^k_0} = \{q_0\}$ and $\bigcap_{k\in \N} \overline{D^k_m} = \{q_m\}$, a sequence of  biholomorphic maps $\psi_k:([-L_k,L_k] \times S^1) \setminus \{q_1, \ldots, q_{m-1}\} \to \hat{S}\setminus (D^{k}_0 \cup D^k_m)$, and a sequence $\textbf{b}_k$ of real numbers with  $\textbf{b}_k-R_k\to +\infty$  
  such that 
$\hat{u}_k:= \mathbf{T}_{-\textbf{b}_k} \circ \widetilde{u}^{R_k} \circ \psi_k$ converges in $C^{\infty}_{\loc}$ to $\hat{u}$.   It follows that there is $\epsilon>0$ and $a<0$, and, for $k$ sufficiently large, simple closed curves $\beta_k:S^1 \to  ([-L_k,L_k] \times S^1) \setminus \{q_1, \ldots, q_{m-1}\}$  homotopic to $\{-L_k\} \times S^1$ for which the image of $\hat{u}_k\circ \beta_k$ lies entirely in $\{a\} \times Q \subset \{a\} \times Y$,  where $Q$ is the tubular neigbourhood of $\gamma$ of size  $\epsilon$. Moreover, we can assume that $\{a\} \times Q$ does not contain any other image point of $\hat{u}_k$. Here and below, the \textit{tubular neighbourhood} of a simple loop $\gamma$ of \textit{size}  $\epsilon$ is the image of the exponential map restricted to the normal bundle of $\gamma$ of radius $\epsilon$ (with respect to the fixed Riemannian metric $g_0$),  and we implicitly assume that $\epsilon$ is so small that this map is injective. 

For any $j\in\{1, \ldots, {\rm n}_0\}$, consider the sequence of cylindrical holomorphic maps $\hat{v}^j_k:=\mathbf{T}_{-\textbf{b}_k} \circ \widetilde{v}^{R_k}_j$. By the choice of $\textbf{b}_k$ and by Lemma \ref{lem:v_no_topbot_level}, this sequence  converges in $C^{\infty}_{\loc}$ to a trivial cylinder over $\gamma_0^j$ in the symplectization of $C^+\alpha_0$. This means in particular that, for $k$ sufficiently large, there are essential simple closed curves $\beta^j_k: S^1 \to  \R \times S^1$ such that the images of $\hat{v}^j_k \circ \beta^j_k$ lie entirely in  $\{a\} \times Q^j \subset \{a\} \times Y$,  
 where $Q^j$ are tubular neighbourhoods of $\gamma_0^j$ of size $\epsilon$. Moreover, we can assume that $\{a\} \times Y$ does  not contain any other image points of $\hat{v}^j_k$,  that the closed curves $\hat{v}^j_k \circ \beta^j_k$ are simple,  and that $\hat{v}^j_k \circ \beta^j_k$ can be written as the composition of the  exponential map restricted to the normal bundle $\{a\} \times \gamma_0^j$ in $\{a\} \times Y$ and a section of that bundle.

\emph{First case:} Let us 
 assume that $\gamma$ is not a multiple of $\gamma_0^j$ for all $j\in\{1,\ldots, {\rm n}_0\}$. 
If $\epsilon$ was chosen sufficiently small, $Q\cap Q^j = \emptyset$  for all $j\in \{1,\ldots, {\rm n}_0\}$. There is retraction $F:[0,1] \times Q\to Q$ to $\gamma\subset Q$ with $F(t,x) \in \gamma$ if and only if $t=1$ or $x\in \gamma$. Hence, we can find a continuous map $W:[0,1] \times S^1\to \{a\} \times Q$ with $W(1,t) = \hat{u}_k \circ \beta_k(t)$ and such that 
 $t\mapsto W(0,t)$ parametrizes (a multiple cover of) $\{a \} \times \gamma$. 
Fix $k$ sufficiently large. Let $Z$ be the positive half-cylinder in $\R \times S^1$ that is bounded by the essential simple loop $t\mapsto \psi_k \circ \beta_k(t)$, and choose a diffeomorphism $\Psi: [1,+\infty)\times S^1 \to Z$ with $\Psi(1,t) = \psi_k\circ\beta_k(t)$.  Let $\hat{W}:(-\infty,0]\times S^1 \to \R \times Y$ be  defined by $(s,t) \mapsto (a-s,W(1,t))$. 
We glue the above cylinders, shift by $\textbf{b}_k$, and obtain a contiuous map $H:\R \times S^1 \to \R \times Y$ given by  $$H = \mathbf{T}_{\textbf{b}_k} \circ \left(\hat{W} \#_{s=0} W {\#}_{{s=1}}  (\mathbf{T}_{-\textbf{b}_k}\circ \widetilde{u}^{R_k}  \circ \Psi)\right).$$
The image of $H$ is disjoint from the sets  $V^{R_k}_{j,1}$, $j=1, \ldots, {\rm n}_0$, and the two asymptotic orbits are $\gamma_{\rho}$ and a reparametrization of a multiple of $\gamma$, respectively. 
By Lemma \ref{lem:homot1}, $\gamma_{\rho}$ and a multiple of $\gamma$ are homotopic in the complement of $\mathcal{L}_0$, and hence in fact $\gamma = \gamma_{\rho}$.  This is only possible if 
 $\hat{u}$ is a trivial cylinder over $\gamma$, which was excluded in the beginning. \\
\emph{Second case:} Let us assume that $\gamma$ is an  $l$-fold multiple of some $\gamma_0^j$ for some $j\in \{1, \ldots, {\rm n}_0\}$, $l\geq 1$. We can assume that $\epsilon$ was chosen so small that tubular neighbourhoods $\hat{Q}^j$ of size $2\epsilon$ of  the simple curves $\hat{\eta}_j: S^1 \to Y$, $t\mapsto  \uppi_2 \circ \hat{v}^j_k \circ \beta_k^j(t)$, $j=1, \ldots, {\rm n}_0$,  are pairwise disjoint. Using a suitable retraction of $\hat{Q}_j$ to $\Im (\hat{\eta}_j)$,  one can construct a continuous map $W:[0,1] \times S^1 \to \{a\} \times \hat{Q}_j$ such that $W(1,t) = \hat{u}_k \circ \beta_k(t)$, such that $W(0,t)$ is a reparametrization of a $l$-fold multiple cover of $\hat{v}^j_k\circ \beta^j_k$, and such that $W((0,1]\times S^1)$ does not intersect the image of $\hat{v}^j_k\circ \beta^j_k$. 
We consider now an orientation reversing smooth map  $\Psi^j: (-\infty, 0] \times S^1 \to Z^j$, where $Z^j$ is the positive half-cylinder bounded by $\psi_k^j \circ \beta_k^j$ with $\Psi_j(s,t) \to +\infty$ as $s\to -\infty$ and such that $\Psi^j(0,t) = W(0,t)$. 
We define $H:\R \times S^1 \to \R \times Y$ by 
$$H= \mathbf{T}_{\textbf{b}_k} \circ \left((\hat{v}^j_k\circ \Psi^j) \#_{s=0} W \#_{s=1}  (\mathbf{T}_{-\textbf{b}_k} \circ \widetilde{u}^{R_k} \circ \Psi)\right).$$
We have $H(s,t) \in V^{R_k}_{j,1}$ if and only if $s>0$. The  asymptotic orbits of $H$ are $\gamma_{\rho}$ and a reparametrization of a  multiple of $\gamma^j_{0}$, respectively. By Lemma \ref{lem:homot1}, $\gamma_{\rho}$ is homotopic in the complement of $\mathcal{L}_0$ to a multiple of $\gamma_0^j$, a contradiction. 

We argue by contradiction and assume that $\kappa_{\rho}^{-}>0$, that is, there is a bottom level $\hat{u}$ in the symplectization of $C_{-}\alpha_0$. The level $\hat{u}$ might have several components of which one is a finite energy cylinder,  say $\hat{u}_0$, with one negative and one positive puncture,   and where the others are finite energy holomorphic planes, say $\hat{u}_1, \ldots, \hat{u}_m$, each with one positive puncture. At the negative puncture, $\hat{u}_0$ is asymptotic to $\gamma_{\rho}$. 
We obtain a contradiction by showing that $m=0$ and that $\hat{u}_0$ is also asymptotic to $\gamma_{\rho}$ at the positive puncture. 
To show that $m=0$, we consider a finite energy plane $\hat{u}_1$ positively asymptotic to a Reeb orbit $\gamma$ of $C_-\alpha$. Taking a sequence of holomorphic disks that, after applying shifts in the target, converge in $C^{\infty}_{\loc}$ to $\hat{u}_1$, one can, by a gluing procedure analogous to the situation above, extend those maps restricted to a subdomain to a continuous disk map $H = (a,f) : B_d \to \R \times Y$ with $a(x) \to -\infty$ as $|x| \to d$. Here $B_d$ is the disk at $0$ of sufficiently large radius $d$. If  $\gamma$ is not a multiple of some $\gamma_0^j$, $j=1, \ldots, {\rm n}_0$,   one can arrange that $\Im H \cap V^R_{1,j}= \emptyset$ and that  $f(re^{2\pi it})$   is asymptotic  to  $\gamma$  as $r\to d$. If $\gamma$ is  an $l$-fold multiple of some $\gamma_0^j$, $l\geq 1$, one can arrange that $H(x) \notin V^R_{1,j}$ if and only if $|x| < d/2$, and that 
$f(re^{2\pi t})$ is asymptotic to an $l$-fold multiple of $\gamma_0^j$ as $r\to d$.  
In both cases we can apply Lemma \ref{lem:homot1} and obtain a contradiction to our assumption that $\alpha_0$ is hypertight in the complement of $\mathcal{L}_0$.  
The arguments to show that $\hat{u}_0$ is a cylinder are symmetric to those that show that $\kappa^{+} = 0$. 
\end{proof}

Hence, with Lemmas~\ref{lem:v_no_topbot_level} and~\ref{lem:u_firstlevels} we conclude that the top resp.\ the bottom level of the holomorphic buildings  $\mathbf{v}_i$, $i\in \{1, \ldots,  {\rm n}_0\}$, and $\mathbf{u}_{\rho}$, $\rho\in \Omega_{\alpha_0}(\mathcal{L}_0)$, are holomorphic curves in the exact symplectic cobordism $(\R \times Y, \lambda^+_1,\omega^+_1,J^{+\infty}_1)$ from $C^+ \alpha_0$ to $ \alpha$ resp.\ in the exact symplectic cobordism $(\R \times Y, \lambda^-_1,\omega^-_1,J^{-\infty}_1)$ from $\alpha$ to $C^{-} \alpha_0$.
 
\subsection{Linking properties near the asymptotics in the top and bottom level of the holomorphic buildings}\label{sec:cobordism levels}
Let us consider now  the top and bottom level of the buildings $\mathbf{v}_j$, $j=1, \ldots, {\rm n}_0$, and $\mathbf{u}_{\rho}$, $\rho \in \Omega_{\alpha_0}(\mathcal{L}_0)$. Fix for now two different homotopy classes $\rho_{1}$ and $\rho_2$, and denote the corresponding buildings  by $\mathbf{u}^R_1=\mathbf{u}^R_{\rho_1}$ and $\mathbf{u}^R_2 = \mathbf{u}^R_{\rho_2}$ with levels  $\overline{u}^{0}_1, \ldots,\overline{u}^{\kappa_{\rho_1}+1}_1$ and $\overline{u}^{0}_2, \ldots,\overline{u}^{\kappa_{\rho_2}+1}_2$. By Lemmas~\ref{lem:v_no_topbot_level} and~\ref{lem:u_firstlevels}, and by  our assumptions, the orbits to which $\overline{v}_j^{\kappa_j + 1}$ is positively asymptotic form the link $\mathcal{L}_0$, and the orbits to which $\overline{u}^{\kappa_{\rho_1}+1}_1$ and $\overline{u}^{\kappa_{\rho_2}+1}_2$ are positively asymptotic  are not homotopic in the complement of $\mathcal{L}_0$ and lie both in proper link classes.  Similarly, the orbits to which $\overline{v}_j^{0}$ are negatively asymptotic form the link $\mathcal{L}_0$, and the orbits to which $\overline{u}^{0}_1$ and $\overline{u}^{0}_2$ are negatively asymptotic are not homotopic in the complement of $\mathcal{L}_0$ and lie both in proper link classes.     
As we will now show,  this linking behaviour persists on the other end of these levels, that is, when the image of the top level is  intersected with  
$\{a\}\times Y$ for $a$ sufficiently small, and the image of the bottom  level is intersected with $\{a\} \times Y$ for $a$ sufficiently large, where the corresponding link might have more components  as the link $\mathcal{L}_0$.  
We discuss only the bottom level $\overline{v}_j^{0}, j=1, \ldots, {\rm n}_0$,  $\overline{u}^{0}_1$,  and $\overline{u}^{0}_2$ in the exact symplectic cobordism $(\R \times Y, \lambda^-_1,\omega^-_1,J^{-\infty}_1)$, the situation for the top level is treated similarly.

We denote the cylindrical 
 components of the levels $\overline{v}^0_j$, $j=1,\ldots, {\rm n}_0$ by  $\hat{v}_j:\R \times S^1 \to \R \times Y$, $j=1, \ldots, {\rm n}_0$, and  the cylindrical components of $\overline{u}^0_i$, $i=1,2$, by $\hat{u}_i: \R \times S^1 \to \R \times Y$. 
By the previous discussion,   $\hat{v}_j$ is   negatively asymptotic to $\gamma_0^j$, $j=1, \ldots, {\rm n}_0$,  and  $\hat{u}_i$ is negatively asymptotic to $\gamma_{\rho_i}$, $i=1,2$. Additionally, denote, in some order, all the remaining components of $\overline{v}_j^0$, $j=1, \ldots, {\rm n}_0$, by $\hat{v}_j$, $j={\rm n}_0+1,\ldots, \hat{\rm n}_0$. These are finite energy planes, positively asymptotic to some Reeb orbits of $\alpha$. 
Since the negative asymptotic orbits of the finite energy cylinders  are all pairwise  disjoint,  any two of the above curves do not have any intersection points in $[a_0,+ \infty)\times Y$ for $a_0>0$ sufficiently large;  in fact the curves are pairwise disjoint by an argument similar as in the proof of Lemma \ref{lem:v_no_topbot_level}. 

We denote by $\hat{V}_j$ the image of $\hat{v}_j$, $j=1, \ldots, \hat{\rm n}_0$, and by $\hat{U}^i$ the image of $\hat{u}_i$, $i=1,2$. Given $a>0$, we put $\hat{V}^a_j := \hat{V}_j \cap (\{a\} \times Y)$, $j=1, \ldots, \hat{\rm n}_0$, and   $\hat{U}^a_i :=\hat{U}_i \cap (\{a\} \times Y)$, $i=1,2$. 
The finite energy cylinders approach a trivial cylinder over their asymptotics in the $C^{\infty}$-sense, and in particular  there is  $a_1\geq a_0$ such that for all $a\geq a_1$ the sets $\hat{V}^a_j$, $\hat{U}^a_i$ can be parametrized by simple closed curves in $\{a\} \times Y$; we choose such curves  $\Gamma^{a}_j$, $j=1, \ldots, \hat{\rm n}_0$, and $\Lambda^a_i$, $i=1,2$, respectively (with orientation compatible with those of the asymptotics).  
We write $\mathcal{L}^a$ for the link formed by the closed curves  $\Gamma^a_j$, $j=1, \ldots, \hat{\rm n}_0$.
 
\begin{lem}\label{lem:Lambdaa}
For all $a\geq a_1$, 
\begin{enumerate}[(i)]
    \item  $\Lambda^{a}_1$ and $\Lambda^{a}_2$ are not homotopic in the complement of $\mathcal{L}^a$, 
\item neither $\Lambda^a_1$ nor $\Lambda^a_2$ is homotopic to any multiple of $\Gamma^a_j$, $j=1, \ldots, {\rm n}_0$,  in the complement of $\mathcal{L}^a$.
\end{enumerate}
\end{lem}
\begin{proof}
Fix $a\geq a_1+1$. 
Then, by the construction, for any $j=1, \ldots, {\rm n}_0$ there is  a sequence $R_k \to +\infty$, cylinders $\mathfrak{A}_k\subset S^1 \times \R$, biholomorphic maps $\psi_k:[-L_k,L_k]\times S^1 \to\mathfrak{A}_k$ such that with  $\textbf{b}_k:=-R_k$,  the sequence 
${\hat{v}}^k_j := \mathbf{T}_{-\textbf{b}_k} \circ \widetilde{v}_j^{R_k}(1) \circ \psi_k$  converges in $C^{\infty}_{\loc}$ to $\hat{v}_j$. Furthermore, there are, for any $j={\rm n}_0 +1, \ldots, \hat{\rm n}_0$, disks  $D_k \subset \R \times S^1$, biholomorphic maps $\psi_k: B_{k} \to D_k$, with disks $B_k\subset \C$ of unbounded radii such that with  $\textbf{b}_k:=-R_k$, the sequence   
$\hat{v}^k_j := \mathbf{T}_{-\textbf{b}_k} \circ \widetilde{v}_j^{R_k}(1) \circ \psi_k$  converges in $C^{\infty}_{\loc}$ to $\hat{v}_j$. 
Similarly, for $i=1,2$,  there is a sequence $
{\textbf b}_k = -R_k$  and  biholomorphic maps $\psi'_k: B_{k} \to D_k$ such that  the sequence $\hat{u}^k_i :=\mathbf{T}_{-\textbf{b}_k} \circ \widetilde{u}_i^{R_k}(1) \circ \psi'_k$  converges in $C^{\infty}_{\loc}$ to $\hat{u}_i$.
We may choose the sequences ${\textbf b}_k = -R_k$ above to be  identical for all $j=1, \ldots, \hat{\rm n}_k, i=1,2$. 

This means that for any given $\epsilon>0$ there is $k_0 = k_0(\epsilon)$ such that for $k\geq k_0$  and for any of the above sequences   there are simple essential closed curves $\beta_k:S^1 \to  [-L_k, L_k]\times S^1$ if $j=1, \ldots, {\rm n}_0$, or simple closed curves $\beta_k:S^1 \to B_k$ if $j={\rm n}_0 + 1,\ldots, \hat{\rm n}_0$  such that 
$\Gamma^k_j:= \hat{v}^k_j \circ \beta_k$ parametrizes the intersection of the image of $\hat{v}^k_j$ with $\{a\} \times Y$, and such that 
the curve $\Gamma^k_j$ lies in the tubular neighbourhood of $\Gamma_j^a$ of size $\epsilon$. The analogous fact holds for the sequences $\hat{u}^k_i$, where we  consider closed curves $\Lambda^k_i = \hat{u}^k_i \circ \beta^i_k$ that lie in  tubular neighbourhoods of $\Lambda_1^a$. 
Moreover, the loops  $\Gamma^k_j$ resp.\  $\Lambda^k_i$ can be written as the composition of the exponential map, restricted to the normal bundle of $\Gamma_j^a$ resp.\  $\Lambda^a_i$, and a section to it. This means that for $\epsilon_0>0$ sufficiently small, it is possible to  construct a diffeomorphism $\Theta:\{a\} \times Y \to \{a\} \times Y$,  with $\Theta \circ \Gamma^k_j = \Gamma^a_j$, $j=1, \ldots, \hat{\rm n}_0$, and $\Theta \circ \Lambda^k_i = \Lambda^a_i$, $i=1,2$. 
Therefore it is sufficient 
to show properties (i) and (ii) with $\Lambda_i^a$, $\Gamma_j^a$, and the link $\mathcal{L}^a$, replaced by  $\Lambda^k_i$, $\Gamma^k_j$, and the link $\mathcal{L}^{a,k}$ that is formed by the curves $\Gamma^k_j$, for some $k\geq k_0(\epsilon_0)$. Note that if $k_0$ is sufficiently large, any link $\mathcal{L}^{k}$ with  $k\geq k_0$,  as a set in $\{a\} \times Y$, coincides with $\Im (\mathbf{T}_{-\textbf{b}_k} \circ \widetilde{v}^{R_k}_j(1)) \cap (\{a\} \times Y)$.  

To show (i), we assume that there is a homotopy $h:[0,1]^2 \to \{a\} \times Y$ from $\Lambda^k_1$ to $\Lambda^k_2$ in the complement of $\mathcal{L}^{k}$. Similarly as in the proof of Lemma~\ref{lem:u_firstlevels}, we choose two half-cylinders $Z_1:(-\infty,0] \times S^1 \to \R \times Y$, $Z_2:[1,\infty) \times S^1 \to \R \times Y$ for which  
\begin{itemize}
    \item $\mathbf{T}_{\textbf{b}_k} \circ Z_1$ are disjoint from the images of $v_j^{R_k}(1)$, $j=1, \ldots, \hat{\rm n}_0$, 
    \item  $\lim_{s\to -\infty} \uppi_2 \circ Z_1(s,t) = \gamma_{\rho_1}(t)$, $\,\, \lim_{s\to +\infty} \uppi_2 \circ Z_2(s,t) = \gamma_{\rho_2}(t)$, 
    \item $Z_1$ and $Z_2$ coincide along their boundaries with $\Lambda^1_k$ and $\Lambda^2_k$, respectively.  
    \end{itemize}
This is clearly possible since 
    the images of $\widetilde{v}_j^{R_k}(1)$, $j=1, \ldots, {\rm n}_0$, do not intersect the images of $\widetilde{u}_i^{R_k}(1)$, $i=1,2$.
We can now apply Lemma \ref{lem:homot1} to the cylinder $H:= \mathbf{T}_{\textbf{b}_k}( Z_1 \# h \# Z_2)$, obtained by gluing and shifting the above maps, and we obtain that $\gamma_{\rho_1}$ and $\gamma_{\rho_2}$ are homotopic in the complement of $\mathcal{L}_0$, a contradiction. 
The argument to exclude (ii) is similar: if (ii) was false, we could  construct a cylindrical map $H:\R \times S^1\to \R \times Y$, negatively asymptotic to some $\gamma_{\rho_i}$,  positively asymptotic to a multiple of  some $\gamma_0^j$, and for which $H(s,t) \in V^{R_k}_j$ if and only if $s\geq1$. Again with Lemma \ref{lem:homot1} we get a contradiction.    
\end{proof}
Since the holomorphic curves above approach trivial cylinders over their asymptotics in the $C^{\infty}$-sense, we obtain in fact the following corollary.
\begin{cor}\label{cor:Lambdaa}
\begin{enumerate}[(i)]
\item For any $a\geq a_1$ there is $R>0$ such that for all $s, s'\geq R$,  the curves $\lambda_1, \lambda_2:S^1 \to \R \times Y$,  
$\lambda_1(t) = \hat{u}_1(s,t)$ and $\lambda_2(t) = \hat{u}_2(s',t)$,  are not homotopic in $([a,  \infty) \times Y )\setminus \bigcup_{i=1}^{\hat{\rm n}_0}\widehat{V}_j$,
 
\item  for any $a\geq a_1$ there is $R>0$ such that for all $s,s'\geq R$ neither of the loops $\lambda_i:S^1 \to \R \times Y$,  
$\lambda_i(t) = \hat{u}_i(s,t)$,  $i\in \{1,2\}$, is homotopic in $([a, \infty) \times Y) \setminus \bigcup_{i=1}^{\hat{\rm n}_0}\widehat{V}_j$ to  a multiple of any  $\gamma_0^j: S^1 \to \R \times Y$, $\gamma_0^j(t) = \hat{v}_j(s',t)$, $j=1, \ldots, {\rm n}_0$.  This means that any homotopy of such loops  in $([a, + \infty) \times Y)$ must intersect in an interior point the set $\bigcup_{i=1}^{\hat{\rm n}_0}\widehat{V}_j$. \end{enumerate}
\end{cor}

\subsection{Linking of the ends of holomorphic curves asymptotic to the same Reeb orbit}\label{sec:non-collaps}
In the previous 
 section we discussed some linking properties of the top and bottom level of the buildings  $\mathbf{v}_j$ and $\mathbf{u}_{\rho}$ when intersected with some horizontal spaces $\{a\} \times Y$. The important task that remains for the proof of Theorem~\ref{thm:stability} is to analyze the situation when we pass to the limit, that is, when we consider the collection of negative asymptotic Reeb orbits of $\phi_\alpha$ of the top level of $\mathbf{v}_j$, $j=1, \ldots, {\rm n}_0$, and $\mathbf{u}_{\rho}$, $\rho \in \Omega_{\alpha_0}(\mathcal{L}_0)$, or the positive asymptotic orbits of $\phi_{\alpha}$ of the bottom level of those buildings. Note that in general, although the levels above are pairwise disjoint, their asymptotics might agree. In other words,  several holomorphic curves might in the limit "collapse" to the same asymptotic orbit. We show now that at least the growth of the number of such holomorphic curves is controlled. We will only discuss the bottom level, and similar results, although not necessary for the proof of Theorem~\ref{thm:stability}, hold for the top level. 

Before we finish with the proof of Theorem~\ref{thm:stability} in Section~\ref{sec:conclusions}, let us here first restrict to a simplified situation. Consider the 
exact symplectic cobordism endowed with an almost complex structure $(W,J)=(\R \times Y, \lambda^-_1,\omega^-_1,J=J^{-\infty}_1)$ from $\alpha$ to $C^-\alpha_0$. In fact, in the following we will only consider holomorphic curves that map into $(0,+\infty) \times Y$, where $W$ coincides with the symplectization of $\alpha$. Recall that $J^{-\infty}_1$ restricted to $(0,+\infty)\times Y$ coincides with $J_{\alpha} \in \mathcal{J}(\alpha)$ and let $j_{\alpha}\in \mathfrak{j}(\alpha)$ be the restriction of $J_{\alpha}$ to $\xi$. 
Let $N \in \N$, and let $\gamma$ be a simple non-degenerate Reeb orbit of $\alpha$. 
Let $\overline{v}_1, \ldots, \overline{v}_N:[0,+\infty)\times S^1 \to (0,+\infty) \times Y$  be a finite collection of disjoint and simply covered   
holomorphic half-cylinders in $(W,J)$, each positively asymptotic to a multiple of $\gamma$. Denote by $\overline{V}_i$, $i=1,\ldots, N$, the images of the curves $\overline{v}_i$ in $\R \times Y$. 
Furthermore, let $\Xi$ be another (not necessarily finite) collection of disjoint holomorphic half-cylinders in $(W,J)$ with image in $[1,\infty) \times Y$,  each of them  positively asymptotic to a multiple of $\gamma$. Assume that the images of all elements in $\Xi$ are disjoint from  $\overline{V}_1,\ldots,\overline{V}_N$. 
 For any $k\in \N$, we denote the $k$th multiple of the Reeb orbit $\gamma$ by  $\gamma_k:S^1 \to Y$, $\gamma_k(t) := \gamma(kt)$. Denote by $\Xi_k\subset \Xi$ the subset  of those half-cylinders in $\Xi$ that are positively asymptotic to $\gamma_k$. 
In the following, we will assume that 
\begin{itemize}
\item[(*)]  
there is $a_0>1$ such that for all $a\geq a_0$ there is  $R\geq 0$ with the following property. For all $\overline{u} \in \Xi$, $s, s^*\geq R$, $j=1,\ldots, N$, the loop defined by $t \mapsto \overline{u}(s,t)$ is not homotopic to a multiple of the loop $t\mapsto  \overline{v}_j(s^*,t)$  in  $\left((a, +\infty) \times Y \right) \setminus  \bigcup_{i=1}^N\overline{V}_i$.
\end{itemize}
We say that $\overline{u}, \overline{u}' \in \Xi$ are equivalent ($\overline{u}\sim \overline{u}'$) if for $a>0$ sufficiently large, there is $R \geq 0$ such that for all $s, s'\geq R$ the loops $t\mapsto \overline{u}(s,t)$ and $t\mapsto \overline{u}'(s',t)$ are homotopic  in $\left((a ,+\infty) \times Y\right) \setminus \bigcup_{i=1}^N \overline{V}_i$. 
\begin{prop}\label{prop:Xi_equi}
Let $\overline{v}_1, \ldots, \overline{v}_N$, $\Xi$ be as above  and assume that (*) holds. Then,  the number of $\sim$-equivalence classes in $\Xi_k$ grows at most linearly in $k$.  
\end{prop}
The proof of the proposition uses in an essential way results by Siefring, \cite{Siefringrelative}, on the relative asymptotic behaviour of a collection of curves asymptotic to a multiple of the same orbit $\gamma$. Let us recall a formula from \cite{Siefringrelative} that describes,  in a convenient coordinate system around a neighbourhood of the trivial cylinder over $\gamma$,  the asymptotic behaviour of  finitely many such curves in a rather simple way.

Let $K\in \N$. The linearization of the Reeb vector field and the almost complex structure $j_{\alpha}$ along the contact distribution $\xi$ induce a self-adjoint operator
$A_K$ on the space $L^2(\gamma_K^*\xi)$ of $L^2$-sections of the bundle $\gamma_K^*\xi$ (see~\cite{Siefringrelative}). In any Hermitian trivialization of $(\gamma^*\xi, j_{\alpha}, d\alpha(\cdot, j_{\alpha}\cdot))$ this operator is  represented by a bounded symmetric perturbation of $-i\frac{d}{dt}$. 
The spectrum of $A_K$ consists of real eigenvalues of multiplicities at most $2$ that accumulate only at $\pm \infty$. 
 Fix in the following a trivialization $\Phi:S^1 \times \R^2 \to \gamma^* \xi$. This induces trivializations of $\gamma_K^*\xi$, which we still denote by $\Phi$ if it is clear from the context.
\begin{thm}(\cite[Theorem 2.4]{Siefringrelative})\label{thm:Richard}
Let $K\in \N$.  Let $w_i$, $i=1, \ldots, n$, be a finite collection of holomorphic half-cylinders in $(W, J)$, positively asymptotic to~$\gamma_K$.
Then, there exist $R>0$, an open tubular  neighbourhood $\mathfrak{U}$ of $\gamma$, a smooth embedding $\widetilde{\Phi} : \R \times \mathfrak{U} \to \R \times S^1 \times \R^2$ satisfying 
$$\widetilde{\Phi}(a,\gamma(t)) = (a,t,0) \in \R \times S^1\times \R^2,$$
proper embeddings $\psi_i:[R, \infty) \times S^1 \to \R \times S^1$ asymptotic to the identity, and $L\in \N$ so that 
\begin{align}\label{eq:siefr}
(\widetilde{\Phi} \circ w_i \circ \psi_i)(s,t) = \left(Ks,Kt, \sum_{j=1}^{L}e^{\lambda_{j}s}e_{i,j}(t)\right), 
\end{align}
where $\lambda_{j}$ are negative eigenvalues of $A_{K}$ with $\lambda_L < \cdots < \lambda_1 < 0$, and the $e_{i,j}$ are eigenvectors (possibly zero) of $\Phi^{-1} A_{K}\Phi$ with eigenvalue $\lambda_{j}$. 
\end{thm}
Let us also recall some  facts about winding numbers of eigenvalues of the asymptotic operator. Assume that the trivialization is symplectic with respect to $dx \wedge dy$ on $\R^2$ and satisfies $\Phi \circ j_0 = j_{\alpha} \circ \Phi$, where $j_0$ corresponds to $i$ in the usual identification $\C \equiv\R^2$. 
With respect to (an iterate of) the trivialization $\Phi$, we  can write $\Phi^{-1} A_K \Phi= -j_0 \frac{d}{dt} + S(t)$ on $L^2(S^1, \R^2)$, where $S(t)$ is a symmetric $2\times 2$ matrix for all $t\in S^1$.  Let $\lambda \neq 0$ be an eigenvalue. Then, any eigenfunction $e$ of $\Phi^{-1} A_K \Phi$ for  $\lambda$ is nowhere zero and hence defines a winding number, which does not depend on the choice of the eigenfunction $e$. We will denote it by $\mathbf{w}^{\Phi}(\lambda) = \mathbf{w}^{\Phi}(\lambda,K) \in \Z$.
Furthermore, by Siefring's work, one can define a relative asymptotic winding number (with respect to a given trivialization) for two non-identical holomorphic half-cylinders  $u,v$ positively asymptotic to a common multiple of $\gamma$, which we denote 
 by $\mathbf{w}^{\Phi}(u,v)$.
It follows from the construction of $\widetilde{\Phi}$ in the proof of Theorem \ref{thm:Richard} in \cite{Siefringrelative} that one can express the asymptotic relative  winding number with the formula \eqref{eq:siefr} as follows: for any $w_{i_1}$ and $w_{i_2}$, $i_1 \neq i_2$, we have that $\mathbf{w}^{\Phi}(w_{i_1},w_{i_2}) = \mathbf{w}^{\Phi}(\lambda_j,K)$, where $j$ is the minimal natural number such that $e_{i_1, j} \neq e_{i_2,j}$. 

We will also use the following fact for winding numbers $\mathbf{w}^{\Phi}(\lambda,K)$ (see~\cite{HWZII}). 
\begin{enumerate}[({w})]
  \item   $\mathbf{w}^{\Phi}(\lambda, K)$ is non-decreasing in $\lambda$ and admits the same value for at most two different eigenvalues.
  \end{enumerate}
\begin{proof}[Proof of Proposition \ref{prop:Xi_equi}]
It will be convenient to choose a trivialization $\Phi$ of $\gamma^*\xi$ such that for all $k\in \N$ it holds that $\mathbf{w}^{\Phi}(\lambda^k_*,k) < 0$, where $\lambda^k_*<0$ denotes the largest negative eigenvalue of $A_k$ (see e.g.\ \cite[Lemma 3.3]{Siefringrelative} for properties of $\mathbf{w}^{\Phi}$ under iteration of $\gamma$).
Denote by $k_i\in \N$ the natural numbers so that $\overline{v}_i$ is positively asymptotic to $\gamma_{k_i}$, the $k_i$th iterate of $\gamma$. Let $\widehat{K}$ be the least common multiple of $k_1, \ldots, k_N$.
For $i=1, \ldots, N$  define the shifted curves $\overline{v}^q_i:[0,+\infty) \times S^1 \to (0,+\infty) \times Y$, $q=0,\ldots, k_i-1$, by $\overline{v}^q_i(s,t) = \overline{v}_i(s,t+\frac{q}{k_i})$. 
The $\frac{\widehat{K}}{k_i}$ iterates of $\overline{v}^q_i$, defined by 
$\check{v}^q_i(s,t) := \overline{v}^q_i(\frac{\widehat{K}}{k_i}s,\frac{\widehat{K}}{k_i}t)$, are asymptotic to $\gamma_{\widehat{K}}$. 
Since the $k_i$ are minimal and $\overline{v}_i$ and $\overline{v}_{i'}$ do not intersect if $i \neq i'$,
the relative winding number $\mathbf{w}^{\Phi}(\check{v}^q_i,\check{v}^{q'}_{i'})$  is well defined if $(i,q) \neq (i',q')$. Note that with the choice of $\Phi$ above,  $\mathbf{w}^{\Phi}(\check{v}^q_i,\check{v}^{q'}_{i'})\leq \mathbf{w}^{\Phi}(\lambda^{\widehat{K}}_*, \widehat{K}) <0$.

Fix $k\in \N$ and denumerate the elements in $\Xi_k$  by $\overline{u}_1, \ldots, \overline{u}_M$. Let $K$ be the least common multiple of $k$ and $\widehat{K}$. 
We consider the holomorphic half-cylinders $\widetilde{v}^q_i: [0,+\infty) \times S^1 \to (0,+\infty) \times Y$, $i=1, \ldots, N$, $q=0, \ldots, k_i -1$, given by $\widetilde{v}^q_i(s,t) = \overline{v}^{q}_i(\frac{K}{k_i} s,\frac{K}{k_i} t)= \check{v}^q_i(\frac{K}{\hat{K}}s,\frac{K}{\hat{K}}t)$, and the half-cylinders  $\widetilde{u}_i: [R_1,+\infty) \times S^1 \to \R \times Y$, $i=1, \ldots, M$, given by $\widetilde{u}_i(s,t) = \overline{u}_i(\frac{K}{k} s,\frac{K}{k} t)$. By the $\R$-invariance of $J$ in $(0,+\infty) \times Y$, these are holomorphic half-cylinders, which are moreover all asymptotic to $\gamma_K$.

We rename the curves $\widetilde{v}^0_1, \ldots, \widetilde{v}^0_{N}, \widetilde{u}_1, \ldots, \widetilde{u}_M$ (in that order) by 
$w_1, \ldots, w_{n}$, $n = N+M$, and apply  Theorem \ref{thm:Richard} to
$w_1, \ldots, w_n$. 
We obtain $R>0$, $\widetilde{\Phi}$, and $\psi_{i}$, $i=1, \ldots, n$, as above, and some $L \in \N$  such that for $s\geq R$,  
\begin{align}\label{eq:asympt_form}
(\widetilde{\Phi} \circ w_i \circ \psi_i)(s,t) = \left(Ks,Kt, \sum_{j=1}^{L} e^{\lambda_{j}s}e_{i,j}(t)\right), 
\end{align}
where $e_{i,j}$ are eigenvectors (possibly $0$) of $A = A_{K}$ with eigenvalues $\lambda_j<0$. For $e(t) \in C^{\infty}(S^1,\R^2)$ we write 
$\frac{r}{K} \star e(t) := e(t + \frac{r}{K})$. 

Let $p_2: \R \times S^1 \times \R^2 \to S^1 \times \R^2$ be the projection to the last two coordinates. 
By the properties of $\widetilde{\Phi}$ and $\psi_i$ in Theorem \ref{thm:Richard}, it is sufficient to show that there is a constant $C$, not depending on $k$, and $s_0 \geq 0$ such that for all $s\geq s_0$ the number of free homotopy classes of the loops $t\mapsto p_2 \circ \widetilde{\Phi} \circ w_i \circ \psi_i(s,t)$, $i=N+1, \ldots, n$, in the complement of the images of the loops 
$t\mapsto p_2 \circ \widetilde{\Phi} \circ w_i \circ \psi_i(s,t)$, $i=1, \ldots, N$, in $S^1\times \R^2$  
are bounded from above by $Ck$. 
Indeed, the image of $\widetilde{\Phi}$ intersected with $\{a\} \times S^1 \times \R^2$ can be chosen to be in $\{a\} \times Q^a$ for a tubular neighbourhood $Q^a$ of $S^1 \times \{0\}$. Hence, if for $s\geq s_0$, $i,j\in\{N+1,\ldots,n\}$, there is a homotopy between the loops  
$t \mapsto p_2 \circ \widetilde{\Phi} \circ w_i \circ \psi_i(s,t)$  and $t\mapsto p_2 \circ \widetilde{\Phi} \circ w_j \circ \psi_j(s,t)$ in the complement of the images of 
$t\mapsto p_2 \circ \widetilde{\Phi} \circ w_i \circ \psi_i(s,t)$, $i=1, \ldots, N$, in $S^1\times \R^2$, one can choose it to be supported in $Q^{Ks}$. That homotopy corresponds to a homotopy in $\{Ks\} \times S^1 \times \R^2$, and by applying $\widetilde{\Phi}^{-1}|_{\{Ks\} \times S^1 \times \R^2}$, we obtain a homotopy of 
$t \mapsto w_i \circ \psi_i(s,t)$ and $t \mapsto w_j \circ \psi_j(s,t)$ in $(a,+\infty) \times \mathfrak{U}$   in the complement of 
$\bigcup_{i=1}^N \overline{V}_i$,   where  $a = a(s)$ 
 is suitably chosen with $a(s) \to \infty$ as $s \to +\infty$. 
The homotopy can be extended in the same set to a homotopy of the loops $t \mapsto w_i(s,t)$ and $t \mapsto w_j(s,t)$.

We write $w^0_i:= w_i$, $i=1, \ldots, n$, and also $w^q_i = \widetilde{v}^q_i$, $i=1, \ldots, N, q=1, \ldots, k_i-1$.
We write $\mathcal{V} := \{w^q_i\, |\, 1\leq i \leq N,\,  1\leq q \leq k_i-1\}$, $\mathcal{U} :=\{w^0_i \, |\, N+1 \leq i \leq n\}$, and $\mathcal{W}= \mathcal{V} \cup \mathcal{U}$. Let $N':=\# \mathcal{V} = \sum_{i=1}^N k_i$.  
For any $w^q_i$ above  we consider the  vector $\xi(w^q_i) = (e^q_{i,j})_{1\leq j \leq L} \in \left(C^{\infty}(S^1,\R^2)\right)^L$. Here $e^q_{i,j}(t) = e_{i,j}(t+\frac{q}{K})$. 
For $r=0, 1, \ldots, K-1$ and $\widetilde{v}\in \mathcal{W}$  we define $$\frac{r}{K} \star \widetilde{v}(s,t) := \widetilde{v}(s,t + \frac{r}{K}).$$ This defines a $\Z_{K}$-action on $\mathcal{V}$. Note that the orbit of $\widetilde{v}^q_i$, $i=1,\ldots, N$, $q=1,\ldots,k_i-1$, under this action is $k_i$-periodic, and that the image of any curve above is invariant under this action. 
By the formula \eqref{eq:asympt_form},  we have  
$$\widetilde{\Phi} \circ w_i^q \circ {\psi}^q_{i}(s,t) = \left(Ks, Kt, \sum_{i=1}^L e^{\lambda_{j}s}e^q_{i,j}\right), $$
where $\psi^q_i(s,t) := \psi_i(s,t+\frac{q}{K})$.

We now describe a partially ordered set $(\mathcal{A}, \preceq)$.
We describe $\mathcal{A}$ as the vertices of a rooted tree $\mathcal{T}$ of height $L+1$.  The partial order $\preceq$ will be the natural order induced by the tree structure. 
The level $0$ will consist only of the root, a vertex $\xi_0$. The  vertices of $\mathcal{T}$ in level $l$, $1\leq l \leq L$,  are identified with all the vectors $\xi \in \left(C^{\infty}(S^1,\R^2)\right)^{l}$ which are obtained from any $\xi(w^q_i)$ by considering its first $l$ entries. A vertex $\xi$ in level $l$ is connected to a vertex $\xi'$ in level $l-1$ if $\xi'$ results from $\xi$ after deleting the $l$th entry. Moreover, all vertices of level $1$ are connected to the root $\xi_0$. 
Note that the leaves of $\mathcal{T}$, i.e.\  the vertices in level $L$, are given by $\xi(w), w\in \mathcal{W}$. See Figure~\ref{fig:tree} for an example.  
The tree structure $\mathcal{T}$ defines a partial order $\preceq$ on its vertices $\mathcal{A}$ (the root is a minimal element).  We write $\xi \prec \xi'$ if $\xi \preceq \xi'$ and $\xi \neq \xi'$.

\begin{figure}
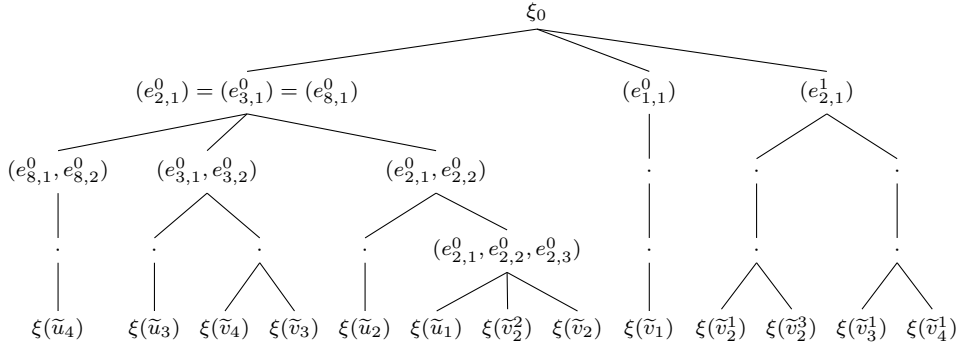

{\SMALL   
\Tree [.$\xi_0$ [.$(e^0_{2,1})=(e^0_{3,1})=(e^0_{8,1})$ [.$(e^0_{8,1},e^0_{8,2})$   [.$\cdot$ $\xi(\widetilde{u}_4)$ ] ]   [.$(e^0_{3,1},e^0_{3,2})$  [.$\cdot$ $\xi(\widetilde{u}_3)$ ] [.$\cdot$ $\xi(\widetilde{v}_4)$ $\xi(\widetilde{v}_3)$ ] ]  [.$(e^0_{2,1},e^0_{2,2})$ 
 [.$\cdot$ $\xi(\widetilde{u}_2)$ ] [.$(e^0_{2,1},e^0_{2,2},e^0_{2,3})$ $\xi(\widetilde{u}_1)$ $\xi(\widetilde{v}^2_2)$ $\xi(\widetilde{v}_2)$ ] ] ]  [.$(e^0_{1,1})$ 
  [.$\cdot$ [.$\cdot$ $\xi(\widetilde{v}_1)$ ] ] ] [.$(e^1_{2,1})$ [.$\cdot$ [.$\cdot$ $\xi(\widetilde{v}^1_2)$ $\xi(\widetilde{v}^3_2)$ ] ] [.$\cdot$ [.$\cdot$ $\xi(\widetilde{v}^1_3)$ $\xi(\widetilde{v}^1_4)$ ] ] ] ]
}
\caption{The tree $\mathcal{T}$ for an example with four half-cylinders $\widetilde{v}_1=w^0_1, \ldots, \widetilde{v}_4=w^0_4$, with $k_1=1$, $k_2 = 4, k_3=k_4=2$, and  with $\mathcal{U}$ consisting of four half-cylinders $\widetilde{u}_1= w^0_{5}, \ldots, \widetilde{u}_4=w_8^0$.}
    \label{fig:tree}
\end{figure}

For a set of distinct vertices $\xi_1, \ldots, \xi_m$, $m\geq 1$,  we denote by $\xi(\xi_1, \ldots, \xi_m)$ the largest vertex $\xi$ in $(\mathcal{A},\preceq)$ with $\xi\preceq \xi_i$ for all $1\leq i \leq m$. In other words, $\xi(\xi_1, \ldots, \xi_m)$  is the vector of maximal length $l$ such that the vector build by the first $l$ entries of any of the $\xi_1, \ldots, \xi_m$ is equal to $\xi$. If $\xi_i$, $1\leq i\leq m$, have different first entry, then we set $\xi(\xi_1, \xi_2, \ldots, \xi_m) = \xi_0$ to be the root.
Note that $\xi(\xi_1) = \xi_1$, for any vertex $\xi_1$. 

We equip $\R^2$ and $S^1= \R/\Z$  with the Euclidean metric and  equip  $S^1 \times \R^2$ with the product metric.  
Let $w\in \mathcal{W}$, let $\psi$ be the corresponding function in formula~\eqref{eq:asympt_form}, and write $\xi(w) = (e_1, \ldots, e_L)$.
For any $l\in \N$,  $0< l <L$, and $s\geq R$, 
\begin{align*}
\| &(p_2 \circ \widetilde{\Phi} \circ w \circ \psi)(s,t) - (Kt, \sum_{i=1}^{l} e^{\lambda_i s} e_i(t)) \| \leq \\&e^{\lambda_{l+1}s} \left( \|e_{l+1} (t) \|  +  e^{(\lambda_{l+2}-\lambda_{l+1})s} \|e_{l+2}(t)\| + \cdots + e^{(\lambda_{L}-\lambda_{l+1})s} \|e_{L}(t)\| \right).  
\end{align*}
Hence there exist $c_1>0$ and $R_1\geq R$ such that for $0<l<L$ and $s\geq R_1$,
\begin{align}\label{eq:c_1}
\|(p_2 \circ \widetilde{\Phi} \circ w \circ \psi)(s,t) &- (Kt, \sum_{i=1}^{l} e^{\lambda_i s} e_i(t))\| \leq c_1e^{\lambda_{l+1}s}.
\end{align}
On the other hand, if $(e'_1, \ldots, e'_{L})$ is a vector and  $e'_l \neq e_l$, for some $l\in \{1, \ldots, L\}$,  then there is $c_0>0$ and  $R_0 \geq R$ such that for all $s\geq R_0$, 
\begin{align}\label{eq:c_0}
c_0e^{\lambda_{l}s}\leq \|(p_2 \circ \widetilde{\Phi} \circ w \circ \psi)(s,t) - (Kt, \sum_{i=1}^{l} e^{\lambda_i s} e'_i(t))\|.
\end{align}

Let $\mathcal{B}$ be the subset of elements $b \in \mathcal{A}$ that can be written in the form $$b = \xi(\xi(\widetilde{v}^{q_1}_{i_1}),\xi(\widetilde{v}^{q_2}_{i_2}),\ldots,  \xi(\widetilde{v}^{q_m}_{q_m}) ),$$ for a non-empty collection of pairwise distinct curves $\widetilde{v}^{q_1}_{i_1}, \widetilde{v}^{q_2}_{i_2},\ldots,   \widetilde{v}^{q_m}_{i_m}$ in $\mathcal{V}$.  There is exactly one minimal element $b_{\min}$ in $\mathcal{B}$. To every $b\in \mathcal{B}$ with $b\neq b_{\min}$ there is exactly one $\hat{b}\in \mathcal{B}$ that is maximal among the ${b}'\in \mathcal{B}$ with ${b}'\prec b$; we call $\hat{b}$ the \textit{predecessor} of $b$.  Note also that $\# \mathcal{B} \leq 2N'-1$. (Note that in the example in Figure~\ref{fig:tree}, e.g.\ $\xi_0$, $(e^0_{2,1})$, $(e^0_{2,1},e^0_{2,2},e^0_{2,3})\in \mathcal{B}$, whereas $(e_{1,1}^0), (e^0_{2,1},e^0_{2,2})\notin\mathcal{B}$.) 

The $\Z_K$-action on the set of functions $\{ e^q_i, \, |\,  i \in \{1, \ldots, N\}, q\in \{0, \ldots, k_i -1\}\, \}$ above induces a $\Z_{K}$-action on $\mathcal{B}$, given by $\frac{r}{K} \star \xi = \frac{r}{K} \star (e_1, \ldots, e_l) = ( \frac{r}{K} \star e_1, \ldots, \frac{r}{K} \star e_l)$, and fixing the root if it lies in $\mathcal{B}$.   Here $\Z_{K} \cong \{0,\frac{1}{K}, \ldots, \frac{K-1}{K}  \} \mod 1$. We define $\underline{\mathcal{B}} = \mathcal{B}/\sim$, where $b\sim b'$ 
if $b$ and $b'$ lie in the same orbit of the $\Z_K$-action.
Since the action preserves the order $\preceq$ and the level, the level of $\underline{b}\in\underline{\mathcal{B}}$ is well defined, and for every $\underline{b}\neq [b_{\min}]$ there is a well-defined predecessor $\underline{\hat{b}}$ of $\underline{b}$.

Given $\underline{b}\in \underline{\mathcal{B}}$, we denote by $\mathcal{V}(\underline{b})$  the set of all $\widetilde{v} \in \mathcal{V}$ with $b\preceq \xi(\widetilde{v})$ for some $b$ with $[b] =  \underline{b}$. Given $s\geq R$, we denote  by $\mathcal{L}_s\subset S^1 \times \R^2$ the link whose components are obtained by the images of the curves $p_2 \circ \widetilde{\Phi} \circ w^0_i \circ \psi^0_i(s,t)$, $i=1, \ldots, N$. For simplicity, we will say that the components are \textit{induced  by} the curves  $w^0_i$.

We fix now $\underline{b} \in \underline{\mathcal{B}}$ and let $l=l_{\underline{b}}$ the level of $\underline{b}$. If $\underline{b} \neq [b_{\min}]$, let $\hat{l}= \hat{l}_{\underline{b}}< l$ be the level of the predecessor $\underline{\hat{b}}$ of $\underline{b}$.
Choose in the following also  $b = ({e}_1, \ldots, {e}_l) \in \mathcal{B}$ with $[b] = \underline{b}$, and consider the functions 
\begin{align}\label{def:zsb}
z^s_{b}: S^1 \to S^1 \times \R^2, \quad t \mapsto \left(Kt, \sum_{i=1}^l e^{\lambda_i s}{e}_i(t)\right), \quad s\geq R.
\end{align}
In the following, when we consider neighbourhoods of $\Im(z^s_b)$, we exclusively consider such neigbourhoods whose intersection with the horizontal planes $\{t_0\} \times \R^2$, $t_0\in S^1$,  consists of a union of pairwise disjoint disks. We say that such a tubular neigbourhood has \textit{width at most (resp.\ at least) $r$} if the radii of all the disks are at most (resp.\ at least) $r$.
\begin{claim}\label{cl:Y}
There is $R_2 \geq R$, and for all $s\geq R_2$ there are two   tubular neighbourhoods  $Q^s_{\underline{b}}\subset     \hat{Q}^s_{\underline{b}}  \subset S^1 \times \R^2$ of $\Im (z^s_b)$, and constants $c_2,c_3, c_4>0$ such that 
\begin{enumerate}[(i)]
\item for each $t_0 \in S^1$, $Q^s_{\underline{b}} \cap \left(\{t_0\} \times \R^2\right)$ is a disjoint union of disks $D^s_i$, $i=0, \ldots, p-1$, $p\in \N$,  of radius $\rho_s$ with 
\begin{equation}\label{eq:ds}
\rho_s \leq c_2e^{\lambda_{l+1}s}, 
\end{equation}
and if ${b} \neq b_{\min}$, 
\begin{equation}\label{eq:ds2}
c_3 e^{\lambda_{\hat{l}+1}s} \leq \mathrm{dist}(D^s_i, D^s_j), \quad i\neq j, 
\end{equation}
\item all components of $\mathcal{L}_s$ that are induced by some $\widetilde{v}\in \mathcal{V}(\underline{b})$ are contained in $Q^s_{\underline{b}}$,
    \item   all  components of $\mathcal{L}_s$ that are induced by some $\widetilde{v}'\in  \mathcal{V} \setminus \mathcal{V}(\underline{b})$ do not intersect $\hat{Q}^s_{\underline{b}}$, and the width of $\hat{Q}_{\underline{b}}^s$ is at least $c_4e^{\lambda_{\hat{l}+1} s}$ if $b\neq b_{\min}$  and at least $1$ otherwise.  
\end{enumerate}
\end{claim}
\begin{claimproof}
Let $p$ be the period of $z_b^s$, i.e.\ the minimal positive  number such that $\frac{p}{K} \star b = b$.  
If $p=1$, then the image of  $z_b^s(t)$ intersects each horizontal surface  $\{t_0\}\times \R^2$, $t_0\in S^1$, exactly once, and so \eqref{eq:ds} always holds.  If $p>1$, let   $q\in \{1, \ldots, p-1\}$.  Then $\frac{q}{K} \star e_i \neq e_i$, for  $\hat{l}+1\leq i\leq l$. Otherwise, the level $l_0$ of $\xi(b,\frac{q}{K} \star b) \in \mathcal{B}$ would satisfy $\hat{l} < l_0 < l$.
It follows that there is $c_3>0$ and $R'_2\geq R$ such that for all $s\geq R'_2$, $q\in \{1, \ldots, p-1\}$,   
\begin{align}\label{eq:distz}2c_3 e^{\lambda_{\hat{l}+1}s} \leq \|z^s_b(t) - z^s_b(t + \frac{q}{K})\|.
\end{align}
Since $l>\hat{l}$, we can choose $R_2\geq R'_2$, a constant $c_2>0$, and  tubular neighbourhoods $Q^s_{\underline{b}}$ of $\Im(z^s_{\underline{b}})$, $s\geq R_2$,  such that (i) holds.

If $\widetilde{v}  \in \mathcal{V}(\underline{b})$, then for some $q\in \N$, $ b \prec \frac{q}{K} \star \widetilde{v}$. Hence, by \eqref{eq:c_1}, there is $R_1\geq 0$, $c_1>0$ such that for all $s\geq R_1$ the component of $\mathcal{L}_s$ induced by  $\widetilde{v}$ lies in a tubular  neighbourhood of the image of $z^s_b$ of width at most  $c_1e^{\lambda_{l+1}s}$.
Since $l>\hat{l}$, we can assure, by choosing $R_2$ larger if necessary, that also  (ii) holds. 

If $b=b_{\min}$, there is nothing to show for (iii). Otherwise, if  $\widetilde{v}'\in \mathcal{V} \setminus \mathcal{V}(\underline{b})$,  then, by definition  $\xi(b',\xi(\widetilde{v}))\prec b$ for all $b'\in \mathcal{B}$ with $[b'] = \underline{b}$. The maximal level of such vertices  $\xi(b', \xi(\widetilde{v}))$ is at most $\hat{l}$. 
By  \eqref{eq:c_0} and \eqref{eq:distz},    there is  $c_4>0$  and $R_0\geq R$ such that for all $s\geq R_0$, the component of $\mathcal{L}_s$ induced by $\widetilde{v}$ lies outside a tubular neighbourhood $\hat{Q}^s_{\underline{b}}$  of $\Im (z^s_b)$ of width at least  $c_4e^{\lambda_{\hat{l}+1}s}$.  Hence, by choosing $R_2$ to be larger if necessary, we can assure that also (iii) holds. 
\end{claimproof}

Consider a partition $\mathcal{U}(\underline{b})$, $\underline{b} \in \underline{\mathcal{B}}$, of the set  $\mathcal{U} = \{\widetilde{u}_1, \ldots, \widetilde{u}_M\}$ as follows: We define $\mathcal{U}(\underline{b}) = \bigcup_{b\in \mathcal{B}, [b] = \underline{b}} \mathcal{U}(b)$, and let $\mathcal{U}({b})$ be the set of all $\widetilde{u} \in \mathcal{U}$ such that the following two conditions hold.
\begin{enumerate}[(A)]
\item either $\xi(\xi(\widetilde{u}), b) \prec b$,  or, for all $b^+\in \mathcal{B}$ with $b\preceq b^+$, $\xi(\xi(\widetilde{u}), b^+) = b$. 
\item if $b\neq b_{\min}$, then $\hat{b}\prec \xi(\xi(\widetilde{u}),b)$, where $\hat{b}$ is the predecessor of $b$.  
\end{enumerate}
It is easy to check that the sets $\mathcal{U}(\underline{b})$ form a partition of $\mathcal{U}$. (Note that in the example in Figure~\ref{fig:tree}, with $\underline{b}_1 := [(e^0_{2,1},e^0_{2,2},e^0_{2,3})]$, $\underline{b}_2 := [(e^0_{3,1},e^0_{3,2},e^0_{3,3})]$, $\underline{b}_3 := [(e^0_{2,1})]$,  we have that $\mathcal{U}(\underline{b}_1) = \{\widetilde{u}_1, \widetilde{u}_2\}$, $\mathcal{U}(\underline{b}_2) = \{\widetilde{u}_3\}$,  $\mathcal{U}(\underline{b}_3) = \{\widetilde{u}_4\}$.)

Fix again $\underline{b}\in \underline{\mathcal{B}}$ and denote by $l$ the level of $\underline{b}$.  
\begin{claim}\label{cl:vclass}
There is $R_3 \geq R_2$ such that for any $\widetilde{u}\in \mathcal{U}(\underline{b})$ and for all $s\geq R_3$, the loop $t \mapsto p_2 \circ \widetilde{\Phi} \circ \widetilde{u} \circ \psi(s,t)$ is homotopic in $\hat{Q}^s_{\underline{b}} \setminus (\mathcal{L}_s \cap \hat{Q}^s_{\underline{b}})$  to a loop in $\partial Q^s_{\underline{b}}$, where $\psi$ is the function in formula \eqref{eq:asympt_form} that is associated to  $\widetilde{u}$.  
\end{claim}
\begin{claimproof}
Let $\widetilde{u} \in \mathcal{U}(\underline{b})$. Note that since  $\mathcal{U}(\underline{b})$ is finite, it is sufficient to show that there is $R_3\geq R_2$ such that for all $s\geq R_3$, the loop $t\mapsto p_2 \circ \widetilde{\Phi} \circ \widetilde{u} \circ \psi(s,t)$ is homotopic in $\hat{Q}^s_{\underline{b}} \setminus (\mathcal{L}_s \cap \hat{Q}^s_{\underline{b}})$  to a loop in $\partial Q_{\underline{b}}^s$.
According to the alternatives in condition 
 (A) we distinguish two cases. 

\textit{Case 1:   There is $b\in \mathcal{B}$ with $[b] = \underline{b}$ such that $\xi(\xi(\widetilde{u}), b) \prec b$}.   

Fix a choice of such $b\in \mathcal{B}$, define $z^s_b$ as in \eqref{def:zsb},  and let $l_0<l$  be the level of $\xi(\xi(\widetilde{u}), b)$ in $\mathcal{T}$. 
It follows from  \eqref{eq:c_1} that there is $c_1$ and $R_1\geq R$ such that for all $s\geq R_1$, 
\begin{align}\label{cl2:eq1} 
\| (p_2 \circ \widetilde{\Phi} \circ \widetilde{u} \circ {\psi})(s,t)- z^s_b(t)\| \leq c_1e^{\lambda_{l_0+1}s}.
\end{align}
On the other hand, by \eqref{eq:c_0},  since $(l_0+1)$st entry of the vectors $b$ and $\xi(\hat{u})$ differ, there is $c_0>0$ and $R_0\geq R$ such that for all $s \geq R_0$, 
\begin{align}\label{cl2:eq2}
c_0e^{\lambda_{l_0+1}s} \leq \|(p_2 \circ \widetilde{\Phi} \circ \widetilde{u} \circ {\psi})(s,t)- z^s_b(t)\|.
\end{align}
By (B) and our assumption, $\hat{l} < l_0< l$, and hence by Claim \ref{cl:Y} (i), (iii), there is $R_3 \geq \max\{R_0,R_1,R_2\}$ such that for all $s\geq R_3$, 
the loop $t\mapsto p_2 \circ \widetilde{\Phi} \circ \widetilde{u} \circ \psi(s,t)$ is contained in $\hat{Q}_{\underline{b}}^s \setminus Q^s_{\underline{b}}$.
Since   $\hat{Q}^s_{\underline{b}}\setminus Q^s_{\underline{b}}$ retracts to $\partial Q^s_{\underline{b}}$, this finishes the proof in this case. 

\textit{Case 2:  There is $b\in \mathcal{B}$ with $[b] = \underline{b}$ such that for all $b^+\in \mathcal{B}$ with $b\preceq b^+$,  $\xi(\xi(\widetilde{u}), b^+) = b$.}

Fix $b\in \mathcal{B}$ with this property. In particular it holds that  $\xi(\xi(\widetilde{u}), b) = b$, and we can assume that the tubular neighbourhoods $Q^s_{\underline{b}}$, $s\geq R_2$,  from Claim \ref{cl:Y} contain also the images of the loops $t\mapsto p_2 \circ \widetilde{\Phi} \circ \widetilde{u} \circ {\psi}(s,t)=:  (Kt, \sum_{j=1}^L e^{\lambda_{j}s}\hat{e}_{j}(t))$. Note that with this notation,   $z^s_b(t) =(Kt, \sum_{j=1}^{l} e^{\lambda_{j}s}\hat{e}_{j}(t)).$ 

If $b$ is of the form $b = \xi(\widetilde{v})$ for some $\widetilde{v} \in \mathcal{V}$, then for $s \geq R_2$ the set $Q^s_{\underline{b}}$ contains only one component of $\mathcal{L}_s$, namely the image of $z^s_b$, and the assertion of the claim holds.
Hence, assume in the following that $b$ is not of this form. Then,  $\hat{e}_{l+1}$  
 is distinct to the $(l+1)$st entry of $\xi(\widetilde{v})$ for all $\widetilde{v} \in \mathcal{V}(b)$. This follows from the  assumptions with   $b^+ = \xi(\widetilde{v})$.  

 If $l\leq L-2$, there is a constant $c_5>0$ such that for all $s\geq R$,   $1\leq i\leq N$, $0\leq q \leq k_i-1$,  
\begin{align}\label{eq:7}\|\sum_{j=l+2}^L\hat{e}^q_{i,j}(t)\|\leq c_5e^{\lambda_{l+2}s}.
\end{align}

We distinguish the following  two sub-cases.  
\begin{enumerate}[(a)]
\item The eigenspace to the eigenvalue $\lambda_{l+1}$ has dimension $2$. 
\item The eigenspace to the eigenvalue $\lambda_{l+1}$ has dimension $1$.
\end{enumerate}

Assume first that 
 (a) holds. 
Since the eigenspace of $\lambda_{l+1}$ is two-dimensional and the eigenfunctions solve the linear ODE  $(\hat{A}_K - \lambda_{l+1} \id) = 0$, the solution of this ODE with any initial condition $y_0$ is an eigenfunction $\hat{e}_{l+1}^{y_0}$ of $\hat{A}_K$.

Let $c_2>0$ be as in Claim \ref{cl:Y}. We choose a path $\alpha:[0,1] \to \R^2$ with 
$\alpha(0) = \hat{e}_{l+1}(0)$ and 
\begin{align}\label{eq:c_4} 2c_2\leq \sup_{t\in S^1}\|\hat{e}^{\alpha(1)}_{l+1}(t)\|.  \end{align}
We may choose $\alpha$ such that for a constant  $c_6>0$, for all $\tau\in [0,1]$, 
\begin{equation}\label{eq:alpha}
c_6 \leq \dist(\alpha(\tau),\hat{e}^q_{i,l+1}(0)), 
\end{equation}
where $i\in \{1, \ldots, N\}$ and $q\in \{0, \ldots, k_i-1\}$ are any numbers with $w^q_i \in \mathcal{V}(b)$.

By \eqref{eq:7} and \eqref{eq:alpha},  there is $R_3 \geq R_2$ such that for all $s \geq R_3$, 
 the function $F^s:[0,1]\times S^1 \to S^1 \times \R^2$,  defined by 
$$F^s(\tau,t) = \left(Kt, \sum_{j=1}^{l} e^{\lambda_js} \hat{e}_{j}(t) + e^{\lambda_{l+1}s}\hat{e}^{\alpha(\tau)}_{l+1}(t) + \sum_{j=l+2}^L e^{\lambda_{j}s}\hat{e}_{j}(t)\right),$$ provides a homotopy in $\left(S^1 \times \R^2 \right) \setminus (\mathcal{L}_s\cap Q_{\underline{b}}^s)$ from the loop $t\mapsto p_2 \circ \widetilde{\Phi} \circ \widetilde{u} \circ \psi(s,t)$ to some loop $\gamma^s:[0,1] \to S^1 \times \R^2$.
By Claim \ref{cl:Y}(iii), we can enlarge $R_3$ if necessary such that for all $s\geq R_3$ the image of $F^s$ lies in $\hat{Q}^s_{\underline{b}}$ and the image of $\gamma^s$ lies in $\hat{Q}^s_{\underline{b}}\setminus Q^s_{\underline{b}}$. 
Hence,  $F^s$ is in fact a homotopy in $\left(S^1 \times \R^2 \right) \setminus \mathcal{L}_s$. This finishes the  proof for sub-case (a).

{To treat sub-case (b), we let $\mathfrak{e}_{\lambda_{l+1}}$ be a generator of the eigenspace of $\lambda_{l+1}$. We then define for $s\geq R_2$ a function   $F^s:[0,1] \times S^1 \to S^1 \times \R^2$ by 
$$F^s(\tau,t) = \left(Kt, \sum_{j=1}^{l} {e}^{\lambda_js} \hat{e}_{j}(t) + e^{\lambda_{l+1}s}\big(\hat{e}_{l+1}(t) + \tau \mu  J_0\mathfrak{e}_{\lambda_{l+1}}(t)\big) + \sum_{j=l+2}^L e^{\lambda_j s}\hat{e}_{i}(t)\right),$$
where $J_0 = 
\big(\begin{smallmatrix} 0 & -1 \\ 1 & 0 \end{smallmatrix}\big)$,   and where we chose $\mu>0$ suitably below. }

By the assumptions, the loop 
$\hat{e}_{l+1}(t)$ is distinct to $e^q_{i,l+1}(t)$ for all 
$i\in \{1, \ldots, N\}$, $q \in \{0, \ldots, k_i-1\}$ with $w_i^q \in \mathcal{V}(\underline{b})$, and in our case,  these curves are equal up to scalar multiplication. 
Hence, by  \eqref{eq:7}, we can choose $R_3$ large if necessary such that for a suitable $\mu>0$ and all $s\geq R_3$, $F_s$  provides again a homotopy in $\left(S^1 \times \R^2 \right) \setminus (\mathcal{L}_s \cap Q_{\underline{b}}^s)$ from the loop $t\mapsto p_2 \circ \widetilde{\Phi} \circ \widetilde{u} \circ \psi(s,t)$ to some loop $\gamma^s:[0,1] \to S^1 \times \R^2$ with the same properties as in sub-case (a). Then we can proceed as in case (a) to obtain a homotopy in $\hat{Q}^s_{\underline{b}}\setminus (\mathcal{L}_s \cap \hat{Q}^s_{\underline{b}})$ to a loop in $\partial Q^s_{\underline{b}}$. 
\end{claimproof}

If $\underline{b}\in \underline{\mathcal{B}}$ is of the form
$\underline{b} = [b]$ with $b = \xi(w^q_i)$, for some $i\in \{1, \ldots, N\}$, and some (and hence any) $q \in \{0, \ldots, k_i-1\}$, then the image $\Im(z^s_b)$ of $z^s_b$ coincides with the image of $t \mapsto p_2 \circ \widetilde{\Phi} \circ w^q_i \circ \psi^q_i(s,t)$ and  for $s\geq R_2$, $t_0 \in S^1$,  each disk in   $Q^s_{\underline{b}}\cap \{t_0\} \times \R^2$ does intersect $\mathcal{L}_s$ in only one point. This means that a loop in $\partial Q^s_{\underline{b}}$ is homotopic in the complement of $\mathcal{L}_s\setminus \hat{V}^i_s$ to a multiple  of $t \mapsto p_2 \circ \widetilde{\Phi} \circ w^q_i \circ \psi^q_i(s,t)$. Hence, by assumption (*),   $\mathcal{U}(\underline{b})$ is empty.
We therefore assume in the following that $\underline{b}$ is not of the above form. We are left with at most $N'-1$ many options for $b$.

We have an  embedding $\iota_{s,b}: S^1 \times \D \to S^1 \times \R^2$ onto $Q^s_{\underline{b}}$, given by $\iota_{s,b} (t,\zeta) = z^s_b(\frac{p}{K}t) + (0,\rho_s\zeta)$, where $p$ is the period of $z^s_b$, and $\rho_s$ is chosen as in Claim \ref{cl:Y}. 
Hence the map $\iota_{s,b}|_{S^1 \times \partial \D}$ induces a homomorphism  $$f_{s,{b}}: \pi_1(S^1 \times \partial \D, y)\cong \Z \times \Z \to \pi_1\left((S^1 \times \R^2)\setminus \mathcal{L}_s, \iota_{s,b}(y)\right).$$  Here, $y=(y_1,y_2)= (0,(1,0))$,  and,  in the identification  $\pi_1(S^1 \times \partial \D, y) \cong \pi_1(S^1, y_1) \times \pi_1(\partial \D, y_2) \cong \Z \times \Z$, we map the generator of the first coordinate to $(1,0)$, and the generator that is obtained by negative parametrization of $\partial \D$ to $(0,1)$. 
Note that $f_{s,b}$ 
 factors through an isomorphism $\pi_1(S^1 \times \partial \D) \to \pi_1(\partial Q^s_{\underline{b}})$ and also through an injective homomorphism $\pi_1(S^1 \times \partial \D) \to \pi_1(\hat{Q}_{\underline{b}}^s\setminus (\mathcal{L}_s \cap \hat{Q}_{\underline{b}}^s)).$

We have shown in 
 Claim \ref{cl:vclass} that for any $\widetilde{u} \in \mathcal{U}(\underline{b})$ and for $s\geq R_3$, 
 the loop $t\mapsto p_2\circ \widetilde{\Phi} \circ \widetilde{u} \circ \psi(s,t)= (Kt, \sum_{j=1}^L e^{\lambda_{j}s}\hat{e}_{j}(t))$ is homotopic  in $\hat{Q}^s_{\underline{b}} \setminus (\mathcal{L}_s \cap \hat{Q}^s_{\underline{b}}) $ to a representative loop  $\eta_s$ of a class $\alpha_s\in\Im f_{s,b}$. Identify $\alpha_s$ via   $f^{-1}_{s,b}$  
 with a tuple $(m_1, m_2) \in \Z \times \Z$. 
 Then $m_1$   is equal to the number of times the loop $t\mapsto p_2\circ \widetilde{\Phi} \circ \widetilde{u} \circ \psi(s,t)$ passes through $\{0\} \times \R^2$, that is, $$m_1=K,$$ and $m_2$ is the negative winding number of the projection of $t\mapsto \iota_{s,b}^{-1}(\eta_s(t))$ in $S^1 \times \partial \D$  to the second coordinate.  
Choose the vertex $b\in \mathcal{B}$ for which $\widetilde{u} \subset \mathcal{U}(b)$, and let $l_0$ be the level of $\xi(\xi(\widetilde{u}), b)$. It follows directly from the proof of Claim \ref{cl:vclass} that for $s$ sufficiently large,  the only term that contributes to $m_2$ is $\hat{e}_{l_0 + 1}$, i.e.\ $$m_2=  -\mathbf{w}^{\Phi}(\lambda_{l_0+1}, K).$$
By (w), we have 
$$ 0 < -\mathbf{w}^{\Phi}(\lambda^K_*, K) \leq -\mathbf{w}^{\Phi}(\lambda_{l_0+1}, K) \leq -\mathbf{w}^{\Phi}(\lambda_{l+1}, K),
$$
where $l\geq l_0$ is the level of $\underline{b}$. 
By our assumption on $\underline{b}$, there are two distinct  
$w^q_i$ and $w^{q'}_{i'}$ in $\mathcal{V}$  with $b \preceq \xi(\xi(w^q_i), \xi(w^{q'}_{i'}))$, and in fact  we can choose $w^q_i$ and $w^{q'}_{i'}$ such that $b =  \xi(\xi(w^q_i), \xi(w^{q'}_{i'}))$. 
Since 
$$\|\widetilde{\Phi} \circ w^q_i(s,t) \circ \psi^q_i(s,t) - \widetilde{\Phi} \circ w^{q'}_{i'} \circ \psi^{q'}_{i'}(s,t)\| = e^{\lambda_{l+1}s}\left( e_{l+1}(t) +  o(e^{(\lambda_{l+1} -\lambda_{l})s})\right),$$
where $e_{l+1}$ is an eigenvector for the eigenvalue $\lambda_{l+1}$,  
it holds that 
$$\mathbf{w}^{\Phi}(w^q_i,w^{q'}_{i'}) = \mathbf{w}^{\Phi}(\lambda_{l+1}, K).$$
Again, by rescaling, we 
 have that  $$\mathbf{w}^{\Phi}(w^q_i,w^{q'}_{i'}) = \frac{K}{\hat{K}}\mathbf{w}^{\Phi}(\overline{v}^q_i,\overline{v}^{q'}_{i'}).$$
Hence, if we set
$\mu = \min\mathbf{w}^{\Phi}(\overline{v}^q_i,\overline{v}^{q'}_{i'})<0$, where the minimum is taken over all $i,i', q,q'$, with $1\leq i,i'\leq N$, $1\leq q \leq k_{i}-1$, $1 \leq q'\leq k_{i'}-1$, we obtain that 
$$0< m_2 \leq - \mu \frac{K}{\hat{K}}\leq -\mu k.$$
We conclude, that for $s$ sufficiently large, there are at most $-\mu (N'-1)k$ homotopy classes of loops in $S^1 \times \R^2 \setminus \mathcal{L}_s$ of the form $t\mapsto p_2\circ \widetilde{\Phi} \circ \widetilde{u} \circ \psi(s,t)$, 
 $\widetilde{u} \in \mathcal{U}$, and where  
$C:=-\mu (N'-1)$ is independent of $k$. 
\end{proof}

\subsection{Conclusions}\label{sec:conclusions}
Let us conclude this section with the proof of Theorem \ref{thm:stability}, using the results obtained so far.

\begin{proof}[Proof of Theorem \ref{thm:stability}]
Let $(Y,\xi)$, $\alpha_0$, $\mathcal{L}_0$ as given in the assumptions, let  
$$\Gamma:= \limsup_{T \to +\infty } \frac{\log\left(\# \Omega^T_{\alpha_0}(\mathcal{L}_{0})\right)}{T} >0.$$
Let $\widetilde{J}_0$ be a regular almost complex structure on $\R \times Y$ as fixed in the beginning of Section \ref{sec:proofmainresult}.  Let $\epsilon>0$, and choose $\delta>0$ such that     
$$\delta < \min\{ \log(\Gamma) - \log(\Gamma-\epsilon),  \delta_{\alpha_{0},\mathcal{L}_0}\},$$
 where $\delta_{\alpha_{0},\mathcal{L}_0}$ is such that Assumption 1 in Section \ref{sec:compactness} holds. 
Let $\alpha$ be any non-degenerate contact form with $d_{C^0}(\alpha,\alpha_0) < \delta$. We let $C^+$ and $C^-$ satisfy the hypothesis of Proposition~\ref{prop:compactness}, and with these choices construct the exact homotopy of exact cobordisms and the moduli spaces as in Sections~\ref{sec:sympcobordisms} and~\ref{sec:moduli}. The results in Sections~\ref{sec:moduli}, \ref{sec:isotopy},  \ref{sec:buildings_limit}, and~\ref{sec:cobordism levels} allow us to work in the setting of Section~\ref{sec:non-collaps} as follows. 
Let $\mathrm{n}_0$ be the number of components of $\mathcal{L}_0$. By Lemma~\ref{lem:v_no_topbot_level}, the bottom level of the buildings $\mathbf{v}_i(1)$, $i=1, \ldots, \mathrm{n_0}$, lies in the cobordism $(\R \times Y, \lambda^-_1,\omega^-_1,J^{-\infty}_1)$ from $\alpha$ to $C^-\alpha$ and their connected components are: 
\begin{itemize}
    \item finite energy holomorphic  cylinders $\hat{v}_1, \ldots, \hat{v}_{\mathrm{n}_0}$
that are negatively asymptotic to the components $\gamma_0^1, \ldots, \gamma^{\mathrm{n}_0}_0$ of $\mathcal{L}_0$, and positively asymptotic to some 
 Reeb orbits of $\alpha$, 
 \item finite energy holomorphic planes $\hat{v}_{\mathrm{n}_0+1}, \ldots,  \hat{v}_{\hat{\mathrm{n}}_0}$ that are positively asymptotic to some Reeb orbits of $\alpha$. 
 \end{itemize} Consider the link $\mathcal{L}_{\epsilon}({\alpha})\subset Y$ formed by (the simple covers of) the Reeb orbits in $\alpha$ to which $\hat{v}_{1}, \ldots, \hat{v}_{\hat{\mathrm{n}}_0}$ are asymptotic to. We claim that $\mathcal{L}_{\epsilon}(\alpha)$ satisfies the statement of the theorem. 

For every $\rho \in \Omega_{\alpha_0}(\mathcal{L}_0)$, by considering the bottom level of the associated building $\mathbf{u}_{\rho}$, we obtain by Lemma  \ref{lem:u_firstlevels} a holomorphic cylinder $\hat{u}_{\rho}$ in the cobordism $(\R \times Y, \lambda^-_1,\omega^-_1,J^{-\infty}_1)$  that is negatively asymptotic to $\gamma_{\rho}$ and positively asymptotic to some Reeb orbit $\eta_{\rho}$ for the contact form $\alpha$.  
Since $\mathbf{u}_{\rho}$ is obtained in the limit as $R\to \infty$ of cylinders $\widetilde{u}_{\rho}^R(1)$  that are positively asymptotic to the Reeb orbit $\gamma_{\rho}$ (for the contact form 
 $C^{+}\alpha_0$), it holds that 
\begin{align}\label{eq:action_in_alpha}
\mathcal{A}_{\alpha}(\eta_{\rho}) \leq  \mathcal{A}_{C^+\alpha_0}(\gamma_{\rho}) =C^+\mathcal{A}_{\alpha_0}(\gamma_{\rho}).
\end{align}
We consider, for any   component $\gamma$ of $\mathcal{L}_{\epsilon}(\alpha)$, the subset $\mathcal{C}_{\alpha}(\gamma)$ of  classes  that "collaps onto $\gamma$", that is, the subset of elements  $\rho\in \Omega_{\alpha_0}(\mathcal{L}_0)$ 
for which $\eta_{\rho}$ is equal to a multiple of $\gamma$. Let  $\mathcal{NC}_{\alpha}$ be the set of the remaining, "non-collapsing"   classes, that is, the set of those $\rho\in \Omega_{\alpha_0}(\mathcal{L}_0)$ for which $\eta_{\rho}$ cannot be written as a multiple of  any component $\gamma$ of $\mathcal{L}_{\epsilon}(\alpha)$. One can show, by using Lemma \ref{lem:homot1}, that $\mathcal{C}_{\alpha}(\gamma) = \emptyset$ 
for all components $\gamma$ of $\mathcal{L}_{\epsilon}(\alpha)$ for which no multiple is a positive asymptotic orbit of the cylinders $\hat{v}_{1}, \ldots, \hat{v}_{{\mathrm{n}}_0}$. Moreover, one can show that all the Reeb orbits $\eta_{\rho}$ with   $\rho \in \mathcal{NC}_{\alpha}$ are 
  pairwise non-homotopic in $Y\setminus \mathcal{L}_{\epsilon}(\alpha)$. The arguments are very similar to those in Section \ref{sec:cobordism levels} and we leave them to the reader.

We claim that for all components $\gamma$ of $\mathcal{L}_{\epsilon}(\alpha)$,
\begin{align}\label{eq:growth_in_C}
\limsup_{T\to \infty} \frac{\log(\#\mathcal{C}^T(\gamma))}{T} = 0,
\end{align}
where $\mathcal{C}_{\alpha}^T(\gamma) := \mathcal{C}_{\alpha}(\gamma)\cap \Omega_{\alpha_0}^T(\mathcal{L}_0)$.  
To see this, let us fix a component $\gamma$ of $\mathcal{L}_{\epsilon}(\alpha)$. For each component $\hat{v}_i$, $i=1,\ldots, \hat{\mathrm{n}}_0$,  (either a cylinder or a plane) that is  positively asymptotic to a multiple of $\gamma$, we choose a holomorphic half-cylinder $\overline{v}_i:[0,+\infty)\times S^1  \to (0,+\infty)\times Y$ in $(\R \times Y, \lambda^-_1,\omega^-_1,J^{-\infty}_1)$ that is asymptotic to the same multiple and with image contained in $\hat{v}_i$. For convenience we reorder those half-cylinders, and enumerate them as $\overline{v}_1, \ldots, \overline{v}_N$, for some $N\leq \hat{\mathrm{n}}_0$. 
Similarly, for each $\rho \in \mathcal{C}_{\alpha}(\gamma)$, we choose a half-cylinder $\overline{u}_{\rho}:[0,\infty) \times S^1 \to (0,+\infty)\times Y$ with the same positive asymptotic as $\hat{u}_{\rho}$ and with image contained in that of $\hat{u}_{\rho}$.   Let $\Xi$ be the collection of the  half-cylinders $\overline{u}_{\rho}$, $\rho \in \mathcal{C}_{\alpha}(\gamma)$. We may assume that all half-cylinders in $\Xi$ and $\overline{v}_1, \ldots, \overline{v}_N$ are pairwise disjoint (we can argue that this holds by taking smaller half-cylinders, or directly deduce it from the construction). By Corollary \ref{cor:Lambdaa} (ii), condition (*) in Section \ref{sec:non-collaps} holds, hence we can apply Proposition \ref{prop:Xi_equi}. To this end, let $\Xi_{k}$ be the set of those $\overline{u} \in \Xi$ that are asymptotic to the $k$th iterate $\gamma_k$, as defined in Section \ref{sec:non-collaps}, and denote by $T_{0}$ the  period of $\gamma$. 
By   \eqref{eq:action_in_alpha}, $C^+\mathcal{A}_{\alpha_0}(\gamma_{\rho}) \geq \mathcal{A}_{\alpha}(\eta_{\rho}) = kT_0$, and hence   
\begin{align}\label{eq:C_Xi}
\#\mathcal{C}_{\alpha}^T(\gamma) \leq  \sum_{k=1}^{\lceil C_+T/T_0\rceil}\# \Xi_k.
\end{align} 
 By Corollary \ref{cor:Lambdaa} (i), the curves in $\Xi$ are pairwise non-equivalent, where the notion of equivalence is that defined in Section \ref{sec:non-collaps}. Hence by Proposition \ref{prop:Xi_equi} and \eqref{eq:C_Xi}, we have that $\mathcal{C}_{\alpha}^T(\gamma)$ grows at most quadratically in $T$. In particular, \eqref{eq:growth_in_C} holds. 

Let $\mathcal{NC}_{\alpha}^T := \mathcal{NC}_{\alpha} \cap \Omega_{\alpha_0}^T(\mathcal{L}_0)$. 
We obtain  
\begin{align}\label{eq:lowerbound_Lambda}
\begin{split}
\limsup_{T\to +\infty} \frac{\log\left(\# \Lambda^{ T}_{\alpha}(\mathcal{L}_{\epsilon}(\alpha))\right)}{ T} &\geq  \limsup_{T\to +\infty} \frac{\log\left(\# \mathcal{NC}_{\alpha}^{(C^+)^{-1}T}\right)}{T}  \\ &=   \limsup_{T\to +\infty} \frac{\log (\#\Omega_{\alpha_0}^{(C^+)^{-1}T}(\mathcal{L}_0))}{T} \\
&= \frac{1}{C^+}\Gamma\\ 
&\geq e^{-\delta} \Gamma > \Gamma-\epsilon,
\end{split}
\end{align}
where we used \eqref{eq:action_in_alpha} in the first line and \eqref{eq:growth_in_C} in the second line. 
The left hand side in \eqref{eq:lowerbound_Lambda} bounds $h_{\topo}(\phi_{\alpha})$ from below (see~\cite{AlvesPirnapasov}) and hence we have
$$h_{\topo}(\phi_{\alpha})> \Gamma-\epsilon.$$
Our choice of $\delta$ depends only on $\epsilon$ and the choice of $\delta_{\alpha_0, \mathcal{L}_0}$, and hence on  $\epsilon, \mathrm{sys}(\alpha_0)$, and $\mathfrak{h}(\alpha_0, \mathcal{L}_0, \widetilde{J}_0)$. 
Note that if we replace in the beginning of Section \ref{sec:proofmainresult} the regular almost complex structure  $\widetilde{J}_0$ with $\widetilde{J}_1$ for which $\mathfrak{h}(\alpha_0,\mathcal{L}_0,\widetilde{J}_1)\geq \mathfrak{h}(\alpha_0,\mathcal{L}_0,\widetilde{J}_0)$, we can keep  $\delta_{\alpha_0,\mathcal{L}_0}$ unchanged.  Hence we can choose $\delta$  depending only on  $\epsilon, \mathrm{sys}(\alpha_0)$, and $\mathfrak{h}(\alpha_0, \mathcal{L}_0)$. 
\end{proof}

\section{Geodesic flows on surfaces}\label{sec:geodesicflows}
An important subclass of Reeb flows on three dimensional manifolds is given by the geodesic flows on the unit cosphere bundles of two dimensional Riemannian surfaces. We consider the map 
    $$h_\topo: \mathrm{Met}(S) \to [0,+\infty)$$
on the space $\mathrm{Met}(S)$ of Riemannian metrics on $S$ that assigns to a Riemannian metric~$g$ the topological entropy $h_\topo(\phi_g)$ of the geodesic flow $\phi_g^t$ on the unit sphere bundle $(T_1)_g S$. 
The objective of this section is to prove Theorem \ref{thm:geod} and Corollary~\ref{cor:robust_T2}.

Theorem~\ref{thm:geod} is not an automatic corollary of the main theorem. Also it does not follow from the main theorem that there is a $C^{\infty}$-generic set of Riemannian metrics for which the assertion holds: the set of contact forms $\alpha_g$ originating from Riemannian metrics  is meager in the set of contact forms on $(T^1 S, \xi_{\rm geo})$. The theorem will follow once the genericity requirements of the main theorem have been independently supplied for geodesic flows. To establish this, we recall some preliminary notions.

{\begin{defn}
Let $g$ be a Riemannian metric on $S$. Each unparametrized closed geodesic $\upgamma$ of $g$ generates two closed Reeb orbits $\upgamma^+$ and $\upgamma^-$ of $\alpha_g$. The Reeb orbits $\upgamma^+$ and $\upgamma^-$ correspond to the two lifts of $\upgamma$  to $(T^1 S, \xi_{\rm geo})$, given by the two unit speed parametrizations of $\upgamma$ corresponding to the two possible orientations of $\upgamma$. We call the $2$-component link $\{\upgamma^+,\upgamma^-\}$ the \textit{transverse lift} of $\upgamma$.
\end{defn}
We will use the following lemma.
\begin{lem} \label{lem:transverseliftsimple}
Let $S$ be a closed surface.
If $\upgamma$ is a simple contractible closed geodesic on $\Sigma$ and $\{\upgamma^+,\upgamma^-\}$ is its transverse lift, then any iterate of $\upgamma^{+}$ is non-contractible in $T^{1} S \setminus \upgamma^{-}$. Similarly, any iterate of $\upgamma^{-}$ is non-contractible in $T^{1} S \setminus \upgamma^{+}$. 
\end{lem}
\begin{proof}
We prove the first assertion of the lemma, since the proof of the second assertion is identical. 

The conclusion of the lemma is automatically true if the genus of $S$ is $\geq 1$,  because in this case an iterate of $\upgamma^{+}$ is isotopic to an iterate of the unit fiber of $T^1 S$, and these are  non-trivial in $\pi_{1}(T^{1} S)$.

We now prove the lemma when $S=S^{2}$.
Since $\upgamma$ is a simple closed curve in $S^2$, it is the boundary of an embedded disk $\mathcal{D}$ in $S^2$. Let $p$ be a point in the interior of $\mathcal{D}$.
The knot $\upgamma^{-}$ is isotopic in the complement of $\upgamma^{+}$ to the unit fiber $T^1_p S^2$ over $p$. Therefore, it suffices to show that any cover of $\upgamma^{+}$ is non-contractible in $T^{1}S^{2} \setminus T^1_p S^2$.  

To show that $\upgamma^{+}$ is not contractible in $T^{1}S^{2}\setminus T^1_p S^2$, notice that $T^{1}S^{2}\setminus T^1_p S^2$  is the unit tangent bundle $T^{1}(S^{2}\setminus p )$ of $S^2 \setminus p$. Letting $p'$ be a point in $S^{2}\setminus p $, it is easy to see that  $\upgamma^{+}$ is isotopic inside $T^{1}(S^{2}\setminus p )$ to the unit tangent fiber  $T^1_{p'}(S^{2}\setminus p )$ of $p'$. The manifold $T^{1}(S^{2}\setminus p )$ is diffeomorphic to the solid torus, and the fiber $T^1_{p'}(S^{2}\setminus p )$ is a generator of $\pi_1(T^{1}(S^{2}\setminus p )) \simeq \mathbb{Z}$. It follows that any iterate of $T^1_{p'}(S^{2}\setminus p )$ is non-contractible in $T^1_{p'}(S^{2}\setminus p )$, and therefore any iterate of $\upgamma^+$ is non-contractible in $T^1(S^{2}\setminus p )$. As explained above, the first assertion of the lemma follows from this fact.
\end{proof}
}

\textit{Proof of Theorem~\ref{thm:geod}.} A Riemannian metric $g$ induces a contact form $\alpha_g$ on the unit cosphere bundle $(T^1 S, \xi_{\rm geo})$ such that the Reeb flow is the geodesic flow. We will show that if $g$ is non-degenerate, then it satisfies the conditions of Theorem~\ref{thm:stability}.

\emph{Non-degeneracy.} The Riemannian metric $g$ is $C^\infty$-generically non-degenerate (i.e., all non-constant closed geodesics are non-degenerate), which is equivalent to the fact that the contact form $\alpha_g$ is non-degenerate. 

\emph{Surface of section.} Furthermore, if $g$ is non-degenerate, then the authors of \cite{CKMS} constructed a special surface of section $\Sigma$ of $\phi_{\alpha_g}$ with the following properties:
\begin{itemize}
    \item[$\diamondsuit$] the Reeb vector field $X_{\alpha_g}$ is positively transverse to the interior of $\Sigma$,
    \item[$\heartsuit$] $\partial \Sigma$ is composed of the lifts of a finite collection of simple contractible closed geodesics of $g$, and the lifts of a finite collection of non-contractible closed geodesics of $g$,
    \item[$\spadesuit$] every trajectory of $\phi_{\alpha_g}$ intersects $\Sigma$.
\end{itemize}

We will use the boundary of the surface of section $\Sigma$ as part of the link. The projection of $\partial \Sigma$ to the surface $S$ is a finite collection $\mathcal{B}$ of unparametrized closed geodesics of $g$. As explained above, each geodesic of $\mathcal{B}$ can be lifted of a pair of closed Reeb orbits of $\alpha_g$. The link $\mathcal{L}$ is formed by the transverse lifts of all closed geodesics contained in $\mathcal{B}$:
\begin{equation}
    \mathcal{L}:= \bigcup_{\upgamma \in \mathcal{B}} \{\gamma^+,\gamma^-\}.
\end{equation}
Clearly, $\partial \Sigma $ is contained in $\mathcal{L}$, but the opposite need not be true. There might be $\upgamma \in \mathcal{B}$ such that only one component of the transverse lift of $\upgamma$ is contained in $\partial \Sigma$.

\emph{Hypertightness in the complement of $\mathcal{L}$.}
Let $\gamma$ be a periodic orbit that does not cover a boundary component of $\Sigma$. By $\spadesuit$, $\gamma$ intersects $\Sigma$ at least once and $\diamondsuit$ implies that all such intersections are positive, thus the topological intersection number of $\gamma$ and $\Sigma$ is positive.  This shows that $\gamma$ and $\partial \Sigma $ are linked. We conclude that $\gamma$ is not contractible in the complement of $\mathcal L$.

If $\gamma$ is contained in $\mathcal{L}$, then $\heartsuit$ gives us two possibilities: either $\gamma$ is a cover of a lift of a non-contractible closed geodesic of $g$, or $\gamma$ is a cover of a lift of a simple contractible closed geodesic of $g$. In the first case $\gamma$ does not bound a disk in $T^1S$  since the lift of a non-contractible curve in $S$ is non-contractible in $T^1 S$. In the second case we have that $\gamma$ is the cover of an  element of the transverse lift $\{\upgamma^+,\upgamma^-\}$ of a simple contractible closed geodesic $\upgamma$ of $g$. Up to renaming the elements of $\{\upgamma^+,\upgamma^-\}$, we can assume that $\gamma$ is a cover of $\upgamma^+$. By Lemma \ref{lem:transverseliftsimple} we know that $\gamma$ is not contractible in $T^1 S \setminus \upgamma^-$. Since $\upgamma^- \subset \mathcal{L}$, we conclude that any disk filling $\gamma$ must intersect $\mathcal{L}$.

It follows that any disk filling of a closed orbit of $\alpha_g$ must intersect $\mathcal{L}$, so that $\alpha_g$ is hypertight in the complement of $\mathcal L$.

\emph{Summary.} 
If $h_{\rm top}(\phi_{\alpha_g})=0$, then $g$ is automatically a point of lower semi-continuity, so we assume in the following that $h_{\rm top}(\phi_{\alpha_g})>0$.
If $g\in \mathrm{Met}_{\rm nd}(S)$, then the contact form $\alpha_g$ is non-degenerate and admits a link $\mathcal L$ such that $\alpha_g$ is hypertight in the complement of $\mathcal{L}$. We can apply Theorem \ref{thm:stability} to $\alpha_0$ and links $\mathcal{L}_0$ that we obtain as a union of $\mathcal{L}$ and suitable links guaranteed by Theorem~\ref{thm:CH_recover}.  
As a consequence, $\alpha_g$ is a point of $d_{C^0}$-lower semi-continuity of $h_{\rm top}$ in the space of contact forms and thus also a point of $d_{C^0}$-lower semi-continuity of the restriction of $h_{\rm top}$ to the space of contact forms induced by Riemannian metrics. Finally, using that the map associating $g$ to $\alpha_g$ is an isometric embedding\footnote{See the discussion in Section \ref{subsec:c0distRiemann}.} of $(\mathrm{Met}(S),\overline{d}_{C^0})$  in $ (\mathcal{R}(T^1 S,\xi_{\rm geo}),{d}_{C^0}) $, we conclude that $g$ is a point of $\overline{d}_{C^0}$-lower semi-continuity of $h_{\rm top}$ in $\mathrm{Met}(S)$, concluding the proof of Theorem \ref{thm:geod}. 
 \qed

We also include the proof of Corollary \ref{cor:robust_T2}. 
\begin{proof}[Proof of Corollary \ref{cor:robust_T2}]
It was previously shown that $h_{\topo}$ is robust at a metric $g_0$ in $T^2$ with $h_{\topo}(\phi_{g_0})>0$ if there exist two (or one)  minimal geodesics that bound an annulus in $T^2$ that is not foliated by heteroclininics between them, \cite[Theorem 34, \S 4.4]{ADMM}. Hence let us consider the (hypothetical) situation that $h_{\topo}(\phi_{g_0})>0$ and that we always, for any pair of neighbouring minimial geodesics, have such a foliation.  Let $0< \epsilon < h_{\topo}(\phi_{g_0})/2$. Let $\alpha_0=\alpha_{g_0}$ be the contact form on $(T^1 T^2, \xi_{\rm geo})$ induced by $g_0$.  By Theorem~\ref{thm:CH_recover}, there is a link $\mathcal{L}_0= \mathcal{L}^{\alpha_0}_{\epsilon}$ of hyperbolic periodic orbits of $\phi_{\alpha_0}$ such that
$$\limsup_{T\to +\infty}\frac{\log(\# \Omega_{\alpha_0}^T(\mathcal{L}_0))}{T} > h_{\topo}(\phi_{g_0})-\epsilon.$$ 
By our  assumptions on the metric, the Reeb flow $\phi_{\alpha'}$ of any contact form $\alpha'$ in a $C^{\infty}$-small neighbourhood of $\alpha$ does not have contractible periodic orbits. Hence we can argue as in the proof of Theorem \ref{thm:main} and apply Theorem \ref{thm:stability} to a sequence of suitably chosen contact forms $\alpha_k$ converging to $\alpha_0$ in the $C^{\infty}$-topology. As in the proof of Theorem \ref{thm:main},  we  obtain that there is $\delta>0$ such that for all $g$ with $\overline{d}_{C^0}(g, g_0)< \delta$ we have that  
$$h_{\topo}(\phi_{g})\geq h_{\topo}(\phi_{g_0})-2\epsilon.$$
\end{proof}

\bibliographystyle{plain}
\bibliography{ReferencesS5}
\end{document}